\documentclass[11pt]{amsart}

\usepackage[hmargin=3.5cm,top=3cm,bottom=3cm,includehead]{geometry}

\usepackage{amsmath, amssymb, amsthm, mathrsfs}
\usepackage[T1]{fontenc}
\usepackage{lmodern}
\usepackage{microtype}
\usepackage{graphicx,color,enumerate} 
\usepackage[pdfborder={0 0 0}]{hyperref}

\newcommand\E{{\mathbb E}}
\newcommand\N{{\mathbb N}}
\newcommand\R{{\mathbb R}}

\newcommand\Sp{{\mathbb S}}

\newcommand\PPP{{\mathbf P}}

\DeclareMathOperator{\supp}{supp}
\DeclareMathOperator{\Det}{det}
\DeclareMathOperator{\Div}{div}

\def\BB{{\mathcal B}}

\def\EE{{\mathcal E}}

\def\MM{{\mathcal M}}

\def\SS{{\mathcal S}}

\def\VV{{\mathcal V}}
\def\WW{{\mathcal W}}

\def\eps{{\varepsilon}}

\newcommand{\Ps}[2]{\left\langle#1,#2\right\rangle}
\newcommand{\norm}[1]{\lVert#1\rVert}
\newcommand{\Norm}[1]{\left\lVert#1\right\rVert}
\newcommand{\abs}[1]{\lvert#1\rvert}
\newcommand{\Abs}[1]{\left\lvert#1\right\rvert}

\newcommand{\wto}{\rightharpoonup}

\newcommand{\beqn}{\begin{equation}}
\newcommand{\eeqn}{\end{equation}}
\newcommand{\bal}{\begin{aligned}}
\newcommand{\eal}{\end{aligned}}

\newtheorem{thm}{Theorem}

\newtheorem{lemma}[thm]{Lemma}

\newtheorem{definition}[thm]{Definition}
\theoremstyle{definition}

\newtheorem{pb}{Problem}
\theoremstyle{remark}
\newtheorem{rem}[thm]{Remark}

\title[Kac's chaos on the Boltzmann's sphere]
{Quantitative and qualitative Kac's chaos on the Boltzmann's sphere}
\author{Kleber Carrapatoso}
\date{}

\address{CEREMADE - Universit\'e Paris Dauphine\\
Place du Mar\'echal De Lattre De Tassigny\\
75775 Paris cedex 16 - FRANCE
}
\email{carrapatoso@ceremade.dauphine.fr}

\begin{document}

\begin{abstract}
We investigate the construction of chaotic probability measures on the Boltzmann's sphere, which is the state space of the stochastic process of a many-particle system undergoing a dynamics preserving energy and momentum. 

Firstly, based on a version of the local Central Limit Theorem (or Berry-Esseen theorem), we construct a sequence of probabilities that is Kac chaotic and we prove a quantitative rate of convergence. Then, we investigate a stronger notion of chaos, namely entropic chaos introduced in \cite{CCLLV}, and we prove, with quantitative rate, that this same sequence is also entropically chaotic.

Furthermore, we investigate more general class of probability measures on the Boltzmann's sphere. Using the HWI inequality we prove that a Kac chaotic probability with bounded Fisher's information is entropically chaotic and we give a quantitative rate. We also link different notions of chaos, proving that Fisher's information chaos, introduced in \cite{HaurayMischler}, is stronger than entropic chaos, which is stronger than Kac's chaos. We give a possible answer to \cite[Open Problem 11]{CCLLV} in the Boltzmann's sphere's framework.

Finally, applying our previous results to the recent results on propagation of chaos for the Boltzmann equation \cite{MMchaos}, we prove a quantitative rate for the propagation of entropic chaos for the Boltzmann equation with Maxwellian molecules.
\end{abstract}

\maketitle


\bigskip


\tableofcontents

\newpage
\section{Introduction}\label{sec:intro}

\subsection{Motivation}
In his celebrated paper \cite{Kac1956}, Kac introduced the notion of propagation of chaos in order to connect a stochastic process of a system of $N$ identical particles undergoing binary collisions to its mean field equation.

Our interest in this paper is to investigate chaotic distributions supported by the phase space of the stochastic process of the $N$-particle system as we shall explain.
We refer to \cite{CCLLV} for a detailed introduction on this topic and on Kac's paper \cite{Kac1956}.

Consider a system of $N$ identical particles of mass $\rho>0$ such that its evolution is described by a jump process with binary collisions that preserves energy and momentum. Let us denote by $i,j$ the particles undergoing the collision, with pre-collisional velocities $v_i,v_j \in \R^d$ and post-collisional velocities $v_i^\ast,v_j^\ast \in \R^d$. We have then the conservation of momentum
$$
\rho v_i^\ast + \rho v_j^\ast = \rho v_i + \rho v_j,
$$
and the conservation of energy
$$
\frac{\rho}{2}|v_i^\ast|^2 + \frac{\rho}{2}|v_j^\ast|^2 = \frac{\rho}{2}|v_i|^2 + \frac{\rho}{2}|v_j|^2.
$$

If the system has initial energy $\EE = \frac{1}{2} \sum_{i=1}^N \rho |v_i|^2\in\R_+$ and initial momentum 
$M=\rho m = \sum_{i=1}^N \rho v_i \in \R^d$, then both energy and momentum will be unchanged under the dynamics. The phase space of this process is then the manifold $\SS^N(\sqrt{\EE},m) \subset \R^{dN}$ defined by
$$
\SS^N(\sqrt{\EE},m) := \left\{ V=(v_1,\dots,v_N)\in \R^{dN} \,\Big\vert\, \frac12\sum_{i=1}^N \rho|v_i|^2 = \EE,\; \sum_{i=1}^N \rho v_i = \rho m \right\},
$$
which is the intersection of a sphere of radius $\sqrt{2\EE/\rho}$ and a hyperplane. This space $\SS^N(\sqrt{\EE},m)$ is in fact a sphere in $\R^{dN}$ of dimension $d(N-1)-1$ with radius $\sqrt{2\EE/\rho - |m|^2/N}$ and center $(m,\dots,m)/\sqrt{N}$. We remark that we need $|m|^2 \le 2N\EE / \rho$ in order to $\SS^N(\sqrt{\EE},m)$ be non empty.

Now choosing units such that the mass $\rho$ of each particle is equal to 2, the total value of kinetic energy is $dN$ and, without loss of generality, choosing $m=0$, the state space of this dynamics is 
\begin{equation}\label{eq:S^N_B}
\SS^N_\BB :=\SS^N(\sqrt{dN},0)= \left\{ V=(v_1,\dots,v_N)\in \R^{dN} \,\Big\vert\, \sum_{i=1}^N |v_i|^2 = dN,\; \sum_{i=1}^N  v_i =  0 \right\}
\end{equation}
and we shall call the manifold $\SS^N_\BB$ the Boltzmann's sphere.


An example of this kind of dynamics is the space homogeneous Boltzmann model that we shall explain. Given a pre-collisional system of velocities $V=(v_1,\dots,v_N)\in \R^{dN}$ and a collision kernel (for more information on the collision kernel we refer to \cite{Villani-BoltzmannBook,MMchaos})
\begin{equation}\label{eq:noyau}
B(z,\cos \theta) = \Gamma(|z|) b(\cos\theta),
\end{equation}
for some nonnegative functions $\Gamma$ and $b$, the process is:

\begin{itemize}
\item for any $i' \ne j'$, pick a random time $T(\Gamma(|v_{i'} - v_{j'}|))$ of collision accordingly to an exponential law of parameter $\Gamma(|v_{i'} - v_{j'}|)$ and choose the minimum time $T_1$ and the colliding pair $(v_i,v_j)$ such that
$$
T_1 = T(\Gamma(|v_{i} - v_{j}|)) = \min_{i',j'} T(\Gamma(|v_{i'} - v_{j'}|)),
$$

\item draw $\sigma\in \Sp^{d-1}\subset \R^d$ according to the law $b(\cos\theta_{ij})$, with 
$$
\cos\theta_{ij} = \sigma \cdot \frac{(v_i-v_j)}{|v_{i} - v_{j}|},
$$

\item after collision the new velocities become
$$
V^\ast_{ij} = ( v_1,\dots, v_i^\ast,\dots,v_j^\ast,\dots,v_N)
$$
where the post-collisional velocities $v_i^\ast$ and $v_j^\ast$ are given by
\begin{equation}\label{eq:vitesses}
v_i^\ast = \frac{v_i+v_j}{2} + \frac{|v_i-v_j|}{2}\sigma, \qquad
v_j^\ast = \frac{v_i+v_j}{2} - \frac{|v_i-v_j|}{2}\sigma.
\end{equation}
\end{itemize}

Iterating this construction we built then the associated Markov process $(\VV_t)_{t\ge 0}$ on $\R^{dN}$.
The equation of the associated law is given by, after a rescaling of time, (see \cite{MMchaos})
\begin{equation}\label{eq:master}
\partial_t G^N_t = L_N G^N_t = \frac{1}{N} \sum_{i<j} \int_{\Sp^{d-1}} \left[G^N_t (V^\ast_{ij}) - G^N_t (V) \right] B(|v_i-v_j|,\cos\theta)\, d\sigma
\end{equation}
with initial data $G^N_0$ and where $V^\ast_{ij}=(v_1,\dots,v_i^\ast,\dots,v_j^\ast,\dots,v_N)$. This equation is known as the \emph{master equation}.

Associated to this process, we have the (limit) spatially homogeneous Boltzmann equation \cite{MMchaos,MMWchaos,Villani-BoltzmannBook}
\begin{equation}\label{eq:Boltzmann}
\partial_t f(t,v) = \int_{\R^d \times \Sp^{d-1}} B(|v-w|,\cos\theta) 
\Big( f(w^\ast)f(v^\ast) - f(w)f(v) \Big) \, dw\,d\sigma
\end{equation}
with initial data $f(0,\cdot) = f_0$ and
where the post-collisional velocities $v^\ast$ and $w^\ast$ are obtained by 
\eqref{eq:vitesses}.

We shall highlight here the models we consider in the last part of this work (see Theorem~\ref{thm:intro-PropChaos} below), and we refer to \cite{Villani-BoltzmannBook} for more details concerning the collision kernel.
Assuming a collision kernel $B$ derived from inverse-power law interaction potentials
$$
\phi(r) = r^{-(s-1)}, \quad s>2,
$$
we have that the collision kernel has the form
\beqn\label{eq:noyau2}
B(z, \cos\theta)=|z|^\gamma \, b(\cos\theta),\quad \gamma = \frac{s-(2d-1)}{s-1},
\eeqn
where the function $b$ is locally smooth and has a nonintegrable singularity
\beqn\label{eq:b}
\sin^{d-2}\theta\,b(\cos\theta) \sim_{\theta \sim 0} C_b\, \theta^{-1-\nu}, 
\quad \nu \in (0,2),\; C_b >0.
\eeqn
In the particular case of three dimensions $d=3$, we have $\gamma = (s-5)/(s-1)$ and $\nu = 2/(s-1)$.
If we replace the angular collision kernel $b$ by a locally integrable one, we speak of cutoff collision kernels (or Grad's cutoff).

We shall consider in this work the case of Maxwellian molecules, in which the collision kernel does not depend on the relative velocity, i.e.\ $\gamma=0$ in \eqref{eq:noyau2}. 
We consider the general assumption
\begin{equation}\label{eq:Bmaxwell}
\left\{
\bal
& B(|v-w|,\cos\theta)=b(\cos\theta) , \\
& \forall\; \alpha>0, \quad 
\int_{0}^\pi b(\cos\theta)\,(1-\cos\theta)^{\alpha + 1/4} \, \sin^{d-2}\theta \, d\theta < + \infty.
\eal
\right.
\end{equation}
This is the same assumption made in \cite{MMchaos}, since in Theorem~\ref{thm:intro-PropChaos} we use their results.
Remark that \eqref{eq:Bmaxwell} includes the true Maxwellian molecules (or Maxwellian molecules without cutoff) in dimension $d=3$, when $\gamma=0$, $\nu=1/2$ and
\begin{equation}\label{eq:Bmaxwell1}
B(z, \cos\theta)= b(\cos\theta) , \qquad b(\cos\theta) \sim_{\theta \sim 0} C_b\, \theta^{-5/2},\quad (d=3).
\end{equation}
Also, it includes the Grad's cutoff Maxwellian molecules, when the singularity is removed,  
\begin{equation}\label{eq:BmaxwellG}
B(z, \cos\theta) = b(\cos\theta), \qquad \int_{0}^\pi b(\cos \theta) \, \sin^{d-2}\theta \, d\theta < +\infty.
\end{equation}
Some results in Theorem~\ref{thm:intro-PropChaos} will consider the general assumption \eqref{eq:Bmaxwell} and others the cutoff Maxwellian molecules \eqref{eq:BmaxwellG}.

\bigskip
The program set by Kac in \cite{Kac1956} was to investigate the behavior of solutions of the mean field equation \eqref{eq:Boltzmann} in terms of the behaviour of the solutions of the master equation \eqref{eq:master}. Moreover, the notion of propagation of chaos introduced by Kac means that if the initial distribution $G^N_0$ is $f_0$-chaotic (Definition \ref{def:chaos} below) then, for all $t>0$, the solution $G^N_t$ of \eqref{eq:master} 
is $f_t$-chaotic, where $f_t$ is the solution of \eqref{eq:Boltzmann}. For more information on this topic we refer to the recent results of Mischler, Mouhot and Wennberg 
\cite{MMchaos,MMWchaos}.

This paper is inspired by the works of Carlen, Carvalho, Le Roux, Loss and Villani \cite{CCLLV} and also of Hauray and Mischler \cite{HaurayMischler}, which investigate chaotic probabilities on the usual sphere in $\R^N$ with radius $\sqrt{N}$ (also called Kac's sphere). This sphere is the phase space of Kac's model, which is a one-dimensional simplification, introduced in \cite{Kac1956}, of the model presented above, with energy conservation only. 

The novelty here is that we investigate chaotic probability sequences in the Boltzmann's sphere $\SS^N_\BB\subset \R^{dN}$ and, furthermore, we prove quantitative rates of chaos convergence. Moreover, we apply our results to the Boltzmann equation with true Maxwellian molecules to prove quantitative propagation of entropic chaos.

\subsection{Definitions and main results}\label{ssec:results}
Let $E$ be a Polish space, then we shall denote by $\PPP(E)$ the space of Borel probability measures on $E$. Furthermore, through this paper, on the space $E^N$ we will only consider symmetric measures, more precisely, we say that $G^N\in\PPP(E^N)$ is symmetric if for all $\varphi\in C_b(E^N)$ we have
$$
\int_{E^N} \varphi \, dG^N = \int_{E^N} \varphi_\sigma \, dG^N,
$$
for any permutation $\sigma$ of $\{1,\dots,N\}$, and where 
$$
\varphi_\sigma := \varphi(V_\sigma)=\varphi (v_{\sigma(1)},\dots,v_{\sigma(N)}),
$$
for $V=(v_1,\dots,v_N) \in E^N$. 


For $G^N\in\PPP(E^N)$ and a integer $\ell\in[1,N]$ we denote by $G^N_\ell$ (or $\Pi_\ell(G^N)$) the $\ell$-marginal of $G^N$, defined by 
$$
\forall \varphi\in C_b(E^\ell), \qquad 
\int_{E^\ell} \varphi \, dG^N_\ell = \int_{E^N} \varphi \otimes {\bf 1}^{\otimes (N-\ell)}  \, dG^N.
$$

We shall use through the paper the same notation to represent a probability measure and its density with respect to the Lebesgue measure.

We can now give the notion of chaos formalized by Kac in \cite{Kac1956}, we also refer to \cite{S6} for an introduction on this topic with a probabilistic approach and to \cite{MischlerEDPX} for a short survey.

\begin{definition}[Kac's chaos]\label{def:chaos}
Consider $f\in\PPP(E)$.
We say that $G^N \in \PPP(E^N)$ is $f$-chaotic (or $f$-Kac chaotic), if for each fixed positive integer 
$\ell$,  $G^N_\ell$ converges to $f^{\otimes \ell}$ in the sense of measures in $\PPP(E^\ell)$ when $N$ goes to infinity, i.e. if for all $\varphi\in C_b(E^\ell)$,
\begin{equation}\label{eq:def-chaos}
\lim_{N\to\infty} \int_{E^\ell} \varphi \, dG^N_\ell = \int_{E^\ell} \varphi \, df^{\otimes \ell}.
\end{equation}
\end{definition}

In fact, it is well known that we need condition \eqref{eq:def-chaos} to hold for only one $\ell\ge 2$ (see for instance \cite{S6}).

We also introduce the Monge-Kantorovich-Wasserstein (MKW) distance and for more information about it we refer to \cite{VillaniOTO&N}. 
Consider an integer $\ell$ and $p\in[1,\infty)$, we define then the space 
$$
\PPP_{p}(E^\ell):= \left\{ F^\ell \in \PPP(E^\ell) ;\; M_p(F^\ell) :=\int_{E^\ell} |X|^p \, dF^\ell(X) < \infty \right\}.
$$
Then, for $F^\ell,G^\ell \in \PPP_{p}(E^\ell)$ we define the MKW distance between $F^\ell$ and $G^\ell$ by
\begin{equation}\label{eq:MKW}
W_p(F^\ell,G^\ell) := \inf_{\pi\in\Pi(F^\ell,G^\ell)} \left( \int_{E^\ell\times E^\ell} d_{E^\ell}(X,Y)^p\, d\pi(X,Y) \right)^{1/p},
\end{equation}
where $\Pi(F^\ell,G^\ell)$ is the set of transfer plan between $F^\ell$ and $G^\ell$, which is the set of probabilty measures on $E^\ell\times E^\ell$ with marginals $F^\ell$ and $G^\ell$ respectively, 
and where we define the distace $d_{E^\ell}$ as 
$$
\forall\, X=(x_1,\dots,x_\ell),Y=(y_1,\dots,y_\ell) \in E^\ell,\qquad 
d_{E^\ell}(X,Y) := \sum_{i=1}^{\ell} d_E(x_i,y_i).
$$
In the paper we will use the Euclidean distance in $E=\R^d$, i.e. $d_E(x_i,y_i)=|x_i-y_i|$ for all $x_i,y_i \in E$. More precisely, we shall use
$$
\forall\, f,g\in\PPP_1(\R^d), 
\qquad W_1(f,g) = \inf_{\pi\in\Pi(f,g)}  \int_{\R^d\times\R^d } |x-y|\, d\pi(x,y)  
$$
and
$$
\forall\, f,g\in\PPP_2(\R^d), 
\qquad W_2(f,g) = \inf_{\pi\in\Pi(f,g)}  \left(\int_{\R^d\times\R^d } |x-y|^2\, d\pi(x,y)\right)^{1/2}.
$$
Moreover, for $F^N, G^N \in \PPP(\SS^N_\BB)$ we shall use in the definition of $W_p(F^N,G^N)$ the Euclidean distance inherited from $\R^{dN}$, which means that for $X,Y \in \SS^N_\BB$ we shall use $d_{\SS^N_\BB}(X,Y) = |X-Y|$.


\smallskip
Let $\gamma$ be the Gaussian probability measure on $\R^d$, $\gamma(v) = (2\pi)^{-d/2}\, e^{-|v|^2/2}$, and $\mu \in \PPP(\R^d)$. We define the relative entropy of $\mu$ with respect to $\gamma$ by
\begin{equation}\label{eq:def-ent-1}
\begin{aligned}
H(\mu | \gamma) := \int_{\R^d}  \log \frac{d\mu}{d\gamma} \, d\mu ,
\end{aligned}
\end{equation}
if $\mu$ is absolutely continuous with respect to $\gamma$, otherwise $H(\mu | \gamma)~:=~+ \infty$. 

Moreover, for $G^N \in \PPP(\SS^N_\BB)$ we define the relative entropy with respect to $\gamma^N$, the uniform probability measure on $\SS^N_\BB$, by
\begin{equation}\label{eq:def-ent}
\begin{aligned}
H(G^N | \gamma^N) 
&:= \int_{\SS^N_\BB} \left( \log \frac{dG^N}{d\gamma^N}   \right) \, dG^N.
\end{aligned}
\end{equation}


We shall now define a stronger notion of chaos, namely the entropic chaos introduced in 
\cite{CCLLV}.

\begin{definition}[Entropic chaos]\label{def:ent-chaos}
We say that the sequence $G^N \in \PPP(\SS^N_\BB)$ is entropically $f$-chaotic, for some $f\in\PPP(\R^d)$, if $G^N$ is $f$-chaotic in Kac's sense (Definition \ref{def:chaos}) and
\beqn\label{eq:condH}
\lim_{N\to\infty}  \frac{1}{N} \, H(G^N | \gamma^N) = H(f|\gamma)
\eeqn
with $H(f|\gamma)<\infty$.
\end{definition}



Finally, with these definitions at hand we can state the main results of the paper.

\begin{thm}\label{thm:intro-FN}
For any $f\in \PPP_{ 6}(\R^d)\cap L^p(\R^d)$ with $1< p\le \infty$, 
there exists a sequence of probability measures $F^N := [f^{\otimes N}]_{\SS^N_\BB} \in \PPP(\SS^N_\BB)$, contructed by conditioning the $N$-fold tensorization of $f$ to the Boltzmann's sphere, such that
\begin{enumerate}[(i)]

\item $F^N$ is $f$-chaotic. More precisely, for any $\ell\ge 1$ fixed there exists a constant $C=C(\ell)>0$ such that for $N\ge \ell + 1$ we have
$$
W_1 (F^N_\ell,f^{\otimes \ell}) \leq \frac{C}{\sqrt{N}};
$$

\item $F^N$ is entropically $f$-chaotic. More precisely, there exists a constant $C~>~0$ such that
$$
\left|\frac{1}{N} \,  H(F^N|\gamma^N) -H(f|\gamma) \right| \leq \frac{C}{\sqrt{N}}.
$$

\end{enumerate}

\end{thm}

Let us now define the relative Fisher's information of a probability measure $\mu\in\PPP(\R^d)$ with respect to $\gamma$ by
\begin{equation}\label{eq:def-fisher-1}
\begin{aligned}
I(\mu|\gamma) := \int_{\R^d} \left|\nabla \log \frac{d\mu}{d\gamma}\right|^2\, d\mu,
\end{aligned}
\end{equation}
and, as we did for entropy, we also define for $G^N\in\PPP(\SS^N_\BB)$ the relative Fisher's information with respect to $\gamma^N$ by
\begin{equation}\label{eq:def-Fisher}
\begin{aligned}
I(G^N | \gamma^N) 
&:=  \int_{\SS^N_\BB} \left|\nabla_{\SS} \log \frac{dG^N}{d\gamma^N}   \right|^2 \, dG^N,
\end{aligned}
\end{equation}
where $\nabla_{\SS}$ stands for the gradient on the Boltzmann's sphere, i.e. the  component of the usual gradient in $\R^{dN}$ that is tangent to the sphere $\SS^N_\BB$.

We define then another stronger notion of chaos, the Fisher's information chaos, in an analogous way of Definition~\ref{def:ent-chaos}.

\begin{definition}[Fisher's information chaos]\label{def:Fisher-chaos}
We say that the sequence $G^N \in \PPP(\SS^N_\BB)$ is Fisher's information $f$-chaotic, for some $f\in\PPP(\R^d)$, if $G^N$ is $f$-chaotic in Kac's sense (Definition \ref{def:chaos}) and
$$
\lim_{N\to\infty}\frac{1}{N} \, I(G^N | \gamma^N) = I(f|\gamma)
$$
with $I(f|\gamma)<\infty$.
\end{definition}

\begin{rem}
The Fisher's information chaos is introduced in \cite{HaurayMischler} in a weaker way, which is in fact equivalent to Definition~\ref{def:Fisher-chaos} thanks to Theorem~\ref{thm:intro-GN}.
\end{rem}

Next, we may compare as follows the several notions of chaos: 
\begin{thm}\label{thm:intro-GN}
Consider $G^N\in\PPP(\SS^N_\BB)$, with $k$-th order moment $M_k(G^N_1)$ bounded, for some $k\ge 6$, and suppose that $G^N_1 \wto f$ in $\PPP(\R^d)$. 

Then, each assertion listed below implies the further one:
\begin{enumerate}[(i)]

\item $N^{-1}I(G^N|\gamma^N) \to I(f|\gamma)$, with $I(f|\gamma)<\infty$.

\item $N^{-1}I(G^N | \gamma^N )$ is bounded and $G^N$ is $f$-chaotic in Kac's sense.

\item $N^{-1}H(G^N|\gamma^N) \to H(f|\gamma)$, with $H(f|\gamma)<\infty$. 

\item $G^N$ is $f$-chaotic in Kac's sense.

\end{enumerate}

\end{thm}
As a consequence, in Definition \ref{def:ent-chaos} of the entropic chaos and in Definition~\ref{def:Fisher-chaos} of Fisher's information chaos, we only need the convergence of the first marginal, i.e. $G^N_1 \wto f$, instead of the convergence of all marginals. Hence, this theorem asserts that Fisher's information chaos implies entropic chaos, which in turns implies chaos (or Kac's chaos). Furthermore, we prove a quantitative rate for the implication $(ii) \Rightarrow (iii)$.

\medskip
Another main result of the paper is a possible answer to \cite[Open Problem 11]{CCLLV} in the setting of Boltzmann's sphere given in Theorem~\ref{thm:intro-op}. First of all, let us state the problem.
For $G^N \in \PPP(\SS^N_\BB)$ and $f \in \PPP_6(\R^d) \cap L^p (\R^d)$ with $p>1$, consider the following two conditions:
\beqn\label{eq:cond1}
\lim_{N \to \infty} \frac1N \, H(G^N | [f^{\otimes N}]_{\SS^N_\BB}) = 0,
\eeqn
and
\beqn\label{eq:cond2}
\forall\, \ell \in \N, \qquad 
\lim_{N \to \infty}  H(G^N_\ell | f^{\otimes \ell}) = 0,
\eeqn
where $[f^{\otimes N}]_{\SS^N_\BB}$ is the probability measure constructed in Theorem~\ref{thm:intro-FN}. In the Kac's sphere setting (i.e.\ $\Sp^{N-1}(\sqrt N)$ instead of $\SS^N_\BB$),
\cite{CCLLV} proved that condition \eqref{eq:cond2} holds when $G^N$ is the conditioned tensor product $G^N = [f^{\otimes N}]_{\SS^N_\BB}$. As discussed in \cite{CCLLV}, conditions \eqref{eq:condH}, \eqref{eq:cond1} and \eqref{eq:cond2} really mean that $G^N$ is "strongly" close to $f^{\otimes N}$, not only in the weak measure sense for marginals as in Kac's chaos.
In view of this, they formulated the following problem.

\begin{pb}[{\cite[Open Problem 11]{CCLLV}}]\label{pb}
Does condition~\eqref{eq:cond1} imply condition \eqref{eq:cond2} ? More generally, does condition \eqref{eq:cond2} hold for a larger and easily recognized class of chaotic sequences, larger than those contructed by means of conditioning tensor products ?
\end{pb}

We give a partial answer to Problem~\ref{pb} in the following theorem.

\begin{thm}\label{thm:intro-op}
Consider $G^N \in \PPP(\SS^N_\BB)$ such that $G^N$ is $f$-chaotic, for some $f\in\PPP(\R^d)$, and suppose that
$$
M_k(G^N_1) \le C, \; k> 2,\qquad 
\frac{1}{N} \, I(G^N|\gamma^N) \le C.
$$
Suppose further that $f\in L^{\infty}(\R^d)$ and $f(v_1)\ge \exp(-a|v_1|^2)$ for some constant $a>0$. Then for any fixed $\ell$, there exists a constant $C = C(d,\ell,\norm{f}_{L^{\infty}},M_k(G^N_1,f)) >0$ such that for all $N\ge \ell + 1$ we have
$$
H(G^N_\ell | f^{\otimes \ell}) \le C\, W_1(G^N_\ell,f^{\otimes \ell})^{\theta(\ell,d,k)}
,
$$
where $\theta(\ell,d,k)$ is constructive and depends on $\ell$, $d$ and $k$. As a consequence, $H(G^N_\ell | f^{\otimes \ell}) \to 0$ as $N\to \infty$ and condition~\eqref{eq:cond2} holds.
\end{thm}

This theorem exhibes a class of chaotic sequences in the Boltzmann's sphere that satisfy condition \eqref{eq:cond2}. At a first sight, the hypotheses needed on $G^N$ and $f$ to \eqref{eq:cond2} be true may seen stronger than the conditioned tensor product, in which case \cite{CCLLV} proved that \eqref{eq:cond2} holds (as said above). However, as remarked in \cite{CCLLV,HaurayMischler}, the conditioned tensor product assumption is not propagated along time by the Boltzmann equation but the assumptions needed in Theorem~\ref{thm:intro-op} may be. It is indeed true for the Boltzmann equation with Maxwellian molecules (see point $(iv)$ of Theorem~\ref{thm:intro-PropChaos} below for a precise statement), hence, in this setting, the assumptions in Theorem~\ref{thm:intro-op} are natural, which gives a satisfying answer to the second question on Problem~\ref{pb} in the Maxwellian case.

The interest here is that, as already remarked in \cite{CCLLV,MMchaos,HaurayMischler}, a natural step on Kac's program would be to study the propagation of 
conditions \eqref{eq:condH} or \eqref{eq:cond1} or \eqref{eq:cond2} (which are stronger than Kac's chaos) under the master equation \eqref{eq:master}. As explained above, as a consequence of Theorem~\ref{thm:intro-op}, the propagation of \eqref{eq:cond2} holds true for Maxwellian molecules. We continue the investigation of these issues in Theorem~\ref{thm:intro-PropChaos} below, proving also the propagation of entropic chaos \eqref{eq:condH} and \eqref{eq:cond1}.

\bigskip
We can apply our previous results to the Boltzmann equation for Maxwellian molecules. Some of the results concern assumption \eqref{eq:Bmaxwell}, i.e.\ Maxwellian molecules with and without cutoff, others concern only the Grad's cutoff Maxwellian molecules \eqref{eq:BmaxwellG}. Thanks to the work on propagation of chaos of \cite{MMchaos}, we can establish the following theorem.

%

\begin{thm}\label{thm:intro-PropChaos}
Let $f_0\in \PPP(\R^d)$ and $G^N_0\in \PPP(\SS^N_\BB)$. Consider then, for all $t>0$, the solution $G^N_t$ of the Boltzmann master equation \eqref{eq:master} with Maxellian molecules (\eqref{eq:Bmaxwell} or \eqref{eq:BmaxwellG}) associated to the initial condition $G^N_0$, and the solution $f_t$ of the limiting Boltzmann equation \eqref{eq:Boltzmann} with Maxellian molecules (\eqref{eq:Bmaxwell} or \eqref{eq:BmaxwellG}) associated to the initial data $f_0$.

Then we have
\begin{enumerate}[(i)]

\item Let \eqref{eq:BmaxwellG} be in force. Consider $f_0\in \PPP_6\cap L^p(\R^d)$ for $p>1$. If $G^N_0$ is entropically $f_0$-chaotic, then for all $t>0$, $G^N_t$ is entropically $f_t$-chaotic, more precisely
$$
\lim_{N\to\infty} \frac1N \, H(G^N_t | \gamma^N) = H(f_t|\gamma).
$$

\item Let \eqref{eq:Bmaxwell} be in force. Consider $f_0\in \PPP_6(\R^d)$ with $I(f_0 | \gamma)<\infty$. If $G^N_0=[f^{\otimes N}_0]_{\SS^N_\BB} \in \PPP(\SS^N_\BB)$ as in Theorem~\ref{thm:intro-FN}, then, for all $t>0$, $G^N_t$ is entropically $f_t$-chaotic. More precisely, for any 
$$
\epsilon < \frac{48}{(7d+6)^2(5d+24)}
$$ 
there exists a constant $C:=C(\epsilon)>0$ such that
$$
\sup_{t\ge 0} \left| \frac{1}{N} \,  H(G^N_t | \gamma^N) - H(f_t|\gamma)\right| 
\le C N^{-\epsilon}.
$$

\item Let \eqref{eq:BmaxwellG} be in force. Consider $f_0\in  \PPP_6\cap L^{\infty}(\R^d)$ and $f_0(v_1) \ge \exp(-\alpha |v_1|^2 + \beta)$ for $\alpha>0$ and $\beta\in\R$. If $G^N_0$ satisfies condition \eqref{eq:cond1}
$$
\lim_{N\to\infty} \frac1N \, H(G^N_0 | [f^{\otimes N}_0]_{\SS^N_\BB}) = 0 ,
$$ 
then, for all $t>0$, $G^N_t$ also satisfies condition \eqref{eq:cond1}
$$
\lim_{N\to\infty} \frac1N \, H(G^N_t | [f^{\otimes N}_t]_{\SS^N_\BB}) = 0.
$$

\item Let \eqref{eq:BmaxwellG} be in force. Consider $f_0\in  \PPP_6\cap L^{\infty}(\R^d)$ and $f_0(v_1) \ge \exp(-\alpha |v_1|^2 + \beta)$ for $\alpha>0$, $\beta\in\R$.
Consider also $G^N_0$ that is $f_0$-chaotic and has $M_k (\Pi_1 (G^N_0))$ and $N^{-1} I(G^N_0 | \gamma^N)$ finite, for some $k>2$.

Then, for all $t\geq 0$, $G^N_t$ satisfies condition \eqref{eq:cond2}
$$
\forall\, \ell \in \N, \qquad 
\lim_{N \to \infty}  H(\Pi_\ell(G^N_t) | f_t^{\otimes \ell}) = 0.
$$

\end{enumerate}

\end{thm}

Theorem \ref{thm:intro-PropChaos} improves the results of \cite{MMchaos} where Kac's chaos is established with a rate but entropic chaos is proved without any rate. Indeed, point $(i)$ here is proved in \cite{MMchaos} and point $(ii)$ gives a quantitative propagation of entropic chaos. Moreover, point $(iii)$ answers a question of \cite[Remark 7.11]{MMchaos} and point $(iv)$ is a consequence of Theorem~\ref{thm:intro-op} as said above.

It is worth mentioning that point $(i)$ was proved in \cite{MMchaos} for both the Maxwellian molecules with cutoff \eqref{eq:BmaxwellG} and the hard spheres case (which corresponds to the collision kernel $B(z,\cos\theta) = |z|$). 
The proof of point $(iii)$ also shows that $(iii)$ is valid for hard spheres, indeed the proof is based on the fact that \eqref{eq:condH} and \eqref{eq:cond1} are equivalent under some hypotheses on $f$ (see Theorem~\ref{thm:comp-ent}) and these properties are also propagated along time in the hard spheres case 
(propagation of $L^\infty$, moments and lower Maxwellian bounds, see e.g.\ \cite{Villani-BoltzmannBook} and the references therein).
However, the results $(ii)$ and $(iv)$ are valid only for the Maxwellian case, the reason behind this is that a key ingredient of the proof is the propagation of the Fisher's information bound, and such property is only know to hold for Maxwellian molecules.

\subsection{Strategy}

We construct a probability on $\SS^N_\BB$ based on tensorization and conditioning of some probabilty measure on $\R^d$. To this purpose, we use an explicit formula for the marginals of the uniform probablity on $\SS^N_\BB$ and a version of the local Central Limit Theorem (also known as Berry-Esseen), which is the cornerstone of the proof.

In order to study more general probabilities on the Boltzmann's sphere, we use an interpolation-type inequality, relating entropy, Fisher's information and the $2$-MKW distance, called HWI inequality from \cite{OttoVillani,LottVillani,VillaniOTO&N}, to show that Kac chaotic probabilities with finite Fisher's information are entropically chaotic.

Finally, the application of our results to the Boltzmann equation is based on recent results of propagation of chaos from \cite{MMchaos} and on the relations of different notions of measuring chaos from the work \cite{HaurayMischler}.

\subsection{Previous works}
In \cite{Kac1956} it is proved that the $N$-fold tensorization of a smooth probability on $\R$ conditioned to the Kac's sphere, i.e. the usual sphere $\Sp^{N-1}(\sqrt N)$, is Kac chaotic. Then, the work \cite{CCLLV} extends this result to a more general class of probabilities on $\R$, introduces the notion of entropic chaos and also proves that the $N$-fold tensorization conditioned to the Kac's sphere is entropically chaotic. Furthermore, the recent work \cite{HaurayMischler} gives quantitative rates of the results before, introduces the notion of Fisher's information chaos and links these three notions of chaos.

\subsection{Organization of the paper}In Section~\ref{sec:unif} we shall study the uniform probability measure on $\SS^N_\BB$. In Section~\ref{sec:chaotic} we construct a chaotic distribution on Boltzmann's sphere based on a probability measure on $\R^d$. Furthermore we prove a quantitative chaos convergence rate and we prove point $(i)$ of Theorem~\ref{thm:intro-FN}. Then, in Section~\ref{sec:entropy} we investigate the entropic and Fisher's information chaos. First, we study the entropic chaos for the probability distribution built before in Section~\ref{sec:chaotic} and we prove point $(ii)$ of Theorem~\ref{thm:intro-FN}. Then, we link these three notions of chaos and investigate a more general class of probability measures on $\SS^N_\BB$, proving Theorem~\ref{thm:intro-GN} and Theorem~\ref{thm:intro-op}. Finally, in Section~\ref{sec:appli} we use our previous results to prove Theorem~\ref{thm:intro-PropChaos}.

\bigskip\noindent
{\bf Acknowledgements.} 
The author would like to thank S.\@ Mischler and C.\@ Mouhot for their constant encouragement, fruitful discussions and careful reading of this paper. 
The author also thanks M.\@ Hauray for discussions on the representation of Fisher's information on the Boltzmann's sphere and A.\@ Einav for discussions on integration over Boltzmann's spheres. Finally, the author thanks the referees for helpful suggestions.

\section{Uniform probability measure}\label{sec:unif}

Consider $V=(v_1,\dots,v_N)\in \R^{dN}$, $r\in\R_+$ and $z\in\R^d$. We define the sphere
\begin{equation*}\label{}
\SS^N(r,z):= \left\{ V=(v_1,\dots,v_N)\in \R^{dN} \,\Big\vert\, \sum_{i=1}^N v_i^2 = r^2,\; \sum_{i=1}^N v_i = z\right\}.
\end{equation*}

We denote by $\gamma^N_{r,z}$ the uniform probability measure on $\SS^N(r,z)$.
We recall that $\SS^N_\BB := \SS^N(\sqrt{dN},0)$ is the Boltzmann sphere and we denote by $\gamma^N := \gamma^N_{\sqrt{dN},0}$ its uniform probability measure. Moreover, we also denote by $\Sp^{n-1}(r) \subset \R^n$ the usual sphere of dimension $n-1$ and radius $r$, $\Sp^{n-1}:=\Sp^{n-1}(1)$ and by $\Abs{\Sp^{n-1}}$ its measure. We can easily compute the measure of $\SS^N(r,z)$ by
\begin{equation}\label{eq:mesureSN}
\left|\SS^N(r,z)\right| 
= \left| \Sp^{d(N-1)-1}\right| \left(r^2 - \frac{|z|^2}{N}\right)^{\frac{d(N-1)-1}{2}}_{+},
\end{equation}

For $V=(v_1,\dots,v_N)\in \R^{dN}$, we shall use through the paper the notation $V_\ell = (v_1,\dots,v_\ell)\in \R^{d\ell}$, $V_{\ell,N} = (v_{\ell+1},\dots,v_N)\in \R^{d(N-\ell)}$ and $\bar V_\ell = \sum_{i=1}^\ell v_i \in\R^{d}$.

We begin with the following result of a change of variables, proved in Appendix~\ref{ap:change}.

\begin{lemma}\label{lem:change} 
Consider $V \in \SS^N(r,z)$. We can make a change of coordinates $(v_1,\dots,v_N) \to (u_1,\dots,u_N)$ in the following way
\begin{equation}\label{eq:change1}
\begin{aligned}
u_N &= \frac{1}{\sqrt{N}} (v_1 + \cdots + v_N) \\
u_k &= \frac{1}{\sqrt{k(k+1)}} (v_1 + \cdots + v_k - k\, v_{k+1}), \quad  1\leq k\leq N-1,
\end{aligned}
\end{equation}
such that the Jacobian is equal to one, $|u_1|^2 + \cdots + |u_{N}|^2= |v_1|^2 + \cdots + |v_N|^2$ and
\begin{equation}\label{eq:change2}
\left\{
\begin{aligned}
&|v_1|^2 + \cdots + |v_N|^2 = r^2 \\
&v_{1,\alpha} + \cdots + v_{N,\alpha} = z_\alpha
\end{aligned}
\right.
\quad \to \quad
\left\{
\begin{aligned}
&|u_1|^2 + \cdots + |u_{N-1}|^2 = r^2 - \frac{|z|^2}{N} \\
&u_{N,\alpha} = \frac{z_\alpha}{\sqrt{N}}, \quad 1\leq \alpha\leq d.
\end{aligned}
\right. 
\end{equation}
\end{lemma}

With these definitions and notations at hand we can study some properties of the uniform probability measure $\gamma^N$ on $\SS^N_\BB$. We remark that these estimates can also be obtained using correlation operators on the Boltzmann's sphere as in Carlen, Carvalho and Loss \cite{CCL}.

\begin{lemma}\label{lem:gammaNl}
We have the following properties
\begin{enumerate}[(i)]

\item for any $\ell\le N-1$ the $\ell$-marginal of $\gamma^N$ is given by $\gamma^N_\ell (dV_\ell)=\gamma^N_\ell (V_\ell) \,dV_\ell$ with
\begin{equation}\label{eq:gammaNl}
\gamma^N_\ell (V_\ell)  = \frac{ \Abs{\Sp^{d(N-\ell-1)-1}} }{\Abs{\Sp^{d(N-1)-1}}} \,
\frac{N^{\frac{d}{2}}}{(N-\ell)^{\frac{d}{2}}} \, \frac{ \left( dN-|V_\ell|^2 - \frac{|\bar V_\ell|^2}{N-\ell}    \right)^{\frac{d(N-\ell-1)-2}{2}}_+ }{(dN)^{\frac{d(N-1)-2}{2}}},
\end{equation}
where $dV_\ell = dv_1\dots dv_\ell$ is the Lebesgue measure on $\R^{d\ell}$.\\

\item the moments of $\gamma^N_\ell$ are uniformly bounded in $N$, more precisely, for $k\ge 1$ we have $M_k(\gamma^N_\ell)~\leq~C_{d,k,\ell}$, where $C_{d,k,\ell}$ depends on $d,k$ and $\ell$.
\end{enumerate}

\end{lemma}

Before the proof, we refer to \cite{Einav} where a Fubini-like theorem on $\SS^N(r,z)$ is proved, which yields a generalization of \eqref{eq:gammaNl} for the $\ell$-marginal of 
$\gamma^N_{r,z}$.

\begin{proof}
Let us split the proof.

\subsubsection*{(i)} We can define $\gamma^N_{r,z}$ by
$$
\gamma^N_{r,z} := \frac{1}{Z^N_{r,z}} \lim_{h\to 0} \frac{1}{h}\left( {\bf 1}_{B^N_z(r+h)} - {\bf 1}_{B^N_z(r)}  \right), \quad B^N_z (r) := \{ V \in \R^{dN}; |V|\leq r ,\; \sum_{i=1}^N v_i = z    \},
$$
where $Z^N_{r,z}$ is the normalization constant so that the integral of $\gamma^N_{r,z}$ is one.

Consider $\varphi \in C(\R^{d\ell})$, for $\ell \leq N-1$, then 
$$
\begin{aligned}
&\Ps{{\bf 1}_{B^N_z(r)}}{\varphi \otimes {\bf 1}^{N-\ell}} \\
&\quad= \int_{\R^{dN}} {\bf 1}_{|V_\ell|^2  + |V_{\ell,N}|^2 \leq r^2} \,
{\bf 1}_{\bar V_\ell + v_{\ell+1} + \cdots + v_N =z}  \, \varphi(V_\ell) \, dV_\ell \,  dV_{\ell,N} \\
&\quad= \int_{\R^{d\ell}}  \varphi(V_\ell) \left( \int_{\R^{d(N-\ell)}} {\bf 1}_{|V_{\ell,N}|^2  \leq r^2 - |V_\ell|^2}
{\bf 1}_{v_{\ell+1} + \cdots + v_N =z - \bar V_\ell}   \, dV_{\ell,N} \right)\,
 dV_\ell \\
&\quad= \int_{\R^{d\ell}} \varphi(V_\ell) \left|\mathbb B^{d(N-\ell-1)}\right| \left( r^2 - |V_\ell|^2 - \frac{|z-\bar V_\ell|^2}{N-\ell} \right)^{\frac{d(N-\ell-1)}{2}}_+ \, dV_\ell,
\end{aligned}
$$
where $|\mathbb B^{d(N-\ell-1)}|$ is the measure of the unit ball in dimension $d(N-\ell-1)$.
We deduce then that the $\ell$-marginal of $\gamma^N_{r,z}$, denoted by $\Pi_\ell(\gamma^N_{r,z})$, is given by
$$
\begin{aligned}
\Pi_\ell(\gamma^N_{r,z} ) &= \frac{1}{Z^{N}_{r,z}} \frac{d}{dr}\left[ \left|\mathbb B^{d(N-\ell-1)}\right| \left( r^2 - |V_\ell|^2 - \frac{|z-\bar V_\ell|^2}{N-\ell} \right)^{\frac{d(N-\ell-1)}{2}}_+   \right] \\
&= \frac{\left|\mathbb B^{d(N-\ell-1)}\right|}{Z^{N}_{r,z}} d(N-\ell-1) \, r \left( r^2 - |V_\ell|^2 - \frac{|z-\bar V_\ell|^2}{N-\ell} \right)^{\frac{d(N-\ell-1)-2}{2}}_+ \\
&= \frac{\left|\Sp^{d(N-\ell-1)-1}\right|}{Z^{N}_{r,z}} r \left( r^2 - |V_\ell|^2 - \frac{|z-\bar V_\ell|^2}{N-\ell} \right)^{\frac{d(N-\ell-1)-2}{2}}_+
\end{aligned}
$$
and in the particular case $r^2=dN$, $z=0$
\begin{equation}\label{eq:gammaNl-0}
\begin{aligned}
\Pi_\ell(\gamma^N )= \gamma^N_\ell  
&= \frac{\left|\Sp^{d(N-\ell-1)-1}\right|}{Z^{N}_{\sqrt{dN},0}} (dN)^{1/2} \left( dN - |V_\ell|^2 - \frac{|\bar V_\ell|^2}{N-\ell} \right)^{\frac{d(N-\ell-1)-2}{2}}_+.
\end{aligned}
\end{equation}
Now we shall compute $Z^N :=Z^{N}_{\sqrt{dN},0}$, with
\begin{equation}\label{eq:Z}
\begin{aligned}
Z^N = \left|\Sp^{d(N-\ell-1)-1}\right| (dN)^{1/2} \int_{\R^{d\ell}}  \left( dN - |V_\ell|^2 - \frac{|\bar V_\ell|^2}{N-\ell} \right)^{\frac{d(N-\ell-1)-2}{2}}_+ \, dV_\ell.
\end{aligned}
\end{equation}
We start by the integral
$$
\begin{aligned}
A =  \int_{\R^{d\ell}}  \left( dN - |V_\ell|^2 - \frac{|\bar V_\ell|^2}{N-\ell} \right)^{\frac{d(N-\ell-1)-2}{2}}_+ \, dV_\ell,
\end{aligned}
$$
with the changement of variable \eqref{eq:change1}-\eqref{eq:change2} (replacing $N$ by $\ell$), with the notation $U = U_{\ell-1}= (u_1,\dots, u_{\ell-1})$ and $ x = u_\ell$ to simplify, we obtain 
$$
\begin{aligned}
A =  \int_{\R^{d\ell}}  \left( dN - |U|^2 - \frac{N}{N-\ell}|x|^2 \right)^{\frac{d(N-\ell-1)-2}{2}}_+ \, dU dx.
\end{aligned}
$$
Changing $U$ to spherical coordinates in dimension $d(\ell-1)$, we have
\begin{equation}\label{eq:A}
\begin{aligned}
A &=  \int_{\R^d}\int_0^{\infty} |\Sp^{d(\ell-1)-1}| \left( dN - \rho^2 - \frac{N}{N-\ell}|x|^2 \right)^{\frac{d(N-\ell-1)-2}{2}}_+ \, \rho^{d(\ell-1)-1} \,d\rho \, dx \\
&= |\Sp^{d(\ell-1)-1}| \int_0^{\infty} \left( \int_{\R^d}\left( dN - \rho^2 - \frac{N}{N-\ell}|x|^2 \right)^{\frac{d(N-\ell-1)-2}{2}}_+ \, dx  \right) \rho^{d(\ell-1)-1} \,d\rho.
\end{aligned}
\end{equation}
Looking first to the integral over $\R^d$ we obtain, changing $x$ to spherical coordinates in dimension $d$,
$$
\begin{aligned}
B &=  \int_{\R^d}\left( dN - \rho^2 - \frac{N}{N-\ell}|x|^2 \right)^{\frac{d(N-\ell-1)-2}{2}}_+ \, dx  \\
&= |\Sp^{d-1}| \int_{0}^{\infty} \left( dN - \rho^2 - \frac{N}{N-\ell}y^2 \right)^{\frac{d(N-\ell-1)-2}{2}}_+ \,y^{d-1} dy,
\end{aligned}
$$
and after some computations we get
$$
\begin{aligned}
B
&= \frac{|\Sp^{d-1}|}{2} \left( \frac{N-\ell}{N}\right)^{d/2} (dN-\rho^2)^{\frac{d(N-\ell)-2}{2}}_+
\int_0^1 (1-y)^{\frac{d(N-\ell-1)-2}{2}} y^{\frac{d-2}{2}}\, dy\\
&= \frac{|\Sp^{d-1}|}{2} \left( \frac{N-\ell}{N}\right)^{d/2} (dN-\rho^2)^{\frac{d(N-\ell)-2}{2}}_+\,
\frac{ \Gamma\left( \frac{d(N-\ell-1)-2}{2} + 1  \right) \Gamma\left( \frac{d-2}{2} + 1  \right) }{ \Gamma\left( \frac{d(N-\ell-1)-2}{2} + \frac{d-2}{2} + 2  \right)}.
\end{aligned}
$$
Plugging this expression in \eqref{eq:A} we get
\begin{equation*}\label{}
\begin{aligned}
A 
&= |\Sp^{d(\ell-1)-1}| \frac{|\Sp^{d-1}|}{2} \left( \frac{N-\ell}{N}\right)^{d/2} \,
\frac{ \Gamma\left( \frac{d(N-\ell-1)-2}{2} + 1  \right) \Gamma\left( \frac{d-2}{2} + 1  \right) }{ \Gamma\left( \frac{d(N-\ell-1)-2}{2} + \frac{d-2}{2} + 2  \right)}\\
&\quad \times \int_0^\infty (dN-\rho^2)^{\frac{d(N-\ell)-2}{2}}_+ \rho^{d(\ell-1)-1} \,d\rho,
\end{aligned}
\end{equation*}
and we can compute the last integral
$$
\begin{aligned}
C &:= \int_0^\infty (dN-\rho^2)^{\frac{d(N-\ell)-2}{2}}_+ \rho^{d(\ell-1)-1} \,d\rho\\
&= \frac12 (dN)^{\frac{d(N-1)-2}{2}}\, \frac{ \Gamma\left( \frac{d(N-\ell)-2}{2} + 1  \right) \Gamma\left( \frac{d(\ell-1)-2}{2} + 1  \right) }{ \Gamma\left( \frac{d(N-\ell)-2}{2} + \frac{d(\ell-1)-2}{2} + 2  \right)}.
\end{aligned}
$$
Finally, plugging this in \eqref{eq:Z}, we obtain
\begin{equation*}\label{}
\begin{aligned}
Z^N &= \left|\Sp^{d(N-\ell-1)-1}\right|\,  \left|\Sp^{d(\ell-1)-1}\right| \frac{\left|\Sp^{d-1}\right|}{2} \left( \frac{N-\ell}{N}\right)^{d/2} \, \frac12 (dN)^{\frac{d(N-1)-1}{2}} \\
&\quad \times
\frac{ \Gamma\left( \frac{d(N-\ell-1)}{2}  \right) \Gamma\left( \frac{d}{2}  \right) }{ \Gamma\left( \frac{d(N-\ell)}{2}   \right)}\,  \frac{ \Gamma\left( \frac{d(N-\ell)}{2}  \right) \Gamma\left( \frac{d(\ell-1)}{2}  \right) }{ \Gamma\left( \frac{d(N-1)}{2}    \right)}
\end{aligned}
\end{equation*}
and using the fact that 
\begin{equation}\label{eq:measureS}
|\Sp^{n-1}| = \frac{2\pi^{n/2}}{\Gamma\left( \frac{n}{2} \right)}
\end{equation}
we have
\begin{equation}\label{eq:Zfin}
\begin{aligned}
Z^N &= \left|\Sp^{d(N-1)-1}\right|\,  (dN)^{\frac{d(N-1)-1}{2}} \, \left( \frac{N-\ell}{N}\right)^{d/2} , 
\end{aligned}
\end{equation}
then we conclude by plugging \eqref{eq:Zfin} in \eqref{eq:gammaNl-0}.

\subsubsection*{(ii)}Let $k\ge 1$ be a even integer. We have then to compute $M_k(\gamma^N_\ell)$
\begin{equation}\label{eq=Mk}
\begin{aligned}
\int_{\R^{d\ell}} |V_\ell|^k\, \gamma^N_\ell(V_\ell)\, dV_\ell 
&=    \frac{ \Abs{\Sp^{d(N-\ell-1)-1}} }{\Abs{\Sp^{d(N-1)-1}}} \,
\frac{  \left(\frac{N}{N-\ell}\right)^{\frac{d}{2}} }{(dN)^{\frac{d(N-1)-2}{2}} }  \\
&\quad \times
\int_{\R^{d\ell}} |V_\ell|^k\, \left( dN-|V_\ell|^2 - \frac{|\bar V_\ell|^2}{N-\ell}    \right)^{\frac{d(N-\ell-1)-2}{2}}_+ \, dV_\ell.
\end{aligned}
\end{equation}

As in the proof of (i), we use the change of coordinates \eqref{eq:change1}-\eqref{eq:change2}, then to simplify we denote $U=U_{\ell-1}=(u_1,\dots,u_{\ell-1})$ and $x=u_\ell$. Hence we can compute the integral
\begin{equation*}\label{}
\begin{aligned}
A_k &=\int_{\R^{d\ell}} |V_\ell|^k\, \left( dN-|V_\ell|^2 - \frac{|\bar V_\ell|^2}{N-\ell}    \right)^{\frac{d(N-\ell-1)-2}{2}}_+ \, dV_\ell\\
&= \int_{\R^{d\ell}} \left( |U|^2 + |x|^2 \right)^{\frac{k}{2}}\, \left( dN-|U|^2 - \frac{N}{N-\ell} |x|^2   \right)^{\frac{d(N-\ell-1)-2}{2}}_+  dU\,dx.
\end{aligned}
\end{equation*}
With another change of coordinates, $U$ to spherical coordinates in dimension $d(\ell-1)$, $x$ also to spherical coordinates in dimension $d$ we have
\begin{equation*}\label{}
\begin{aligned}
A_k 
&= \left|\Sp^{d(\ell-1)-1}\right| \left|\Sp^{d-1}\right| \int_0^\infty \!\!\! \int_0^\infty \left( \rho^2 + y^2 \right)^{\frac{k}{2}}\, \left( dN-\rho^2 - \frac{N}{N-\ell} y^2   \right)^{\frac{d(N-\ell-1)-2}{2}}_+ \rho^{d(\ell-1)-1} y^{d-1} d\rho\,dy \\
&\leq C \left|\Sp^{d(\ell-1)-1}\right| \left|\Sp^{d-1}\right| \int_0^\infty \rho^k \left\{ \int_0^\infty \left( dN-\rho^2 - \frac{N}{N-\ell} y^2   \right)^{\frac{d(N-\ell-1)-2}{2}}_+ y^{d-1}\, dy\right\}\rho^{d(\ell-1)-1}\, d\rho\\
&+C \left|\Sp^{d(\ell-1)-1}\right| \left|\Sp^{d-1}\right| \int_0^\infty \left\{ \int_0^\infty y^k \left( dN-\rho^2 - \frac{N}{N-\ell} y^2   \right)^{\frac{d(N-\ell-1)-2}{2}}_+ y^{d-1}\, dy\right\}\rho^{d(\ell-1)-1}\, d\rho\\
&=: I_1 + I_2. 
\end{aligned}
\end{equation*}
For the first term we have (already computed in (i))
\begin{equation*}\label{}
\begin{aligned}
I_1 
&= \frac12 \left|\Sp^{d(\ell-1)-1}\right| \left|\Sp^{d-1}\right| \left(  \frac{N-\ell}{N} \right)^{\frac{d}{2}}
\frac{\Gamma\left( \frac{d(N-\ell-1)}{2}\right)\Gamma\left( \frac{d}{2}\right)}
{\Gamma\left( \frac{d(N-\ell)}{2}\right)} \\
&\quad \times
\int_0^\infty (dN-\rho^2)^{\frac{d(N-\ell)-2}{2}}\rho^{d(\ell-1)-1+k}\,d\rho\\
&=\frac12 \left|\Sp^{d(\ell-1)-1}\right|  \left|\Sp^{d-1}\right| \left(  \frac{N-\ell}{N} \right)^{\frac{d}{2}}
\frac{\Gamma\left( \frac{d(N-\ell-1)}{2}\right)\Gamma\left( \frac{d}{2}\right)}
{\Gamma\left( \frac{d(N-\ell)}{2}\right)}\\
&\quad \times 
\frac12 (dN)^{\frac{d(N-1)-2+k}{2}}
\frac{\Gamma\left( \frac{d(N-\ell)}{2}\right)\Gamma\left( \frac{d(\ell-1)+k}{2}\right)}
{\Gamma\left( \frac{d(N-1)+k}{2}\right)}.
\end{aligned}
\end{equation*}
In the same way, we can compute the second term to get
\begin{equation*}\label{}
\begin{aligned}
I_2
&= \frac12 \left|\Sp^{d(\ell-1)-1}\right|   \left|\Sp^{d-1}\right| \left(  \frac{N-\ell}{N} \right)^{\frac{d+k}{2}}
\frac{\Gamma\left( \frac{d(N-\ell-1)}{2}\right)\Gamma\left( \frac{d+k}{2}\right)}
{\Gamma\left( \frac{d(N-\ell)+k}{2}\right)} \\
&\quad \times
\int_0^\infty (dN-\rho^2)^{\frac{d(N-\ell)-2+k}{2}}\rho^{d(\ell-1)-1}\,d\rho\\
&=\frac12  \left|\Sp^{d(\ell-1)-1}\right|  \left|\Sp^{d-1}\right| \left(  \frac{N-\ell}{N} \right)^{\frac{d+k}{2}}
\frac{\Gamma\left( \frac{d(N-\ell-1)}{2}\right)\Gamma\left( \frac{d+k}{2}\right)}
{\Gamma\left( \frac{d(N-\ell)+k}{2}\right)} \\
&\quad \times 
\frac12 (dN)^{\frac{d(N-1)-2+k}{2}}
\frac{\Gamma\left( \frac{d(N-\ell)+k}{2}\right)\Gamma\left( \frac{d(\ell-1)}{2}\right)}
{\Gamma\left( \frac{d(N-1)+k}{2}\right)}.
\end{aligned}
\end{equation*}
Plugging this two estimates in \eqref{eq=Mk} we obtain after some simplifications
\begin{equation*}\label{}
\begin{aligned}
M_k(\gamma^N_\ell) &\leq \frac{ \Abs{\Sp^{d(N-\ell-1)-1}} }{\Abs{\Sp^{d(N-1)-1}}} \,
\frac{  \left(\frac{N}{N-\ell}\right)^{\frac{d}{2}} }{(dN)^{\frac{d(N-1)-2}{2}} }\,(I_1+I_2)\\
&\le (dN)^{\frac{k}{2}} 
\frac{\Gamma\left( \frac{d(N-1)}{2}\right)}{\Gamma\left( \frac{d(N-1)+k}{2}\right)}\,
\frac{\Gamma\left( \frac{d(\ell-1)+k}{2}\right)}{\Gamma\left( \frac{d(\ell-1)}{2}\right)}
+
(dN)^{\frac{k}{2}} 
\frac{\Gamma\left( \frac{d(N-1)}{2}\right)}{\Gamma\left( \frac{d(N-1)+k}{2}\right)}\,
\frac{\Gamma\left( \frac{d+k}{2}\right)}{\Gamma\left( \frac{d}{2}\right)}.
\end{aligned}
\end{equation*}
Using the fact that for $k$ even we have 
$$
\begin{aligned}
\Gamma\left( \frac{n}{2} + \frac{k}{2} \right) &= \frac{(n+k-2)}{2}\, \frac{(n+k-4)}{2} \cdots 
\frac{n}{2} \,\Gamma\left( \frac{n}{2}\right) \\&= \frac{1}{2^{\frac{k}{2}}}
\underbrace{(n+k-2)(n+k-4)\cdots n}_{k/2 \text{ terms}}\, \Gamma\left( \frac{n}{2}\right),
\end{aligned}
$$
we conclude that
\begin{equation}\label{}
\begin{aligned}
M_k(\gamma^N_\ell) 
&\le \frac{ (dN)^{\frac{k}{2}} }{[d(N-1)+k-2][d(N-1)+k-4]\cdots [d(N-1)]} \\
&\quad\times
\Big(   [d(\ell-1)+k-2][d(\ell-1)+k-4]\cdots [d(\ell-1)] \\ 
&\qquad\quad+ (d+k-2)(d+k-4)\cdots d     \Big)\\
&\leq \frac{(dN)^{\frac{k}{2}}}{[d(N-1)]^{\frac{k}{2}}}
\Big(   [d(\ell-1)+k-2][d(\ell-1)+k-4]\cdots [d(\ell-1)] \\
&\qquad\qquad\qquad\qquad+ (d+k-2)(d+k-4)\cdots d     \Big)\\
&\leq 2^{\frac{k}{2}}\Big(   [d(\ell-1)+k-2][d(\ell-1)+k-4]\cdots [d(\ell-1)] \\
&\qquad\qquad\qquad\qquad+ (d+k-2)(d+k-4)\cdots d     \Big)\\
&\leq C_{d,k,\ell},
\end{aligned}
\end{equation}
where $C_{d,k,\ell}$ depends only on $d$, $k$ and $\ell$. 

We proved then a uniform bound in $N$ for $k$ even. If $k$ is odd we use $|v|^k \leq |v|^{k-1}+|v|^{k+1}$ with the last estimate to conclude.

\end{proof}

Now, using this explicit formula for $\gamma^N_\ell$ computed above, we prove that $\gamma^N$ 
is $\gamma$-chaotic, where $\gamma$ is the Gaussian probability measure in $\R^d$, i.e. $\gamma(v) = (2\pi)^{-d/2}\,e^{-|v|^2/2}$, for $v\in\R^d$. The proof presented here is an adaptation of \cite{DiaconisFreedman1987}, where it is proved that the uniform probability measure on the sphere $\Sp^{n-1}(\sqrt{n}) \subset \R^n$ is $\gamma_1$-chaotic, with $\gamma_1(x) = (2\pi)^{-1/2}\, e^{-x^2/2}$ the one-dimensional Gaussian measure.

\begin{lemma}\label{lem:gammaN-chaos}
The sequence of probability measures $\gamma^N \in \PPP(\SS^N_\BB)$ is $\gamma$-chaotic, more precisely, for any integer $\ell$ such that $d\ell \le d(N - 2)-3$ we have
$$
{\norm{\gamma^N_\ell - \gamma^{\otimes \ell}}}_{L^1} \leq 2\,\frac{d(\ell+2)+2}{dN-d(\ell+2)-2}.
$$
\end{lemma}

\begin{proof}
Let $\ell$  be an even integer. Then we have
$$
\begin{aligned}
\frac{ \Abs{\Sp^{d(N-\ell-1)-1}} }{\Abs{\Sp^{d(N-1)-1}}} &= \frac{1}{\pi^{\frac{d\ell}{2}}}\,
\frac{\Gamma\left( \frac{d(N-1)}{2} \right)}{\Gamma\left( \frac{d(N-\ell-1)}{2} \right)}\\
&= \frac{(dN)^{\frac{d\ell}{2}}}{(2\pi)^{\frac{d\ell}{2}}} \, \left( 1 -\frac{d+2}{dN}  \right) \left( 1 -\frac{d+4}{dN}   \right)\cdots \left( 1 -\frac{d(\ell+1)}{dN}   \right).
\end{aligned}
$$
By the explicit formula of $\gamma^N_\ell$ in Lemma~\ref{lem:gammaNl} we obtain
$$
\gamma^N_\ell = \frac{\left(\frac{N}{N-\ell}\right)^{\frac{d}{2}}}{(2\pi)^{\frac{d\ell}{2}}} \, \left( 1 -\frac{d+2}{dN}  \right)\cdots \left( 1 -\frac{d(\ell+1)}{dN}   \right)\,
\left( 1-\frac{|V_\ell|^2}{dN} - \frac{|\bar V_\ell|^2}{dN(N-\ell)}    \right)^{\frac{d(N-\ell-1)-2}{2}}_+.
$$
Since $\gamma^N_\ell$ and $\gamma^{\otimes \ell}$ are probability densities, the $L^1$ norm of their difference can be computed in the following way
\begin{equation}\label{eq:L1}
{\norm{\gamma^N_\ell - \gamma^{\otimes \ell}}}_{L^1} = 2 \int_{\R^{d\ell}} \left( \frac{\gamma^N_\ell}{\gamma^{\otimes \ell}} - 1\right)_+ \, \gamma^{\otimes \ell} \,dV_\ell,
\end{equation}
and we shall denote 
$$
\frac{\gamma^N_\ell}{\gamma^{\otimes \ell}} =\left(\frac{N}{N-\ell}\right)^{\frac{d}{2}}\, h(V_\ell)\, A
$$ 
with 
$$
\begin{aligned}
h(V_\ell) &:=  e^{\frac{|V_\ell|^2}{2}}\,
\left( 1-\frac{|V_\ell|^2}{dN} - \frac{|\bar V_\ell|^2}{dN(N-\ell)}    \right)^{\frac{d(N-\ell-1)-2}{2}}_+ 
\end{aligned}
$$
and
$$
\begin{aligned}
A &:= \left( 1 -\frac{d+2}{dN}  \right)\cdots \left( 1 -\frac{d(\ell+1)}{dN}   \right).
\end{aligned}
$$

We obtain that
$$
\begin{aligned}
\log h(V_\ell) &= \frac{|V_\ell|^2}{2} + \frac{d(N-\ell-1)-2}{2}\log \left( 1-\frac{|V_\ell|^2}{dN} - \frac{|\bar V_\ell|^2}{dN(N-\ell)}    \right)\\
&\leq \frac{|V_\ell|^2}{2} + \frac{d(N-\ell-1)-2}{2}\log \left( 1-\frac{|V_\ell|^2}{dN}     \right),
\end{aligned}
$$
and since the function $\alpha(z) = z/2 + [(d(N-\ell-1)-2)/2]\log ( 1-z/dN )$ has a maximum for $z=d(\ell+1) + 2$, we deduce
\begin{equation}\label{eq:logh}
\log h(V_\ell) \leq \frac{d(\ell+1)+2}{2} + \frac{d(N-\ell-1)-2}{2}\log \left( 1-\frac{d(\ell+1)+2}{dN}     \right),
\end{equation}
for $d\ell \le d(N-1)-3$.

On the other hand, for the quantity $A$, we have
\begin{equation}\label{eq:logA}
\begin{aligned}
\log\left[ \left( 1 -\frac{d(\ell+1)+2}{dN}  \right)A  \right] &= \sum_{j=1}^{(d(\ell+1)+2)/2} \log \left( 1 -\frac{2j}{dN}  \right) \\
&\le \int_0^{(d(\ell+1)+2)/2} \log \left( 1 -\frac{2x}{dN}  \right)\, dx\\
&= - \frac{d(N-\ell-1)-2}{2}\log\left( 1 -\frac{d(\ell+1)+2}{dN}  \right) - \frac{d(\ell+1)+2}{2},
\end{aligned}
\end{equation}
again for $d\ell \le d(N-1)-3$.

Combining \eqref{eq:logh} and \eqref{eq:logA} we obtain
$$
\log\left[ h(V_\ell)\left( 1 -\frac{d(\ell+1)+2}{dN}  \right)A  \right]\le 0
$$
and then
$$
 \left( 1 -\frac{d(\ell+1)+2}{dN}  \right) \frac{\gamma^N_\ell}{\gamma^{\otimes \ell}} \leq \frac{(N-\ell)^{\frac{d}{2}}}{N^{\frac{d}{2}}},
$$
which implies
$$
\begin{aligned}
\frac{\gamma^N_\ell}{\gamma^{\otimes \ell}}-1 
&\le \frac{d(\ell+1)+2}{dN-d(\ell+1)-2}.
\end{aligned}
$$
Plugging this expression in \eqref{eq:L1} we deduce
$$
{\norm{\gamma^N_\ell - \gamma^{\otimes \ell}}}_{L^1} \leq \frac{2d(\ell+1)+4}{dN-d(\ell+1)-2},
$$
which is valid if $\ell$ is even. 

Finally, if $\ell$ is odd, then $\ell+1$ is even and we shall write 
$$
{\norm{\gamma^N_\ell - \gamma^{\otimes \ell}}}_{L^1}\leq {\norm{\gamma^N_{\ell+1} - \gamma^{\otimes \ell+1}}}_{L^1} \le 2\,\frac{d(\ell+2)+2}{dN-d(\ell+2)-2}
$$ 
for $d\ell \le d(N - 2)-3$, which concludes the proof.

\end{proof}


\section{Chaotic sequences in Kac's sense}\label{sec:chaotic}

In this section, inpired by the work \cite{CCLLV}, we shall construct a chaotic sequence of probability measures on the Boltzmann's sphere based on the tensorization of some suitable probability $f$ on $\R^d$ and conditioning to $\SS^N_\BB$. We shall give a quantitative rate of the chaos convergence, proving a precise version of point $(i)$ in Theorem~\ref{thm:intro-FN}.

First of all, we define 
\begin{equation}\label{eq:def-Z'N}
Z_N (f; r,z) = \int_{\SS^N(r,z)} f^{\otimes N} \, d\gamma^N_{r,z},\quad \text{and}\quad
Z'_N (f; r,z) = \int_{\SS^N(r,z)} \frac{f^{\otimes N}}{\gamma^{\otimes N}} \, d\gamma^N_{r,z},
\end{equation}
for $r\in\R_+$ and $z\in\R^d$, and we shall investigate their asymptotic behaviour.
We remark that, since $\gamma^{\otimes N}$ is constant on $\SS^N(r,z)$, we have
$$
Z'_N (f; r,z) = \frac{Z_N (f; r,z)}{\gamma^{\otimes N}}
$$
and we shall study in the sequel only the behaviour of $Z'_N(f; r,z)$.


Define the space $\PPP_{k}(\R^d):=\{f\in\PPP(\R^d);\; M_k(f):=\int |v|^k f\,dv <\infty\}$, for some $k\ge1$. Let us  
consider $f\in\PPP_{ 6}(\R^d)\cap L^p(\R^d)$, for some $p>1$, a probability measure that verifies
\begin{equation}\label{eq:hyp-f}
\begin{aligned}
&\int_{\R^d}vf(v)\,dv =0 \,,\qquad\qquad \qquad
\int_{\R^d}v\otimes v\, f(v)\,dv =\EE I_d ,
  \\
&\int_{\R^d}\abs{v}^2\,f(v)\,dv =d\EE=E \,,\qquad
\int_{\R^d}(\abs{v}^2-E)^2 f(v)\,dv =\Sigma^2 ,
\end{aligned}
\end{equation}
where $I_d$ is the $d$-dimensional identity matrix.

\subsection{Preliminary results}
Before study the asymptotic behaviour of $Z'_N$, we shall state some preliminary results that will be useful in the sequel.


Consider $(\VV_j)_{j\in\N^\ast}$ a sequence of random variables i.i.d.\ in $\R^d$ with 
same law $f$, then the law of the couple $(\VV_1,\VV_1^2)$ is
\begin{equation}\label{eq:h}
h(v,u) =   f(v) \, \delta_{u = |v|^2} \in \PPP(\R^d \times \R_+).
\end{equation}

Moreover, we have the following lemma.
\begin{lemma}
The random variable $S_N := \sum_{j=1}^N (\VV_j,|\VV_j|^2)$ has law $s^N(z,u)\,dzdu$ with
\begin{equation*}
s^N(z,u) := \frac{\left| \SS^N(\sqrt{u},z)\right|}{2\left(u-\frac{|z|}{N^2} \right)^{1/2} N^{d/2}} \, Z_N(f;\sqrt{u},z),
\end{equation*}
where $z\in\R^d$ and $u\in\R_+$.
\end{lemma}

\begin{proof}
Let $\varphi\in C_b(\R^d \times \R_+)$, with the change of coordinates \eqref{eq:change1}-\eqref{eq:change2} $v\to u$, we have
$$
\begin{aligned}
\E\left[ \varphi\left( \sum_{j=1}^N \VV_j,\sum_{j=1}^N |\VV_j|^2 \right)\right]
&= \int_{\R^{dN}} \varphi\left( \sum_{j=1}^N v_j,\sum_{j=1}^N |v_j|^2 \right)\, f^{\otimes N}\, dV\\
&=\int_{\R^{dN}} \varphi\left( \sqrt{N}u_N,\sum_{j=1}^N |u_j|^2 \right)\, f^{\otimes N}\, dU.
\end{aligned}
$$
Denoting $r^2=\sum_{j=1}^{N-1} |u_j|^2$ and splitting the integral, the last equation is equal to
$$
\int_{0}^{\infty}\!\!\!\int_{\R^d} \varphi(\sqrt{N}u_N,r^2 + |u_N|^2) \left\{ \left|\Sp^{d(N-1)-1}(r)\right| \int_{\Sp^{d(N-1)-1}(r)} f^{\otimes N} \, d\sigma^{d(N-1)-1}_{r} \right\} du_N\,dr
$$
where $\sigma^{n-1}_R$ is the uniform probability measure on $\Sp^{n-1}(R)$. Making the change of coordinates $w=r^2 + |u_N|^2$ and $z=\sqrt{N}u_N$, we obtain
$$
\begin{aligned}
&\int_{0}^{\infty}\!\!\!\int_{\R^d} \varphi(z,w) \left\{ \frac{\left|\Sp^{d(N-1)-1}\left(\sqrt{w-
\frac{|z|^2}{N}}\right)\right|}{2\left(w-\frac{|z|}{N^2} \right)^{1/2}N^{d/2}} \int_{\Sp^{d(N-1)-1}\left(\sqrt{w-\frac{|z|^2}{N}}\right)} f^{\otimes N} \, d\sigma^{d(N-1)-1}_{\sqrt{w-|z|^2/N}} \right\} dz\,dw \\
&= \int_{0}^{\infty}\!\!\!\int_{\R^d} \varphi(z,w) \left\{ \frac{\left| \SS^N(\sqrt{w},z)\right|}{ 2\left(w-\frac{|z|}{N^2} \right)^{1/2} N^{d/2} } \,  Z_N(f;\sqrt{w},z) \right\} dz\,dw,
\end{aligned}
$$
from which we conclude.
\end{proof}

Since $S_N$ is the summation of independent random variables, its law's density is also given by
\begin{equation}\label{eq:sN}
s^N(z,u) = h^{\ast N} (z,u),
\end{equation}
and we deduce from the lemma above
\begin{equation}\label{eq:ZN-hN}
Z_N(f;\sqrt{u},z) = \frac{2\left(u-\frac{|z|}{N^2} \right)^{1/2} \, N^{d/2} \, h^{(\ast N)}(z,u) }{\left| \SS^N(\sqrt{u},z)\right|} \,.
\end{equation}

\begin{lemma}\label{lem:h-moment}
If $f\in \PPP_{2k}(\R^d)$ then $h\in \PPP_{k}(\R^{d+1})$.
\end{lemma}

\begin{proof}
Let $y=(v,u) \in \R^{d+1}$ with $v\in\R^d$ and $u\in\R$. Then we have
\begin{equation*}
\begin{aligned}
\int_{\R^{d+1}} \abs{y}^{k} \, h(y) \,dy &=\int_{\R^{d+1}}\left( |v|^2 + |u|^2 \right)^{k/2} f(v) \delta_{u=|v|^2} \, dv\, du   \\
&\leq C_{k}\left( \int_{\R^{d}} \abs{v}^k\, f(v) \,dv + 
 \int_{\R^{d}} \abs{v}^{2k}\, f(v) \, dz \right),
\end{aligned}
\end{equation*}
from which we conclude.
\end{proof}


\begin{lemma}\label{lem:h-Lq}
Suppose $f\in L^{p}(\R^d)$ for some $p>1$. Then $h^{\ast 2}\in L^q(\R^{d+1})$ 
if
\begin{enumerate}[(i)]
\item for $d=1$: $1<q<p$ and $q<\frac{2p}{p+1}$

\item for $d=2$: $q\leq p$

\item for $d\geq 3$: if $f\in L_s(\R^d)$ ($s>0$), for $q<p$ and
$$
q = \frac{(d-2)(p-1) + sp}{(d-2)(p-1) + s} > 1.
$$
\end{enumerate}
\end{lemma}

\begin{proof}
We compute first $h^{\ast 2}(v,u)$ with $v,v' \in \R^d$ and $u,u' \in \R$.
\begin{equation*}
\begin{aligned}
h^{\ast 2}(v,u) &= \int_{\R^d}\!\int_{\R} h(v-v',u-u')\, h(v',u') \,du' dv'\\
&= \int_{\R^d} f(v-v')\, f(v') \left\{\int_{\R} \delta_{u-u'=\abs{v-v'}^2} \delta_{u'=\abs{v'}^2} du'\right\} dv' \\
&= \int_{\R^d} f(v-v')\, f(v') \,\delta_{u=\abs{v-v'}^2 - \abs{v'}^2} \,dv'.
\end{aligned}
\end{equation*}

Moreover, we have
\begin{equation*}
\begin{aligned}
\delta_{u=\abs{v-v'}^2 - \abs{v'}^2} =\delta_{u=2\Abs{\frac{v}{2}-v'}^2 + \frac{\abs{v}^2}{2}}.
\end{aligned}
\end{equation*}

Then we can compute the $L^q$ norm of $h^{\ast 2}$,
\begin{equation}\label{eq:h2a}
\begin{aligned}
&\int_{\R^d}\!\int_{\R} {\Abs{h^{\ast 2}(v,u)}}^{q} \, dv\,du \\
&\qquad= \int_{\R^d}\!\int_{\R} 
 {\Abs{ \int_{\R^d} f(v-v')\, f(v') \, \delta_{u=2\Abs{\frac{v}{2}-v'}^2 + \frac{\abs{v}^2}{2}} \,dv'}}^{q}dv\,du\\
&\qquad \leq \int_{\R^d}\!\int_{\R} 
 {\Abs{ \left(\int_{\R^d} \delta_{\Abs{\frac{v}{2}-v'}^2 = \frac{u}{2} - \frac{\abs{v}^2}{4}} \,dv'\right)^{(q-1)/q} 
 \left(\int_{\R^d} f(v-v')^q \, f(v')^q \,\delta_{\Abs{\frac{v}{2}-v'}^2 = \frac{u}{2} - \frac{\abs{v}^2}{4}}\,dv'\right)^{1/q}}}^{q}dv\,du. 
\end{aligned}
\end{equation}
where we used Holder's inequality. 

We look to the integral over $\delta$, using $w=\frac{v}{2}-v'$
\begin{equation*}
\begin{aligned}
\int_{\R^d} \delta_{\Abs{w}^2 = \frac{u}{2} - \frac{\abs{v}^2}{4}} \,dw
 = \abs{\Sp^{d-1}} \int_\R \delta_{r^2 = \frac{u}{2} - \frac{\abs{v}^2}{4}} \,r^{d-1}\,dr
\end{aligned}
\end{equation*}
where we changed to polar coordinates and then, with $z=r^2$
\begin{equation}\label{eq:delta}
\begin{aligned}
\int_{\R^d} \delta_{\Abs{w}^2 = \frac{u}{2} - \frac{\abs{v}^2}{4}} \,dw
& = \frac{\abs{\Sp^{d-1}}}{2} \int_\R \delta_{z = \frac{u}{2} - \frac{\abs{v}^2}{4}}\, z^{(d-2)/2}\,dz \\
&= \frac{\abs{\Sp^{d-1}}}{2} \left(  \frac{u}{2} - \frac{\abs{v}^2}{4}   \right)^{(d-2)/2}.
\end{aligned}
\end{equation}
Therefore we obtain, plugging \eqref{eq:delta} in \eqref{eq:h2a} and using Fubbini,
\begin{equation*}
\begin{aligned}
&\int_{\R^d}\!\int_{\R} {\Abs{h^{\ast 2}(v,u)}}^{q} \, dv\,du \\
&\qquad \leq \int_{\R^d}\!\int_{\R^d} f(v-v')^q \, f(v')^q  
 \Bigg\{\int_{\R} \left[ \frac{\abs{\Sp^{d-1}}}{2} \left(  \frac{u}{2} - \frac{\abs{v}^2}{4}   \right)^{(d-2)/2}  \right]^{q-1}\, \delta_{u=2\Abs{\frac{v}{2}-v'}^2 + \frac{\abs{v}^2}{2}}\,du\Bigg\}dv\,dv'  \\
 &\qquad=  \frac{{\abs{\Sp^{d-1}}}^{q-1}}{2^{q-1}} \int_{\R^d}\!\int_{\R^d} \,{\Abs{\frac{v}{2}-v'}}^{(d-2)(q-1)} \, f(v-v')^q \, f(v')^q  \, dv\,dv' \quad =:A
\end{aligned}
\end{equation*}

Now we have the cases $d=1$, $d=2$ and $d\geq 3$:

\subsubsection*{(i) $d=1$}
Splitting the expression, we have
\begin{equation*}
\begin{aligned}
A &\leq \int_{\Abs{\frac{v}{2}-v'}\leq1} \frac{f(v-v')^q \, f(v')^q}{\Abs{\frac{v}{2}-v'}^{q-1}} \,   dv\,dv'  + \int_{\R^{d}}\!\int_{\R^{d}} f(v-v')^q \, f(v')^q  dv\,dv' \\
& =: T_1 + T_2. 
\end{aligned}
\end{equation*}
For the last estimate we have $T_2 \leq {\norm{f}}_{L^q}^{2q} \leq {\norm{f}}_{L^p}^{2q} $ (because $q<p$ and $f$ is a probability measure), and for the first term we use Holder's inequality
\begin{equation*}
\begin{aligned}
 T_1 &\leq  \left( \int_{\Abs{\frac{v}{2}-v'}\leq1} \frac{1}{\Abs{\frac{v}{2}-v'}^{(q-1)p/(p-q)}}\,dv\,dv'  \right)^{(p-q)/p} \left( \int_{\Abs{\frac{v}{2}-v'}\leq1} f(v-v')^p \, f(v')^p \,dv\,dv'\right)^{q/p}.
\end{aligned}
\end{equation*}
Then, the first integral converges if $(q-1)p/(p-q)<1$, which give us $T_1\leq C {\norm{f}}_{L^p}^{2q}$ if
$$
q< \frac{2p}{p+1}.
$$

\subsubsection*{(ii) $d=2$}
In this case we have
\begin{align*}
A &\leq \frac{{\abs{\Sp^{1}}}^{q-1}}{2^{q-1}} \int_{\R^{d}}\!\int_{\R^{d}} f(v-v')^q \, f(v')^q  dv\,dv' \\
&=\frac{{\abs{\Sp^{1}}}^{q-1}}{2^{q-1}} {\norm{f}}_{L^q}^{2q}\leq \frac{{\abs{\Sp^{1}}}^{q-1}}{2^{q-1}} {\norm{f}}_{L^p}^{2q}. 
\end{align*}

\subsubsection*{(iii) $d\geq 3$}
We have, using $w=v-v'$ and $u = v'$
\begin{align*}
A &= \frac{{\abs{\Sp^{d-1}}}^{q-1}}{2^{q-1}} \int_{\R^d}\!\int_{\R^d} \,{\Abs{\frac{v}{2}-v'}}^{(d-2)(q-1)} \, f(v-v')^q \, f(v')^q  dv\,dv'
\\
&= \frac{{\abs{\Sp^{d-1}}}^{q-1}}{2^{q-1}} \frac{1}{2^{(d-2)(q-1)}}\int_{\R^d}\!\int_{\R^d} \,{\abs{w-u}}^{(d-2)(q-1)} \, f(w)^q \, f(u)^q  dw\,du\\
&\leq\frac{{\abs{\Sp^{d-1}}}^{q-1}}{2^{(d-1)(q-1)}} \Bigg\{  2C   
\left( \int_{\R^d} {\abs{w}}^{(d-2)(q-1)} \, f(w)^q \, dw \right)\left( \int_{\R^d} f(u)^q \, du \right) \Bigg\}\\
&\leq C {\norm{f}}_{L^q}^{q} {\norm{f}}_{L^q_m}^{q}
\end{align*}
where we have used ${\abs{w-u}}^{(d-2)(q-1)}\leq C \left( {\abs{w}}^{(d-2)(q-1)} + {\abs{u}}^{(d-2)(q-1)}    \right)$ and $m=(d-2)(q-1)$. 

Finally, we have ${\norm{f}}_{L^q}^{q}\leq {\norm{f}}_{L^p}^{q}$ and with the hypothesis $f\in L^p\cap L_s$, we have ${\norm{f}}_{L^q_m}< \infty$ for $m=s(p-q)/(p-1)$ and $q<p$ (see Lemma \ref{lem:aux} in Appendix~\ref{ap:reg}), more precisely for
$$
q = \frac{(d-2)(p-1) + sp}{(d-2)(p-1) + s} > 1.
$$

\end{proof}

\subsection{Asymptotic behaviour of $Z'_N$}
In this section we shall study the behaviour of $Z'_N$ when $N$ goes to infinity. First of all, let us state a version of the Central Limit Theorem, also known as Berry-Esseen type theorem, which is the main ingredient of the proof of the asymptotic of $Z'_N$ in 
Theorem~\ref{thm:ZN}. 
The proof of the CLT presented here is a slightly adaptation of \cite[Theorem 4.6]{HaurayMischler} (see also \cite[Theorem 27]{CCLLV}).

\begin{thm}[Central Limit Theorem]\label{thm:clt}
Let $g\in \PPP_{3}(\R^D)$ such that, for some integer $k\ge 1$, we have $g^{\ast k} \in L^p(\R^D)$ for some $p>1$. Moreover, assume that 
\begin{equation}\label{eq:hyp-tcl}
\int_{\R^D} x\, g(x) \, dx =0 , \qquad \int_{\R^D} (x\otimes x)\, g(x)\, dx =I_D , \qquad
\int_{\R^D} |x|^3 g(x)\, dx \leq C_3.
\end{equation}
Then there exists a constant $C=C(D,p,\norm{g^{\ast k}}_{L^p})>0$ and $N(k,p)$ such that for all $N>N(k,p)$ we have 
\begin{equation*}\label{eq:clt}
\norm{g_N - \gamma}_{L^\infty} = \sup_{x\in\R^D} |g_N(x) - \gamma(x)| \leq \frac{C}{\sqrt{N}},
\end{equation*}
where $g_N(x) = N^{D/2} g^{\ast N}(\sqrt{N}x)$ is the normalized $N$-convolution power of $g$.
\end{thm}

In the sequel we will need the following lemma, and we refer again to \cite[Proposition 26]{CCLLV} and \cite[Lemma 4.8]{HaurayMischler} for its proof.

\begin{lemma}\label{lem:fourier}
\begin{enumerate}[(i)]
\item Consider $g\in \PPP_{3}(\R^D)$ satisfying \eqref{eq:hyp-tcl}. Then, there exists $\delta \in (0,1)$ such that
$$
\forall \, \xi \in B(0,\delta) \qquad |\widehat g(\xi)| \le e^{-|\xi|^2/4}.
$$

\item Consider $g \in \PPP(\R^D)\cap L^p(\R^D)$ for $1<p\le \infty$. For any $\delta >0$ there exists $\kappa(\delta) = \kappa(M_3(g),\norm{g}_{L^p}, \delta) \in (0,1)$ such that
$$
\sup_{|\xi| \ge \delta} |\widehat g(\xi)| \le \kappa(\delta).
$$

\end{enumerate}
\end{lemma}

\begin{proof}[Proof of Theorem \ref{thm:clt}]
We remark that
$$
\widehat g_N(\xi) =\widehat g \left( \frac{\xi}{\sqrt{N}} \right)^N , \qquad
\widehat \gamma_N(\xi) =\widehat \gamma \left( \frac{\xi}{\sqrt{N}} \right)^N.
$$
We have $g^{\ast k} \in L^1\cap L^p$, for $p\in(1,\infty]$, and then by the Hausdorff-Young inequality we deduce that $\widehat{(g^{\ast k} )} = (\widehat g)^k$ lies in $L^{p'}\cap L^\infty$ with $p'\in (1,\infty]$. Furthermore, $\widehat g_N (\xi) \in L^1$ for any $N\ge kp'$. Hence we shall use the inverse Fourier transform to write
\begin{equation}\label{eq:g-gamma}
\begin{aligned}
| g_N(x) - \gamma(x)| 
&= (2\pi)^D\left| \int_{\R^D} e^{i \xi\cdot x} \left(\widehat g_N(\xi) - \widehat\gamma(\xi) \right) \, d\xi \right| \\
&\le (2\pi)^D \int_{\R^D}  \left| \widehat g_N(\xi) - \widehat\gamma(\xi) \right|\, d\xi.
\end{aligned}
\end{equation}
Spliting the last integral in low and high frequencies, we obtain
$$
\begin{aligned}
\int_{\R^D}  \left| \widehat g_N(\xi) - \widehat\gamma(\xi) \right|\, d\xi 
&\le \int_{|\xi|\ge \sqrt{N} \delta} |\widehat g_N(\xi)| \, d\xi 
+  \int_{|\xi|\ge \sqrt{N} \delta} |\widehat \gamma(\xi)| \, d\xi \\
&\quad+ \int_{|\xi|< \sqrt{N} \delta} |\widehat g_N(\xi) - \widehat\gamma(\xi)| \, d\xi \\
&=: T_1 + T_2 + T_3,
\end{aligned}
$$
for some $\delta\in (0,1)$.

For the first term, we write
$$
\begin{aligned}
T_1 
&\le \int_{|\xi|\ge \sqrt{N} \delta} \left|\widehat g \left( \frac{\xi}{\sqrt{N}} \right)\right|^N \, d\xi 
= N^{D/2} \int_{|\eta|\ge\delta} |\widehat g(\eta)| \, d\eta \\
&\leq N^{D/2} \left( \sup_{\eta\ge \delta} |\widehat g(\eta)^k|   \right)^{N/k-p'} \int_{|\eta|\ge \delta} |\widehat g(\eta)^k|^{p'} \,d\eta \\
&\le N^{D/2} \kappa(\delta)^{N/k - p'} C_{D,p} \norm{g^{\ast k}}_{L^{p}}^{p'} 
\end{aligned}
$$
where $\delta \in (0,1)$ is given by Lemma \ref{lem:fourier}-(i) and $\kappa(\delta)$ is given by Lemma \ref{lem:fourier}-(ii) applied to $g^{\ast k}$ (because we have supposed only $g^{\ast k} \in L^p$). We get the same estimate for the second term, then we obtain that there exists a constant $C=C(D,p,\norm{g^{\ast k}}_{L^p})$ such that 
$$
T_1 + T_2 \le \frac{C}{\sqrt{N}}.
$$


Finally, for the third term we have
$$
T_3 = \int_{|\xi|< \sqrt{N} \delta} \frac{|\widehat g_N(\xi) - \widehat\gamma(\xi)|}{|\xi|^3} |\xi|^3 \, d\xi
$$
and we can estimate
$$
\begin{aligned}
\frac{|\widehat g_N(\xi) - \widehat\gamma(\xi)|}{|\xi|^3} 
&= \frac{1}{N^{3/2}} \, \frac{|\widehat g(\xi/\sqrt N)^N - \widehat\gamma(\xi/\sqrt N)^N|}{|\xi/\sqrt N|^3} \\
&=  \frac{1}{N^{3/2}} \, \frac{|\widehat g(\xi/\sqrt N) - \widehat\gamma(\xi/\sqrt N)|}{|\xi/\sqrt N|^3} \times \left| \sum_{k=0}^{N-1} \widehat g(\xi/\sqrt N)^k  \widehat\gamma(\xi/\sqrt N)^{(N-k-1)}  \right|.
\end{aligned}
$$
Moreover, point (i) in Lemma~\ref{lem:fourier} implies
$$
\left| \sum_{k=0}^{N-1} \widehat g(\xi/\sqrt N)^k  \widehat\gamma(\xi/\sqrt N)^{(N-k-1)}  \right| \le \sum_{k=0}^{N-1} e^{-\frac{k|\xi|^2}{4N}}e^{-\frac{(N-k-1)|\xi|^2}{4N}}
\le N e^{-\frac{|\xi|^2}{8}}.
$$
Hence, we obtain
$$
\begin{aligned}
T_3 &\le  \frac{1}{N^{3/2}} \left(  \sup_{\eta} \frac{|\widehat g(\eta) - \widehat\gamma(\eta)|}{|\eta|^3} \right) \int_{\R^D} N e^{-\frac{|\xi|^2}{8}} |\xi|^3 \, d\xi \\
&\le \frac{1}{\sqrt N}  \left( M_3(g) + M_3(\gamma) \right) C_D,
\end{aligned}
$$
and we finish the proof gathering the estimates of $T_1$, $T_2$ and $T_3$ togheter with \eqref{eq:g-gamma}.

\end{proof}

With these results we are able to state the following theorem about the asymptotic behaviour of $Z'_N$.

\begin{thm}
\label{thm:ZN}
Consider $f\in \PPP_{6}(\R^d) \cap L^p(\R^d)$, with $p>1$, satisfying \eqref{eq:hyp-f}. Then we have
\begin{equation*}\label{eq:thm-Z'Nbis}
\begin{aligned}
Z'_N(f;r,z) 
& =
\frac{\sqrt{2d}}{\Sigma\EE^{d/2}} \, 
\frac{(dN)^{\frac{d(N-1)-2}{2}}}{\left( r^2 - \frac{|z|^2}{N} \right)^{\frac{d(N-1)-2}{2}}} \,
\frac{e^{-\frac{dN}{2}}}{e^{-\frac{r^2}{2}}} 
\\
&\quad \times
\left[  \exp{\left(-\frac{|z|^2}{2\EE N}
-\frac{(r^2-NE)^2}{2\Sigma^2 N}\right)}  + O\left(1/\sqrt{N}\right)\right]
\end{aligned}
\end{equation*}
and in the particular case $r^2=dN$ and $z=0$, we have
\begin{equation*}\label{}
\begin{aligned}
Z'_N(f;\sqrt{dN},0) = \frac{\sqrt{2d}}{\Sigma \EE^{d/2}} \, 
\left[\exp\left(  - \frac{N(d-E)^2}{2\Sigma^2}   \right) + O\left(1/\sqrt{N}\right)\right].
\end{aligned}
\end{equation*}

\end{thm}

\begin{proof}

Let us introduce 
$$
g(v,u) = \Sigma \, \EE^{d/2} \, h(\EE^{1/2}v,E+ \Sigma u) \in \PPP(\R^{d+1}),
$$
with $v\in \R^d$ and $u\in \R$. Since $h$ lies in $\PPP_{ 3}(\R^{d+1})$ by Lemma~\ref{lem:h-moment} and $h^{\ast 2} \in L^q(\R^{d+1})$ for some $q \in (1,p)$ thanks to Lemma~\ref{lem:h-Lq}, we have $g \in \PPP_{ 3}(\R^{d+1})$ and $g^{\ast 2}\in L^q(\R^{d+1})$.

Moreover $g$ verifies (by construction)
\begin{equation*}
\int_{\R^{d+1}}y\, g(y) \,dy =0 \,,\quad 
\int_{\R^{d+1}}(y\otimes y) \, g(y)\,dy = I_{d+1} ,
\end{equation*}
where $I_{d+1}$ is the identity matrix in dimension $d+1$.

We can now apply Theorem \ref{thm:clt} to $g$, which implies that there exists $C>0$ and $N_0$ such that for all $N>N_0$,
\begin{equation*}
\begin{aligned}
\sup_{(v,u)\in \R^d\times \R} \left| g_N(v,u) - \gamma(v,u)\right| \leq \frac{C}{\sqrt{N}},
\end{aligned}
\end{equation*}
where $g_N(v,u) = N^{(d+1)/2} g^{\ast N}(\sqrt{N}v,\sqrt{N}u)$ is the normalized $N$-convolution power of $g$, with
$$
\begin{aligned}
g^{\ast N}(\sqrt{N}v,\sqrt{N}u)
&=  \Sigma \, \EE^{d/2} \, h^{\ast N}(\EE^{1/2} \sqrt{N} v,NE + \Sigma \sqrt{N}u) ,
\end{aligned}
$$ 
and 
$$ 
\gamma(v,u) = \frac{e^{-|v|^2/2}}{(2\pi)^{d/2}}\frac{e^{-u^2/2}}{(2\pi)^{1/2}}
$$
is the Gaussian measure in dimension $d+1$ (recall that we have $v\in \R^d$ and $u\in\R$).
It follows that
\begin{equation}\label{eq:hNconvolution}
\begin{aligned}
\sup_{(v,u)\in \R^d\times \R}
\left| 
h^{\ast N}(v,u)  - \frac{\Sigma^{-1}\EE^{-d/2}}{N^{(d+1)/2}}\, \gamma\left(\EE^{-1/2}N^{-1/2}v,\frac{u-NE}{\Sigma\sqrt{N}}\right)\right| 
\leq \frac{C}{\sqrt{N}}\,\frac{\Sigma^{-1}\EE^{-d/2}}{N^{(d+1)/2}}.
\end{aligned}
\end{equation}
Gathering \eqref{eq:hNconvolution} and \eqref{eq:ZN-hN} we obtain
\begin{equation*}
\begin{aligned}
&Z_N(f;r,z)  \\ 
&\qquad = \frac{2\,N^{d/2}\,\left(r^2-\frac{|z|}{N^2} \right)^{1/2}}{\left| \SS^N(r,z)\right|} \frac{\Sigma^{-1}\EE^{-d/2}}{N^{(d+1)/2}} \frac{1}{(2\pi)^{(d+1)/2}}
\left[  \exp{\left(-\frac{|z|^2}{2\EE N}
-\frac{(r^2-NE)^2}{2\Sigma^2 N}\right)}  + O\left(1/\sqrt{N}\right)\right].
\end{aligned}
\end{equation*}
Using \eqref{eq:mesureSN} we have
\begin{equation*}\label{}
\begin{aligned}
Z_N(f;r,z)  
& = \frac{2\,N^{d/2}\,\left(r^2-\frac{|z|}{N^2} \right)^{1/2}}{\left| \Sp^{d(N-1)-1}\right|}\left(r^2 - \frac{|z|^2}{N}   \right)^{-\frac{d(N-1)-1}{2}}_{+} \frac{\Sigma^{-1}\EE^{-d/2}}{N^{(d+1)/2}} \frac{1}{(2\pi)^{(d+1)/2}}\\
&\quad\times
\left[  \exp{\left(-\frac{|z|^2}{2\EE N}
-\frac{(r^2-NE)^2}{2\Sigma^2 N}\right)}  +  O\left(1/\sqrt{N}\right)\right].
\end{aligned}
\end{equation*}
Thanks to the formula 
$$
|\Sp^{n-1}| = \frac{2\pi^{n/2}}{\Gamma(n/2)}
$$
and to Stirling's formula,
$$
\Gamma\left( an+b \right) = \sqrt{2\pi} \, (an)^{\frac{an+b-1}{2}} \, e^{-an}\left( 1 + O(1/n)\right),
$$
we have 
$$
\Gamma\left(\frac{d(N-1)}{2}\right) = \sqrt{2\pi} \, (dN)^{\frac{d(N-1)-1}{2}} \, 2^{-\frac{d(N-1)-1}{2}}\, e^{-\frac{dN}{2}}\left( 1 + O(1/N)\right)
$$
and then
\begin{equation*}\label{eq:ZN}
\begin{aligned}
Z_N(f;r,z)  
& =
\frac{\sqrt{2d} }{\Sigma\EE^{d/2}} \, \left(\frac{e^{-\frac{r^2}{2}}}{(2\pi)^{\frac{dN}{2}}}\right) \,
\frac{(dN)^{\frac{d(N-1)-2}{2}}}{\left(r^2 - \frac{|z|^2}{N}  \right)^{\frac{d(N-1)-2}{2}}}\,
\frac{e^{-\frac{dN}{2}}}{e^{-\frac{r^2}{2}}} \,
\\
&\quad\times
\left[  \exp{\left(-\frac{|z|^2}{2\EE N}
-\frac{(r^2-NE)^2}{2\Sigma^2 N}\right)}  + O\left(1/\sqrt{N}\right)\right],
\end{aligned}
\end{equation*}
which implies for the case $r^2=dN$ and $z=0$
\begin{equation*}\label{}
\begin{aligned}
&Z'_N(f;\sqrt{dN},0)  
 =
\frac{\sqrt{2d} }{\Sigma\EE^{d/2}} \left[  \exp{\left(-\frac{N(d-E)^2}{2\Sigma^2}\right)}  
+ O\left(1/\sqrt{N}\right)\right].
\end{aligned}
\end{equation*}

\end{proof}

\subsection{Conditioned tensor product}

Consider now 
$$
F^N = [  f^{\otimes N} ]_{\SS^N_{B}} = \frac{f^{\otimes N}}{Z_N(f;\sqrt{dN},0)}\, \gamma^N
$$ 
the restriction of the $N$-fold tensor of $f$ to the Boltzmann's sphere $\SS^N_{B}$, where $f$ verifies $\eqref{eq:hyp-f}$ with $E=d$, more precisely with
$$
E=\int |v|^2 f = d,
$$
i.e. $f$ has the same second order moment that $\gamma$.

We have then the following theorem, which is a precise version of point $(i)$ in Theorem~\ref{thm:intro-FN}.

\begin{thm}\label{thm:FN-chaos}
Consider $f \in \PPP_{6}(\R^d) \cap L^p(\R^d)$, with $p>1$. Then, the sequence of probability measure $F^N\in\PPP(\SS^N_\BB)$ defined by $F^N = [  f^{\otimes N} ]_{\SS^N_{B}}$ is $f$-chaotic. 

More precisely, for any fixed $\ell$ there exists a constant $C:=C(\ell)>0$ such that for $N\ge \ell +1$ we have
\begin{equation*}\label{eq:FN-chaos}
W_1 (F^N_\ell,f^{\otimes \ell}) \leq {\norm{F^N_\ell -f^{\otimes \ell}}}_{L^1_1} \leq \frac{C}{\sqrt{N}}.
\end{equation*}

\end{thm}

\begin{proof}
With the notation $V = (v_1,\dots,v_N)\in \R^{dN}$, $V_\ell = (v_i)_{1\leq i\leq\ell}$, $V_{\ell,N} = (v_i)_{\ell+1\leq i\leq N}$ and $\bar V_\ell = \sum_{i=1}^{\ell} v_i$, we have from the definition of $F^N$
\begin{equation*}
\begin{aligned}
F^N (dV)&= \frac{f^{\otimes N}(V) \gamma^N(dV)}{Z_N(f;\sqrt{dN},0)} \\
&= \frac{f^{\otimes \ell}}{\gamma^{\otimes \ell}} (V_{\ell}) \, \frac{1}{Z_N'(f;\sqrt{dN},0)} \, 
\frac{f^{\otimes N-\ell}}{\gamma^{\otimes N-\ell}} (V_{\ell,N}) \, \gamma^N(dV) .
\end{aligned}
\end{equation*}
We recall that $\gamma^N = \gamma^N_{\sqrt{dN},0}$ and we have
\begin{equation*}
\begin{aligned}
\gamma^N_{\sqrt{dN},0}(dV) 
&= \gamma^N_\ell(dV_\ell) \, \gamma^{N-\ell}_{\sqrt{dN-|V_\ell|^2},z} (dV_{\ell,N})
\end{aligned}
\end{equation*}
where $z = -\sum_{i=1}^{\ell}v_i = - \bar V_\ell$.
We fix $\ell\ge 1$ and $N\ge \ell +1$, then we have
\begin{align*}
F^N_\ell (V_\ell) &= \int_{\R^{d(N-\ell)}} F^N(V) \, dV_{\ell,N}\\
&= \frac{f^{\otimes \ell}}{\gamma^{\otimes \ell}} (V_{\ell})\, \frac{\gamma^N_\ell(V_\ell)}{Z_N'(f;\sqrt{dN},0)}  \,
\int_{\SS^{N-\ell}\left(\sqrt{dN-|V_\ell|^2},z\right)} \frac{f^{\otimes N-\ell}}{\gamma^{\otimes N-\ell}} (V_{\ell,N}) \,
\gamma^{N-\ell}_{\sqrt{dN-|V_\ell|^2},z} (dV_{\ell,N}) \\
&=\frac{f^{\otimes \ell}}{\gamma^{\otimes \ell}} (V_{\ell})\,
\frac{Z_{N-\ell}'\left(f;\sqrt{dN-|V_\ell|^2},-\bar V_\ell \right)}{Z_N'(f;\sqrt{dN},0)} \, \gamma^N_\ell(V_\ell).
\end{align*}
Let us first compute the ratio betwenn $Z'_{N-\ell}$ and $Z'_N$, by Theorem~\ref{thm:ZN} we have 
\begin{align*}
\frac{Z_{N-\ell}'\left(f;\sqrt{dN-|V_\ell|^2},-\bar V_\ell \right)}{Z_N'(f;\sqrt{dN},0)} 
& = \frac{\left( d(N-\ell) \right)^{\frac{d(N-\ell-1)-2}{2}}}{\left(dN-|V_\ell|^2-\frac{|\bar V_\ell|^2}{N-\ell} \right)^{\frac{d(N-\ell-1)-2}{2}}} \, 
\frac{e^{-\frac{d(N-\ell)}{2}}}{e^{-\frac{(dN-|V_\ell|^2)}{2}}} \\
& \times
\left[  \exp{\left(-\frac{|\bar V_\ell|^2}{2\EE (N-\ell)}
-\frac{(d\ell-|V_\ell|^2)^2}{2\Sigma^2 (N-\ell)}\right)}  + O\left(N^{-1/2}\right)\right].
\end{align*}
Using the later expression with Lemma \ref{lem:gammaNl} one obtains
\begin{align*}
F^N_\ell (V_\ell) 
&=\frac{f^{\otimes \ell}}{\gamma^{\otimes \ell}} (V_{\ell})\,
\frac{\left( d(N-\ell) \right)^{\frac{d(N-\ell-1)-2}{2}}}{\left(dN-|V_\ell|^2-\frac{|\bar V_\ell|^2}{N-\ell} \right)^{\frac{d(N-\ell-1)-2}{2}}} \, 
\frac{e^{\frac{d\ell}{2}}}{e^{\frac{|V_\ell|^2}{2}}} \\
&\qquad \times
 \left[  \exp{\left(-\frac{|\bar V_\ell|^2}{2\EE (N-\ell)}
-\frac{(d\ell - |V_\ell|^2)^2}{2\Sigma^2 (N-\ell)}\right)}  + O\left(N^{-1/2}\right)\right] \\
&\qquad \times
\frac{ \Abs{\Sp^{d(N-\ell-1)-1}} }{\Abs{\Sp^{d(N-1)-1}}} \,
\frac{ \left( dN-|V_\ell|^2 - \frac{|\bar V_\ell|^2}{N-\ell}    \right)^{\frac{d(N-\ell-1)-2}{2}}_+ }{(dN)^{\frac{d(N-1)-2}{2}} \left(\frac{N-\ell}{N}\right)^{\frac{d}{2}} } \\
 &= f^{\otimes \ell}\left[  \exp{\left(-\frac{|\bar V_\ell|^2}{2\EE (N-\ell)}
-\frac{(d\ell - |V_\ell|^2)^2}{2\Sigma^2 (N-\ell)}\right)}  + O\left(N^{-1/2}\right)\right]
{\bf 1}_{dN-|V_\ell|^2 - \frac{|\bar V_\ell|^2}{N-\ell} > 0}
 \\
&\qquad \times
 \frac{ \Abs{\Sp^{d(N-\ell-1)-1}} }{\Abs{\Sp^{d(N-1)-1}}}\, \frac{(d(N-\ell))^{\frac{d(N-\ell-1)-2}{2}}}{(dN)^{\frac{d(N-1)-2}{2}}} \,
 \left( \frac{N}{N-\ell} \right)^{d/2} 
 \, \left(2\pi e\right)^{\frac{d\ell}{2}}  .
\end{align*} 
Since 
$$
\left( \frac{N}{N-\ell} \right)^{d/2}  =O(1),
$$
we have
\begin{equation}\label{eq:FN_l}
\begin{aligned}
F^N_\ell (V_\ell) 
 &= f^{\otimes \ell}(V_\ell) \, \theta^N_1(V_\ell) \, \theta^N_2(V_\ell)
  \end{aligned} 
\end{equation}
with
\begin{equation}\label{eq:FN_l-theta}
\begin{aligned}
&\theta^N_1 = \left[  \exp{\left(-\frac{|\bar V_\ell|^2}{2\EE (N-\ell)}
-\frac{(d\ell - |V_\ell|^2)^2}{2\Sigma^2 (N-\ell)}\right)}  + O\left(N^{-1/2}\right)\right]
{\bf 1}_{dN-|V_\ell|^2 - \frac{|\bar V_\ell|^2}{N-\ell} > 0},
\\
&\theta^N_2 = \frac{ \Abs{\Sp^{d(N-\ell-1)-1}} }{\Abs{\Sp^{d(N-1)-1}}}\, \frac{(d(N-\ell))^{\frac{d(N-\ell-1)-2}{2}}}{(dN)^{\frac{d(N-1)-2}{2}}} \, \left(2\pi e\right)^{\frac{d\ell}{2}}.
\end{aligned}
\end{equation}
Thanks to Stirling's formula again, we obtain
$$
\frac{ \Abs{\Sp^{d(N-\ell-1)-1}} }{\Abs{\Sp^{d(N-1)-1}}} = \left( \frac{dN}{2\pi}\right)^{\frac{dl}{2}}\left(  1 + O(N^{-1})\right), \qquad 
\theta^N_2 =   1 + O(N^{-1}).
$$
Moreover we can easily see by \eqref{eq:FN_l-theta} that ${\Vert \theta^N_1 \Vert}_{L^\infty} \le C$ uniformly in $N$, and
\begin{equation}\label{eq:thetaN-1}
\begin{aligned}
|\theta^N_1(V_\ell) - 1| &= |\theta^N_1(V_\ell) - 1|{\bf 1}_{|V_\ell|\le R} + 
|\theta^N_1(V_\ell) - 1|{\bf 1}_{|V_\ell|\ge R} \\
&\le \left| \left(\frac{|\bar V_\ell|^2}{2\EE (N-\ell)}
+\frac{(d\ell - |V_\ell|^2)^2}{2\Sigma^2 (N-\ell)}\right) + O\left(1/\sqrt{N}\right)\right|{\bf 1}_{|V_\ell|\le R} + C \frac{|V_\ell|^b}{R^{b}} {\bf 1}_{|V_\ell|\ge R} \\
&\le C\left( \frac{R^2}{N} + \frac{R^4}{N} + O\left(1/\sqrt{N}\right)\right) {\bf 1}_{|V_\ell|\le R} + C \frac{|V_\ell|^b}{R^{b}} {\bf 1}_{|V_\ell|\ge R} ,
\end{aligned}
\end{equation}
for some $R>0$ and $b\ge0$.

Finally, choosing $R=N^{1/8}$ and $b=4$ one has
\begin{equation*}\label{}
\begin{aligned}
{\norm{F^N_\ell - f^{\otimes\ell}}}_{L^1_1} 
&= {\norm{(\theta^N_1\theta^N_2-1)f^{\otimes\ell}}}_{L^1_1} \\
&\le (\theta^N_2-1){\norm{\theta^N_1 \, f^{\otimes\ell}}}_{L^1_1} 
  + {\norm{(\theta^N_1-1)f^{\otimes\ell}}}_{L^1_1}  \\
&\le \frac{C}{N} {\norm{f^{\otimes\ell}}}_{L^1_1} 
+ \frac{C}{\sqrt{N}}  {\norm{f^{\otimes\ell}}}_{L^1_1} 
+ \frac{C}{\sqrt{N}} {\norm{f^{\otimes\ell}}}_{L^1_5}\\
&\leq \frac{C\ell}{N} {\norm{f}}_{L^1_1} 
+ \frac{C\ell}{\sqrt{N}}  {\norm{f}}_{L^1_1} 
+ \frac{C\ell}{\sqrt{N}}  {\norm{f}}_{L^1_5}.
\end{aligned}
\end{equation*}

\end{proof}

\section{Entropic and Fisher's information chaos}\label{sec:entropy}

We recall that in the Subsection~\ref{ssec:results} we defined the relative entropy and relative Fisher's information of a probability measure. Moreover, we defined stronger notions of chaos, namely the entropic chaos in Definition~\ref{def:ent-chaos} and the Fisher's information chaos in Definition~\ref{def:Fisher-chaos}. 
We prove in this section precise versions of point (ii) in Theorem~\ref{thm:intro-FN}, Theorem~\ref{thm:intro-GN} and Theorem~\ref{thm:intro-op}.

\subsection{Entropic chaos for the conditioned tensor product}\label{ssec:ent-FN}
We shall study now the entropic chaoticity of the probability measure $F^N = [ f^{\otimes N} ]_{\SS^N_\BB}$ with quantitative rate in the following theorem, which is a precise version of point $(ii)$ of Theorem~\ref{thm:intro-FN}.

\begin{thm}\label{thm:ent-chaos}
Let $f\in \PPP_{6}(\R^d)\cap L^p(\R^d)$ for some $p>1$ verify
$\int v f = 0$ and $\int|v|^2 f=d$.
Then,
the sequence of probabilities $F^N:=[f^{\otimes N}]_{\SS^N_\BB} \in \PPP(\SS^N_\BB)$ is entropically $f$-chaotic. 
More precisely, there exists $C>0$ such that we have
\begin{equation*}\label{eq:ent-chaos}
\begin{aligned}
\left|\frac{1}{N} \,  H(F^N|\gamma^N) -H(f|\gamma) \right| \leq \frac{C}{\sqrt{N}}.
\end{aligned}
\end{equation*}
\end{thm}

\begin{proof}
We write 
\begin{equation*}\label{}
\begin{aligned}
\frac{1}{N} \,  H(F^N | \gamma^N) 
&= \frac{1}{N} \,  \int_{\SS^N_\BB} \left( \log \frac{dF^N}{d\gamma^N}   \right) \, dF^N \\
&= \frac{1}{N} \,  \int_{\SS^N_\BB} \left( \log \frac{f^{\otimes N}}{Z_N'(f;\sqrt{dN},0)\,\gamma^{\otimes N}}   \right) \, dF^N \\
&= \int_{\R^d} \left( \log \frac{f}{\gamma}   \right) dF^N_1 - \frac{1}{N} \, \log Z_N'(f;\sqrt{dN},0).
\end{aligned}
\end{equation*}

Thanks to the assumptions on $f$, we can use Theorem \ref{thm:ZN} to obtain
\begin{equation*}\label{}
\begin{aligned}
\frac{1}{N} \, H(F^N | \gamma^N) 
&= \int_{\R^d} \left( \log \frac{f}{\gamma}   \right) dF^N_1 + O(1/N).
\end{aligned}
\end{equation*}
Using \eqref{eq:FN_l}-\eqref{eq:FN_l-theta} we have $F^N_1(v) = \theta^N_1(v)\, \theta^N_2(v)\, f(v)$ or more precisely
\begin{equation*}\label{}
\begin{aligned}
F^N_1(v) = f(v) \left( e^{-\frac{|v|^2}{2N}-\frac{|v|^4}{2N}} + O\left(1/\sqrt{N}\right) \right) \left( 1 + O(1/N)\right)=:\theta^N(v)\, f(v),
\end{aligned}
\end{equation*}
and then
\begin{equation}\label{eq:HN-H}
\begin{aligned}
\frac{1}{N} \, H(F^N | \gamma^N) -
 H(f|\gamma) &= \int_{\R^d} (\theta^N -1)\, f \left( \log \frac{f}{\gamma}   \right) + O(1/N).
\end{aligned}
\end{equation}

We estime now the first term of the right-hand side, denoted by $T$,
\begin{equation*}\label{}
\begin{aligned}
|T| &\le \int_{\R^d} |\theta^N -1|\, f\, |\log \gamma|\, dv 
+ \int_{\R^d} |\theta^N -1|\, f \, |\log f|\, dv \\
&\le \int_{\R^d} |\theta^N -1|\, f\, C (1+|v|^2)\, dv 
+ \int_{\R^d} |\theta^N -1|\, f|\log f|\, dv\\
&=: T_1 + T_2.
\end{aligned}
\end{equation*}

We recall that (already computed in equation~\eqref{eq:thetaN-1})
$$
|\theta^N - 1| \le C
\left( \frac{R^2}{N} + \frac{R^4}{N} + \frac{1}{\sqrt{N}}\right) {\bf 1}_{|v|\le R} + C \frac{|v|^k}{R^k} {\bf 1}_{|V_\ell|\ge R}
$$
for some $k\ge 0$ and $R>0$. Then, for the first term we have
\begin{equation*}\label{}
\begin{aligned}
|T_1| &\le \int_{B_R} |\theta^N -1|\, f\ (1+|v|^2) + 
\int_{B_R^C} |\theta^N -1|\, f\, (1+|v|^2) \\
&\le \left( \frac{R^2}{N} + \frac{R^4}{N} + \frac{1}{\sqrt{N}} \right) {\norm{f}}_{L^1_2}
+ \frac{1}{R^k}\Big( M_{k}(f) + M_{k+2}(f)\Big)\\
&\le \frac{C_f}{\sqrt{N}}
\end{aligned}
\end{equation*}
where we have chosen $R=N^{1/8}$ and $k=4$.

For the last term $T_2$, define $A>1$ and $B_R = \{ v\in\R^d;\; |v|\le R  \}$, then we have
\begin{equation*}\label{}
\begin{aligned}
|T_2| &\le \int_{B_R} |\theta^N -1|\, f \, |\log f| + 
\int_{B_R^C} |\theta^N -1|\, f \, |\log f|\, {\bf 1}_{f\ge A} \\
&\quad + \int_{B_R^C} |\theta^N -1|\, f \, |\log f|\, {\bf 1}_{1\le f\le A}
+ \int_{B_R^C} |\theta^N -1|\, f \, |\log f|\, {\bf 1}_{e^{-|v|^2}\le f\le 1}\\
&\quad 
+ \int_{B_R^C} |\theta^N -1|\, f \, |\log f|\, {\bf 1}_{0\le f\le e^{-|v|^2}}.
\end{aligned}
\end{equation*}
Now we compute each one of this five terms.
First, we deduce that
\begin{equation*}\label{}
\begin{aligned}
|T_{2,1}| &\le \left( \frac{R^2}{N} + \frac{R^4}{N} + \frac{1}{\sqrt{N}} \right)\int_{B_R}  f \, |\log f|  = \left( \frac{R^2}{N} + \frac{R^4}{N} + \frac{1}{\sqrt{N}} \right) C_f.
\end{aligned}
\end{equation*}
For the second term, we use that $f\, |\log f| \le f^{(1+p)/2} \le f^p / A^{(p-1)/2}$ over 
$\{ f\ge A , |v| \ge R\}$, and then
\begin{equation*}\label{}
\begin{aligned}
|T_{2,2}| &\le \frac{{\norm{f}}_{L^p}^p}{A^{(p-1)/2}}.
\end{aligned}
\end{equation*}
Using $f\, |\log f| \le f\, |\log A|  $ over $ \{ 1\le f\le A , |v| \ge R\} $ for the third one, we obtain
\begin{equation*}\label{}
\begin{aligned}
|T_{2,3}| &\le \frac{\log A}{R^k}\, M_k(f).
\end{aligned}
\end{equation*}
Thanks to $f\, |\log f| \le f |v|^2 \le f |v|^{m+2}/R^{m}  $ over $ \{ e^{-|v|^2}\le f\le 1 , |v| \ge R\} $, we get
\begin{equation*}\label{}
\begin{aligned}
|T_{2,4}| &\le \frac{1}{R^{m}}\, M_{m+2}(f).
\end{aligned}
\end{equation*}
Finally, by $f\, |\log f| \le 4\sqrt{f} \le 4 e^{-|v|^2/2}$ over $ \{ 0\le f\le e^{-|v|^2} , |v| \ge R\} $
\begin{equation*}\label{}
\begin{aligned}
|T_{2,4}| &\le C\, e^{-R}.
\end{aligned}
\end{equation*}

Putting togheter all this terms, we have
\begin{equation*}\label{}
\begin{aligned}
|T_2| &\leq \left( \frac{R^2}{N} + \frac{R^4}{N} + \frac{1}{\sqrt{N}} \right) C_f 
+ \frac{{\norm{f}}_{L^p}^p}{A^{(p-1)/2}} + \frac{\log A}{R^k}\, M_k(f)
+ \frac{M_{m+2}(f)}{R^{m}} + C\, e^{-R} \\
&\leq  \frac{C_f}{\sqrt{N}}
\end{aligned}
\end{equation*}
choosing $A^{(p-1)/2}=R^k$, $R=N^{1/8}$, $k=6$ and $m=4$.

We have then $|T|\leq C \, N^{-1/2}$ and we conclude plugging it in \eqref{eq:HN-H}.

\end{proof}

\subsection{Relations between the different notions of chaos}
First of all, we start with the following lemma and we refer to \cite{CCLLV,HaurayMischler,CoursMischler} and the references therein for a proof.

\begin{lemma}\label{lem:EntropRepresentH} For all probabilities $\mu,\nu \in \PPP(Z)$ on a locally compact metric space, we have 
\begin{equation*}\label{lem:EntroRep1}
\begin{aligned}
H(\mu|\nu) &= \sup_{\varphi \in C_b(Z)} \left\{\int_Z \varphi \, d\mu - \log \left( \int_Z e^\varphi \, d\nu \right) \right\}
\\ 
&= \sup_{\varphi \in C_b(Z), \, \int_Z e^\varphi \, d\nu = 1 }\int_Z \varphi \, d\mu.
\end{aligned}
\end{equation*}
\end{lemma}

The following theorem is an adaptation of \cite[Theorem 17]{CCLLV}, where the same result is proved for probability measures on the usual sphere $\Sp^{N-1}(\sqrt N)$ in $\R^N$.

\begin{thm}
\label{thm:ent-semicont-general}
Consider $g\in\PPP_{6}(\R^d) \cap L^p(\R^d)$, for some $p\in(1,\infty]$, where 
$g$ satisfies $\int v g = 0$ and $\int|v|^2 g=d$.
Consider $G^N$ a probability measure on $\SS^N_\BB$ such that for some positive integer $\ell$, we have $G^N_\ell \wto \pi_\ell$ 
in $\PPP(\R^{d\ell})$ when $N$ goes to infinity.

Then, we have
\begin{equation*}\label{}
\begin{aligned}
\frac{1}{\ell} \, H(\pi_\ell|g^{\otimes \ell}) \leq \liminf_{N\to\infty} 
\frac{1}{N}\, H\left(G^N| \left[ g^{\otimes N}  \right]_{\SS^N_\BB}  \right) .
\end{aligned}
\end{equation*}

\end{thm}

\begin{proof}
Let fix a function $\varphi:=\varphi(v_1,\dots,v_\ell) \in C_b(\R^{d\ell})$ such that 
\begin{equation}\label{eq:EntropRepreHfgamma}
\int_{\R^{d\ell}} e^\varphi \, g^{\otimes \ell} = 1, \qquad   H(\pi_\ell|g^{\otimes \ell}) \leq \int_{\R^{d\ell}} \varphi \, d\pi_\ell + \eps
\end{equation}
for some $\eps >0$, which is possible thanks to Lemma~\ref{lem:EntropRepresentH}. We introduce the function 
$$
\Phi(v_1, ..., v_N) := \varphi(v_1,\dots,v_\ell) + \cdots + \varphi(v_{(m-1)\ell + 1},\dots,v_{m\ell}),
$$
where $m$ is the integer part of $N/\ell$, i.e. $N=m\ell+r$ with $0\leq r\leq \ell-1$.
Thanks again to Lemma~\ref{lem:EntropRepresentH} we have
\begin{equation*}
\begin{aligned}
 \frac{1}{N} \,  H\left(G^N|\left[ g^{\otimes N}  \right]_{\SS^N_\BB}\right) 
 &\ge  {1 \over N} \int_{\SS^N_\BB} \Phi \,   G^N(dV) -  {1 \over N} \log \left( \int_{\SS^N_\BB}e^\Phi \, d\left[ g^{\otimes N}  \right]_{\SS^N_\BB} \right).
\end{aligned}
\end{equation*}

For the first term of the right-hand side, using the symmetry of $G^N$ and the convergence of its $\ell$-marginal, we have
$$
{1 \over N} \int_{\SS^N_\BB} \Phi \,   G^N(dV) = \frac{m}{N} \int_{\R^{d\ell}} \varphi \, dG^N_\ell \xrightarrow[N\to\infty]{} \frac{1}{\ell} \int_{\R^{d\ell}} \varphi \, d\pi_\ell.
$$

We note that the second term of the right-hand side can be written in the following way
$$
\int_{\SS^N_\BB}e^\Phi\, d\left[ g^{\otimes N}  \right]_{\SS^N_\BB} = 
\frac{ 1}{Z'_N (g; \sqrt{dN},0)} \int_{\SS^N_\BB} e^\Phi  \left(\frac{g}{\gamma}\right)^{\otimes N}\,d\gamma^N
$$
since 
$$
\left[ g^{\otimes N}  \right]_{\SS^N_\BB} = \frac{g^{\otimes N}}{Z_N(g;\sqrt{dN},0)}\, \gamma^N.
$$

Applying Theorem \ref{thm:ZN} and thanks to $\int |v|^2\, g=d$ we get 
$$
Z'_N( g;\sqrt{dN},0) = \frac{\sqrt{2d}}{\Sigma(g)} \left(1+O(1/\sqrt{N})\right), 
$$
where $\Sigma(g)$ is given by \eqref{eq:hyp-f} applied to $g$,
and then
\begin{equation}\label{eq:limZNg}
\lim_{N\to\infty} \left( \frac{1}{N} \, \log Z'_N(g;\sqrt{dN},0)\right) = 0.
\end{equation}

For the other term, denoting $u=(v_1,\dots,v_{m\ell})$, $w=(v_{m\ell+1},\dots,v_{N})$ and 
$\bar w = v_{m\ell+1}+\cdots + v_{N} $, we write
$$
\begin{aligned}
&\int_{\SS^N(\sqrt{dN},0)} e^\Phi \left(\frac{g}{\gamma}\right)^{\otimes N}\,d\gamma^N \\
&\quad= \int_{\R^{dr}}  \frac{|\Sp^{d(N-r-1)-1}|}{|\Sp^{d(N-1)}|}\, 
\frac{\left(dN-|w|^2-\frac{|\bar w|^2}{N-r} \right)^{\frac{d(N-r-1)-2}{2}}}{(dN)^{\frac{d(N-1)-2}{2}}} \, \left(\frac{N}{N-r}\right)^{\frac{d}{2}} \, \left(\frac{g}{\gamma}\right)^{\otimes r}\\
&\qquad \times
\left\{   \int_{\SS^{\ell m}(\sqrt{dN-|w|^2},-\bar w)} \left(\frac{e^\varphi \, g^{\otimes \ell}}{\gamma^{\otimes \ell}}\right)^{\otimes m}\,d\gamma^N_{\sqrt{dN-|w|^2},-\bar w}  \right\} dw
\end{aligned}
$$
where the integral in $dw$ have to be taken over the region 
$$
\{ w\in \R^{dr} \,|\, dN - |w|^2 - |\bar w|^2/(\ell m) > 0   \}.
$$

We recognize that the last integral is equal to $Z'_m\left( e^\varphi \, g^{\otimes \ell}; \sqrt{dN-|w|^2}, -\bar w  \right)$ (where $Z'_m$ is a multi-dimensional version of $Z'_N$, obtained replacing $N$ by $m\ell$) and by Theorem \ref{thm:ZN} we have
$$
\begin{aligned}
&Z'_m\left( e^\varphi \, g^{\otimes \ell}; \sqrt{dN-|w|^2}, -\bar w  \right)
\\
&\quad = O(1) \times 
\frac{(d\ell m)^{\frac{d(\ell m-1)-2}{2}}}{\left( dN-|w|^2 - \frac{|\bar w|^2}{\ell m} \right)^{\frac{d(\ell m-1)-2}{2}}} \,
\frac{e^{-\frac{d\ell m}{2}}}{e^{-\frac{(dN-|w|^2)}{2}}} 
\end{aligned}
$$
and using \eqref{eq:measureS}, we get
$$
\begin{aligned}
\int_{\SS^N(\sqrt{dN},0)}e^\Phi \left(\frac{g}{\gamma}\right)^{\otimes N}\,d\gamma^N 
&= C\int_{\R^{dr}} e^{-\frac{|w|^2}{2}}  
\left(\frac{g}{\gamma}\right)^{\otimes r} dw\\
&= O(1)\times (2\pi)^{dr/2} \int_{\R^{dr}} g^{\otimes r} dw = O(1).
\end{aligned}
$$

With these estimates at hand, we can deduce 
$$
\liminf_{N\to\infty} \left( -\frac{1}{N} \,  \log\int_{\SS^N(\sqrt{dN},0)} e^\Phi \left(\frac{g}{\gamma}\right)^{\otimes N}\,d\gamma^N  \right) \ge 0
$$
and together with \eqref{eq:limZNg} we obtain
$$
\liminf_{N\to\infty} \frac{1}{N} \,  H(G^N|\left[ g^{\otimes N}  \right]_{\SS^N_\BB}) 
 \ge  \frac{1}{\ell} \int_{\R^{d\ell}} \varphi \, d\pi_\ell \ge \frac{1}{\ell} \, H(\pi_\ell | g^{\otimes \ell}) - \eps.
$$
Since $\eps$ is arbitrary, we can conclude letting $\eps\to 0$.

\end{proof}

Our aim now is to give an analogous result of Theorem~\ref{thm:ent-semicont-general} for the Fisher's information. 
However the strategy here is different, it is not based on the asymptotic behaviour of $Z'_N$ like before, but on a geometric approach following \cite{HaurayMischler}, where this analogous result is proved in the Kac's sphere setting.
To this purpose, firstly we shall present some results to conclude with the Theorem~\ref{thm:fisher-semicont}.

Consider $W=(w_1,\dots,w_N) \in \R^{dN}$ and $V=(v_1,\dots,v_N) \in \SS^N_\BB$, where we recall that
$v_i=(v_{i,\alpha})_{1\le \alpha\le d}$, $w_i=(w_{i,\alpha})_{1\le \alpha\le d} \in \R^d$ for all $1\le i\le N$.

Let $P_h$ be the projection on the hyperplane $\{ X\in \R^{dN} \;;\; \sum_{i=1}^{N} x_i=0 \}$, then it can be computed in the following way
\begin{equation*}
P_h W = W - \sum_{\alpha=1}^d \left(W\cdot \frac{e^N_\alpha}{|e^N_\alpha|} \right) \frac{e^N_\alpha}{|e^N_\alpha|},
\end{equation*}
where $e^N_\alpha = (e_\alpha, \dots, e_\alpha) \in \R^{dN}$ with $e_\alpha = (\delta_{\alpha\beta})_{1\le \beta\le d} \in \R^d$. Since $|e^N_\alpha|=\sqrt N$ we obtain
\begin{equation}\label{P_h}
P_h W = W - \frac{1}{N}\sum_{\alpha=1}^d (W\cdot e^N_\alpha)\, e^N_\alpha.
\end{equation}
Moreover, the projection $P_s$ on the sphere $\{ X\in \R^{dN} \;;\; \sum_{i=1}^{N} |x_i|^2=dN \}$ is given by
\begin{equation}\label{P_s}
P_s W = \sqrt{dN} \, \frac{W}{|W|}.
\end{equation}
Hence the projection $P_{\SS}$ on the Boltzmann's sphere $\SS^N_\BB$ can be computed as the composition of the others, i.e. $P_{\SS} = P_s \circ P_h$, more precisely
\begin{equation}\label{P_S}
\begin{aligned}
P_{\SS} W &= (P_s \circ P_h) W \\
&= \sqrt{dN} \, \frac{P_h W}{|P_h W|}\\
&= \sqrt{dN} \, \frac{W - \frac{1}{N}\sum_{\alpha=1}^d (W\cdot e^N_\alpha)\, e^N_\alpha}{\left|W - \frac{1}{N}\sum_{\alpha=1}^d (W\cdot e^N_\alpha)\, e^N_\alpha \right|},
\end{aligned}
\end{equation}
or in coordinates, for $1\le j\le N$ and $1\le \beta\le d$,
\begin{equation}\label{P_Sjbeta}
\begin{aligned}
(P_{\SS} W)_{j,\beta}
& = \frac{\sqrt{dN}}{\left|W - \frac{1}{N}\sum_{\alpha=1}^d (W\cdot e^N_\alpha)\, e^N_\alpha \right|} \, \left( w_{j,\beta} - \frac{1}{N}\sum_{k=1}^N w_{k,\beta} \right).
\end{aligned}
\end{equation}

Consider $V\in \SS^N_\BB$ and a smooth function $F$ defined on $\SS^N_\BB$. Then
the gradient $\nabla_h$ on $\{ X\in \R^{dN} \;;\; \sum_{i=1}^{N} x_i=0 \}$ is (recall that $\nabla$ stands for the usual gradient on $\R^{dN}$)
$$
\nabla_h F(V) = \nabla F(V) -\frac{1}{N} \sum_{i=1}^N \sum_{\alpha=1}^d \partial_{v_{i,\alpha}} F(V) \, e^N_\alpha.
$$
Moreover, the gradient $\nabla_s$ on the sphere $\{ X\in \R^{dN} \;;\; \sum_{i=1}^{N} |x_i|^2=dN \}$ is given by
$$
\nabla_s F(V) = \nabla F(V) -\left(\frac{V}{|V|} \cdot \nabla F(V) \right)\frac{V}{|V|}.
$$
Combining them we can compute the gradient on $\SS^N_\BB$, which is given by
\begin{equation}\label{Grad_SB}
\begin{aligned}
\nabla_{\SS} F(V) &= \nabla_h F(V) -\left(\frac{V}{|V|} \cdot \nabla_h F(V) \right)\frac{V}{|V|} \\
&= \nabla F(V) -\frac{1}{N} \sum_{i=1}^N \sum_{\alpha=1}^d \partial_{v_{i,\alpha}} F(V) \, e^N_\alpha \\
&\quad - \left[V \cdot \nabla F(V) -\frac{1}{N} \sum_{i=1}^N \sum_{\alpha=1}^d \partial_{v_{i,\alpha}} F(V) \, (e^N_\alpha \cdot V)    \right] \frac{V}{|V|^2} \\
&=\nabla F(V) -\frac{1}{N} \sum_{i=1}^N \sum_{\alpha=1}^d \partial_{v_{i,\alpha}} F(V) \, e^N_\alpha - \left[V \cdot \nabla F(V) \right] \frac{V}{|V|^2},
\end{aligned}
\end{equation}
since $e^N_\alpha \cdot V = \sum_{i=1}^N v_{i,\alpha} = 0$ because $V\in \SS^N_\BB$.

Let $\Phi$ be a smooth vector field on $\R^{dN}$, which written in composants is $\Phi(V) = (\Phi_1(V), \dots , \Phi_N(V))$ with $\Phi_i(V) = (\Phi_{i,1}(V), \dots, \Phi_{i,d}(V))$ for $1\le i \le N$. We denote by $\Div_{\SS}$ the divergence on $\SS^N_\BB$, then it can be computed in the following way

\begin{equation*}
\begin{aligned}
\Div_{\SS} \Phi(V) &= \sum_{j=1}^N \sum_{\beta=1}^d \nabla_{\SS} \Phi_{j,\beta}(V) \cdot e_{j,\beta},
\end{aligned}
\end{equation*}
where $e_{j,\beta} = (\delta_{jk}\delta_{\beta\gamma})_{(1\le k\le N)(1\le \gamma \le d)} \in \R^{dN}$. Using \eqref{Grad_SB} and after some simplifications we obtain
\begin{equation}\label{div_S}
\begin{aligned}
\Div_{\SS} \Phi(V) &= \Div \Phi(V) - \frac{1}{N}\sum_{j=1}^N \sum_{\beta=1}^d \sum_{i=1}^N \partial_{v_{i,\beta}} \Phi_{j,\beta}(V)
- \sum_{j=1}^N \sum_{\beta=1}^d  V \cdot \nabla \Phi_{j,\beta}(V) \, \frac{v_{j,\beta}}{|V|^2}.
\end{aligned}
\end{equation}

%
%
%
%
%

\begin{lemma}\label{lem:ipp}
Consider a function $F$ and a vector field $\Phi$, smooth enough, defined on $\SS^N_\BB$. 
Then the following integration by parts formula on $\SS^N_\BB$ holds
$$
\int_{\SS^N_\BB} \left\{ \nabla_\SS F(V) \cdot \Phi(V) + F(V) \Div_\SS \Phi(V) - \frac{d(N-1)-1}{dN} \, F(V) \Phi(V) \cdot V \right\} d\gamma^N(V) = 0.
$$
\end{lemma}

\begin{proof}
The proof presented here is an adaptation of \cite[Lemma 4.16]{HaurayMischler}.
Let $\chi$ be a smooth function with compact support on $\R_+$ and define for $V\in \R^{dN}$
$$
\phi(V) := \chi(|P_h V|) \, (F\circ P_\SS) (V) \, (\Phi \circ P_\SS) (V).
$$
We can compute $\Div \phi(V)$ and after some simplifications using the formul\ae{} for the projections \eqref{P_h} and \eqref{P_S}, the gradient \eqref{Grad_SB} and the divergence \eqref{div_S} on $\SS^N_\BB$ we get
\begin{equation}\label{div_phi}
\begin{aligned}
\Div\phi(V) &= \frac{\chi'(|P_h V|)}{\sqrt{dN}} \, F(P_\SS V) \, P_\SS V \cdot \Phi(P_\SS V) 
 \\
&+ \chi(|P_h V|)  \, \nabla_{\SS} F(P_\SS V)  \cdot \Phi(P_\SS V) \,
\frac{\sqrt{dN}}{|P_h V|} \\
&+ \chi(|P_h V|)  \, F(P_\SS V) \, \Div_{\SS} \Phi(P_\SS V) \,
\frac{\sqrt{dN}}{|P_h V|}.
\end{aligned}
\end{equation}
Integrating \eqref{div_phi} we get
\begin{equation*}
\begin{aligned}
&\int_{\R^{dN}}   F(P_\SS V) \, P_\SS V \cdot \Phi(P_\SS V) \,\frac{\chi'(|P_h V|)}{\sqrt{dN}} \, dV \\
&\quad+ \int_{\R^{dN}} \Big[  \nabla_{\SS} F(P_\SS V)  \cdot \Phi(P_\SS V) +
    F(P_\SS V) \, \Div_{\SS} \Phi(P_\SS V) \Big] \, \chi(|P_h V|) \,
\frac{\sqrt{dN}}{|P_h V|} \, dV = 0.
\end{aligned}
\end{equation*}
Using the change of coordinates $V=(v_1,\dots,v_N) \to U=(u_1,\dots,u_N)$ given by Lemma~\ref{lem:change} and then the variables $w = \sum_{i=1}^N |u_i|^2$ and $z=\sqrt{N}u_N$, we obtain that the last expression is equal to
\begin{equation*}
\begin{aligned}
&\int_0^\infty\!\!\!\int_{\R^d}\left\{ 
\frac{|\Sp^{d(N-1)-1}|}{2\, N^{d/2}} \left(w-\frac{|z|^2}{N} \right)^{\frac{d(N-1)-2}{2}}
\int_{\SS^{N}(w,z)}   F(V) \,  V \cdot \Phi(V) \, d\gamma^N_{w,z} \right\}
\frac{\chi'\left(\sqrt{w-\frac{|z|^2}{N}}\right)}{\sqrt{dN}}  \, dz\, dw \\
&+ \int_0^\infty\!\!\!\int_{\R^d} \Bigg\{ 
\frac{|\Sp^{d(N-1)-1}|}{2\, N^{d/2}} \left(w-\frac{|z|^2}{N} \right)^{\frac{d(N-1)-2}{2}}
\\ &\qquad
\int_{\SS^{N}(w,z)}
\Big[  \nabla_{\SS} F(P_\SS V)  \cdot \Phi(P_\SS V) +
    F(P_\SS V) \, \Div_{\SS} \Phi(P_\SS V) \Big]\, d\gamma^N_{w,z} \Bigg\} 
    \chi\left(\sqrt{w-\frac{|z|^2}{N}}\right) \,
\frac{\sqrt{dN}}{\sqrt{w-\frac{|z|^2}{N}}} \, dz \, dw,
\end{aligned}
\end{equation*}
and then we get
\begin{equation*}
\begin{aligned}
&\int_0^\infty\!\!\!\int_{\R^d} 
 \left(w-\frac{|z|^2}{N} \right)^{\frac{d(N-1)-2}{2}} 
\frac{\chi'\left(\sqrt{w-\frac{|z|^2}{N}}\right)}{dN}  \, dz\, dw 
\left(\int_{\SS^{N}_\BB}   F(V) \,  V \cdot \Phi(V) \, d\gamma^N\right)\\
&+ \int_0^\infty\!\!\!\int_{\R^d} 
 \left(w-\frac{|z|^2}{N} \right)^{\frac{d(N-1)-3}{2}}
    \chi\left(\sqrt{w-\frac{|z|^2}{N}}\right)  dz \, dw
\left(\int_{\SS^{N}_\BB}
\Big[  \nabla_{\SS} F(V)  \cdot \Phi(V) +
F( V) \, \Div_{\SS} \Phi(V) \Big]\, d\gamma^N \right)
\\
&= 0.
\end{aligned}
\end{equation*}
Since we have
$$
\begin{aligned}
&\int_0^\infty\!\!\!\int_{\R^d} \left(w-\frac{|z|^2}{N} \right)^{\frac{d(N-1)-2}{2}} 
\chi'\left(\sqrt{w-\frac{|z|^2}{N}}\right)  dz\,dw =
\\
&\qquad -[d(N-1)-1] \int_0^\infty\!\!\!\int_{\R^d} \left(w-\frac{|z|^2}{N} \right)^{\frac{d(N-1)-3}{2}}
    \chi\left(\sqrt{w-\frac{|z|^2}{N}}\right)  dz\,dw,
\end{aligned}
$$
we obtain the result
$$
\int_{\SS^N_\BB} \left\{ \nabla_\SS F(V) \cdot \Phi(V) + F(V) \Div_\SS \Phi(V) - \frac{d(N-1)-1}{dN} \, F(V) \Phi(V) \cdot V \right\} d\gamma^N(V) = 0.
$$
\end{proof}

With these results at hand we are able to state the following theorem, which is the Fisher's information version of Theorem~\ref{thm:ent-semicont-general} and the proof is an adaptation of \cite[Theorem 4.15]{HaurayMischler}.

\begin{thm}
\label{thm:fisher-semicont}
Consider $G^N$ a probability measure on $\SS^N_\BB$ such that for some positive integer $\ell$, we have $G^N_\ell \wto \pi_\ell$ 
in $\PPP(\R^{d\ell})$ when $N$ goes to infinity.

Then, we have
\begin{equation*}\label{}
\begin{aligned}
\frac{1}{\ell} \, I(\pi_\ell|\gamma^{\otimes \ell}) \leq \liminf_{N\to\infty} \frac{1}{N}\, I(G^N|\gamma^N) .
\end{aligned}
\end{equation*}

\end{thm}

\begin{proof}
Let us denote $G^N =: g^N \gamma^N$.
Using \cite{HaurayMischler} we have the following representation formula
\begin{equation*}
\begin{aligned}
I(G^N|\gamma^N) &= \int_{\SS^N_\BB} |\nabla_\SS \log g^N|^2 g^N \, d\gamma^N \\
&= \sup_{\Phi\in C^1_b(\R^{dN};\R^{dN})} \int_{\SS^N_\BB}\left( \nabla_\SS\log g^N\cdot \Phi - \frac{|\Phi|^2}{4}\right) g^N \, d\gamma^N 
\end{aligned}
\end{equation*}
and we obtain by Lemma~\ref{lem:ipp}
\begin{equation}\label{Fisher-rep}
\begin{aligned}
I(G^N|\gamma^N) 
&= \sup_{\Phi\in C^1_b(\R^{dN};\R^{dN})} \int_{\SS^N_\BB}\left(  \frac{d(N-1)-1}{dN} \, \Phi(V) \cdot V -  \Div_\SS \Phi(V) - \frac{|\Phi(V)|^2}{4} \right) g^N \, d\gamma^N.
\end{aligned}
\end{equation}
Furthermore for $\pi_\ell$ we have, also from \cite{HaurayMischler},
\begin{equation*}\label{}
\begin{aligned}
I(\pi_\ell|\gamma^{\otimes \ell}) 
= \sup_{\varphi\in C^1_b(\R^{d\ell};\R^{d\ell})} \int_{\R^{d\ell}}\left( \varphi\cdot V_\ell - \Div\varphi -\frac{|\varphi|^2}{4}\right) \pi_\ell. 
\end{aligned}
\end{equation*}
Let us fix $\eps>0$ and choose $\varphi$ such that
$$
\frac1\ell \, I(\pi_\ell|\gamma^{\otimes \ell}) -\eps
\le \frac1\ell \, \int_{\R^{d\ell}}\left( \varphi\cdot V_\ell - \Div\varphi -\frac{|\varphi|^2}{4}\right) \pi_\ell
$$
Denote $N = q\ell + r$, $0\le r < \ell$, and define $V_N = (V_{\ell,1}, \dots ,V_{\ell,q} ,V_r)$. Choosing $\Phi(V_N) := ( \varphi(V_{\ell,1}) , \dots , \varphi(V_{\ell,q}) ,0) \in C^1_b(\R^{dN};\R^{dN})$ we obtain from \eqref{Fisher-rep} and the symmetry of $G^N$
\begin{equation*}
\begin{aligned}
\frac1N \, I(G^N|\gamma^N) &\ge \frac1N \, \int_{\SS^N_\BB}\left(  \frac{d(N-1)-1}{dN} \, \Phi(V_N) \cdot V_N -  \Div_\SS \Phi(V_N) - \frac{|\Phi(V_N)|^2}{4} \right) G^N(dV_N) \\
&\ge \frac qN \, \int_{\R^{d\ell}}\left(  \frac{d(N-1)-1}{dN} \, \varphi(V_\ell) \cdot V_\ell -  \Div \varphi(V_\ell) - \frac{|\varphi(V_\ell)|^2}{4} \right) G^N_\ell(dV_\ell) + \frac{R(N)}{N},
\end{aligned}
\end{equation*}
with
$$
R(N) = \int_{\R^{d\ell}} \sum_{k=1}^\ell \sum_{i=1}^\ell \sum_{\beta=1}^d 
\left(  \frac1N \, \partial_{v_{i,\beta}} \varphi_{k,\beta} 
+ \frac{1}{dN} (\partial_{v_{i,\beta}} \varphi_{k,\beta}) v_{i,\beta} v_{k,\beta} \right)
G^N_\ell(dV_\ell).
$$
The last expression is bounded if $\nabla \varphi$ decreases rapidly enough at infinity. Hence, passing to the limit we obtain
$$
\begin{aligned}
\liminf_{N\to\infty} \frac1N \, I(G^N|\gamma^N) &\ge \frac1\ell \, \int_{\R^{d\ell}} 
\left( \varphi \cdot V_\ell - \Div \varphi - \frac{|\varphi|^2}{4}   \right) \pi_\ell \\
&\ge \frac1\ell I(\pi_\ell | \gamma^{\otimes \ell}) - \eps,
\end{aligned}
$$
and we conclude letting $\eps\to 0$.
\end{proof}

We can prove now precise versions of implications $(i)\Rightarrow (ii)$ and $(iii)\Rightarrow (iv)$ of Theorem~\ref{thm:intro-GN} as follows.

\begin{thm}\label{thm:EntropyToKac}
Consider $G^N \in \PPP(\SS^N_\BB)$ such that 
$G^N_1 \wto f$ in $\PPP(\R^d)$. We have the following properties:

\begin{enumerate}[(i)]

\item If $H(f|\gamma)<\infty$ and 
$
\displaystyle \lim_{N\to\infty} \frac{1}{N} \,  H(G^N | \gamma^N) = H(f | \gamma),
$
then $G^N$ is $f$-Kac's chaotic.

\item If $I(f|\gamma)<\infty$ and 
$
\displaystyle \lim_{N\to\infty} \frac{1}{N} \,  I(G^N | \gamma^N) = I(f | \gamma),
$
then $G^N$ is $f$-Kac's chaotic.
\end{enumerate}

\end{thm}

\begin{proof}
Let us fix $\ell \in \N^\ast$. Since $G^N_1 \wto f$ in $\PPP(\R^d)$ we know by \cite[Proposition 2.2]{S6} that $G^N$ is tight. 
Then there exists a subsequence $G^{N'}$ and $\pi_\ell \in \PPP(\R^{d\ell})$ such that $G^{N'}_\ell \wto \pi_\ell$ in $\PPP(\R^{d\ell})$, when $N'$ goes to infinity (and in particular $\pi_1 = f$).

\subsubsection*{(i)} 
By Theorem~\ref{thm:ent-semicont-general} we have
$$
 \frac{1}{\ell} \, H(\pi_\ell | \gamma^{\otimes \ell}) \leq \liminf_{N'\to\infty}  \frac{1}{N'} \,  H(G^{N'} | \gamma^{N'}) = H(f | \gamma).
$$
Since we also have the reverse inequality by superadditivity of the entropy functional, we obtain
$$
\begin{aligned}
H(\pi_\ell | \gamma^{\otimes \ell}) -\ell H(f|\gamma) 
&= \int  \pi_\ell \log \frac{\pi_\ell}{\gamma^{\otimes \ell}}  - \ell \int f \log \frac{f}{\gamma}\\
&= \int  \pi_\ell \log \frac{\pi_\ell}{\gamma^{\otimes \ell}}  -  \int \pi_\ell \log \frac{f^{\otimes \ell}}{\gamma^{\otimes \ell}}\\
&= \int f^{\otimes \ell}\left( \frac{\pi_\ell}{f^{\otimes \ell}}\log\frac{\pi_\ell}{f^{\otimes \ell}}
 - \frac{\pi_\ell}{f^{\otimes \ell}} + 1 \right) \\
&= 0,
\end{aligned}
$$
which implies $\pi_\ell = f^{\otimes \ell}$ a.e. on $\{ f^{\otimes \ell} > 0\}$, since the function $z\mapsto z\log z - z + 1$ is equal to $0$ in $z=1$. Thanks to $\pi_\ell,f^{\otimes \ell} \in \PPP(\R^{d\ell})$, we obtain
$$
\int_{\{ f^{\otimes \ell} > 0\}} \pi_\ell = \int_{\{ f^{\otimes \ell} > 0\}} f^{\otimes \ell} = 1.
$$
It follows that $\pi_\ell = f^{\otimes \ell}$ a.e on $\R^{d\ell}$, so the whole sequence $G^N_\ell$ converges to $f^{\otimes \ell}$ and thus $G^N$ is $f$-chaotic.

\subsubsection*{(ii)}
The proof of point (ii) being similar, thanks to Theorem~\ref{thm:fisher-semicont} and the superadditivity of the Fisher's information \cite{CarlenFisher}, we skip it.

\end{proof}

Recall another notion of entropic chaos stated in \eqref{eq:cond1}, as proposed in \cite[Theorem 9 and Open Problem 11]{CCLLV} and \cite[Remark 7.11]{MMchaos}, for $G^N\in\PPP(\SS^N_\BB)$ and $f\in\PPP_6 \cap L^p (\R^d)$ with $p>1$, we consider the following property
\begin{equation}\label{strong-ent-chaos}
\lim_{N\to\infty} \frac{1}{N} \, H\left( G^N | [ f^{\otimes N} ]_{\SS^N_\BB} \right) = 0.
\end{equation}
Let us now investigate the relation between condition \eqref{strong-ent-chaos} and the entropic chaos (Definition~\ref{def:ent-chaos}) in the following result, which shows that, under some assumptions on $f$, they are equivalent.

\begin{thm}\label{thm:comp-ent}
Let $f\in \PPP_6(\R^d) \cap L^\infty(\R^d)$ and $G^N\in\PPP(\SS^N_\BB)$ such that $G^N_1\wto f$. Suppose further that  $f(v_1) \ge \exp(-\alpha|v_1|^2 + \beta)$ for some $\alpha>0$ and $\beta\in\R$.
Then the following asserstions are equivalent:
\begin{enumerate}[(i)]

\item
$
\displaystyle \lim_{N\to\infty} \frac1N \, H\left(G^N | [f^{\otimes N}]_{\SS^N_\BB} \right) = 0
$;
\smallskip
\item 
$
\displaystyle \lim_{N\to\infty} \frac1N \, H(G^N | \gamma^N ) = H(f|\gamma).
$
\end{enumerate}
\end{thm}

\begin{rem}
We remark that both conditions $(i)$ and $(ii)$ imply that $G^N$ is $f$-chaotic. Indeed, in \cite[Theorem 19]{CCLLV} is proved that $(i)$ implies the $f$-chaoticity of $G^N$ in the Kac's sphere framework, the generalization to the Boltzmann's sphere case is straightforward. Finally, 
the fact that condition $(ii)$ implies that $G^N$ is $f$-chaotic follows from Theorem~\ref{thm:EntropyToKac}.
\end{rem}

\begin{proof}

Denote $G^N =: g^N \gamma^N$ and $F^N = [f^{\otimes N}]_{\SS^N_\BB} =: f^N \gamma^N$. Then we write
\begin{equation}\label{eq:HGN}
\begin{aligned}
H(G^N | \gamma^N ) &= \int_{\SS^N_\BB} \left(\log \frac{g^N}{f^N}\right) g^N \, d\gamma^N + \int_{\SS^N_\BB} \left(\log f^N\right) g^N \, d\gamma^N \\
&= H(G^N | [f^{\otimes N}]_{\SS^N_\BB} )  + \int \log f^{\otimes N} \, dG^N 
- \int \log \gamma^{\otimes N} \, dG^N  - \log Z'_N(f;\sqrt{dN},0)\\
&= H(G^N | [f^{\otimes N}]_{\SS^N_\BB} )  + N\int_{\R^d} \log f \, dG^N_1 
+\frac{dN}{2}(\log 2\pi + 1)  - \log Z'_N(f;\sqrt{dN},0)
\end{aligned}
\end{equation}
using the symmetry of $G^N$, the explicit formula for $\gamma^{\otimes N}$ and the fact that $M_2(G^N) = dN$.
Since $M_2(f)=d$, we obtain
\begin{equation*}
\begin{aligned}
\frac1N\,H(G^N | \gamma^N ) - H(f | \gamma)
&= \frac1N\, H(G^N | [f^{\otimes N}]_{\SS^N_\BB} )  + \int_{\R^d}(G^N_1-f) \log f -\frac1N\,\log Z'_N(f;\sqrt{dN},0).
\end{aligned}
\end{equation*}
The third term of the right-hand side goes to $0$ as $N\to\infty$ thanks to Theorem~\ref{thm:ZN}. Hence we only need to prove that the second term of the right-hand side vanishes as $N\to\infty$, which implies that $(i)$ is equivalent to $(ii)$. 

With the assumptions on $f$ we obtain $|\log f| \le \log \norm{f}_{L^\infty} + \alpha|v|^2 + \beta \le C_1(1+|v|^2)$. Consider $R>1$ and we have 
$$
\int_{|v|>R} (1+|v|^2)f < \frac{1}{R^4}\int_{|v|>R} |v|^4 f + \frac{1}{R^4}\int_{|v|>R} |v|^6 f \le C_2 R^{-4} .
$$ 
Let $\chi_R$ be a smooth function such that $0\le \chi_R \le 1$, $\chi_R(v)=1$ for $|v|\le R $ and $\chi_R(v)=0$ for $|v|\ge R+1$. We can split the integral to be estimated in the following way
\begin{equation}\label{chi}
\int_{\R^d}(G^N_1-f) \log f = \int_{\R^d}\chi_R (G^N_1-f) \log f 
+ \int_{\R^d}(1-\chi_R)(G^N_1-f) \log f .
\end{equation}

Let us show first that $H(G^N_1)=\int G^N_1\log G^N_1$ is bounded. If we assume condition $(ii)$ then $N^{-1} H(G^N| \gamma^N)$ is bounded. On the other hand, if we assume $(i)$, from \eqref{eq:HGN} we have
$$
\frac1N\,H(G^N | \gamma^N ) 
\leq \frac1N\, H(G^N | [f^{\otimes N}]_{\SS^N_\BB} )  + \log\norm{f}_{L^\infty} 
+\frac{dN}{2}(\log 2\pi + 1)  - \frac1N\,\log Z'_N(f;\sqrt{dN},0),
$$
and again $N^{-1} H(G^N| \gamma^N)$ is bounded. Moreover, we obtain thanks to \cite{BartheCEM2006} that 
$$
{H(G^N_1 | \gamma^N_1)} \le C \, \frac{H(G^N | \gamma^N) }{N}
$$
for some $C>0$ and can write 
$$
H(G^N_1 | \gamma)=H(G^N_1 | \gamma^N_1) + \int \log \frac{\gamma^N_1}{\gamma} \, G^N_1,
$$
which is bounded thanks to the explicit computation of $\gamma^N_1$ in Lemma~\ref{lem:gammaNl} and to the Lemma~\ref{lem:gammaN-chaos}. We deduce, since $H(G^N_1|\gamma) = H(G^N) + d(\log 2\pi + 1)/2$, that $H(G^N_1)$ is bounded either if we assume $(i)$ or $(ii)$.

Then, for the first term of \eqref{chi}, since $\chi_R \log f$ is a bounded function, $G^N_1$ converges weakly to $f$ in $\PPP(\R^d)$ and $H(G^N_1)$ is bounded, we obtain that $\int \chi_R (G^N_1-f) \log f\to 0$ as $N\to\infty$. For the second term of \eqref{chi} we write (recall that $\int (1+|v|^2)G^N_1 = 1+d = \int (1+|v|^2)f$ )
$$
\begin{aligned}
\Abs{\int_{\R^d}(1-\chi_R)(G^N_1-f) \log f} &\le C_1\int_{\R^d}(1-\chi_R)(1+|v|^2)(G^N_1+f) \\
&\le C_1 C_2 R^{-4} + C_1(1+d) - C_1\int_{\R^d}\chi_R(1+|v|^2)G^N_1 .
\end{aligned}
$$
The function $\chi_R(1+|v|^2)$ being bounded and continuous, we know that $\int \chi_R(1+|v|^2) (G^N_1-f)\to 0$ as $N\to\infty$. Thus passing to the limit in the last expression we obtain
$$
\begin{aligned}
\limsup_{N\to\infty}\Abs{\int_{\R^d}(1-\chi_R)(G^N_1-f) \log f} &\le C_1 C_2 R^{-4} + C_1(1+d) - C_1\int_{\R^d}(\chi_R)(1+|v|^2)f \\
&\le 2C_1 C_2 R^{-4}
\end{aligned}
$$
which concludes the proof letting $R\to\infty$.

\end{proof}

\begin{rem}
In the setting of the Kac's sphere (usual sphere $\Sp^{N-1}(\sqrt N)$), we find in \cite[Theorem 21]{CCLLV} a proof of $(i)$ implies $(ii)$ without the assumption $f(v_1) \ge \exp(-\alpha|v_1|^2 + \beta)$. We can adapt it to our case in the following way. 

\begin{proof}[Proof of $(i)\Rightarrow (ii)$]
We write from \eqref{eq:HGN} and for $\delta>0$
\begin{equation*}
\begin{aligned}
\frac1N\,H(G^N | \gamma^N ) &\le \frac1N\,H(G^N | [f^{\otimes N}]_{\SS^N_\BB} )  + \int \log (f+\delta) \, G^N_1 
+\frac{d}{2}(\log 2\pi + 1)  - \frac1N\,\log Z'_N(f;\sqrt{dN},0).
\end{aligned}
\end{equation*}
Since $\log(f+\delta)$ is a bounded function thanks to $f\in L^\infty$, $H(G^N_1)$ is bounded and $G^N_1\wto f$ in $\PPP(\R^d)$ we have $\int \log (f+\delta) \, G^N_1 \to \int \log (f+\delta) \, f$ as $N\to\infty$ . We can pass to the limit $N\to \infty$ to obtain
\begin{equation*}
\begin{aligned}
\limsup_{N\to\infty}\frac1N\,H(G^N | \gamma^N ) &\le  \int \log (f+\delta) \, f 
+\frac{d}{2}(\log 2\pi + 1).
\end{aligned}
\end{equation*}
Now letting $\delta\to 0$, by dominated convergence we obtain
\begin{equation*}
\begin{aligned}
\limsup_{N\to\infty}\frac1N\,H(G^N | \gamma^N ) &\le  \int f\log f  
+\frac{d}{2}(\log 2\pi + 1) = H(f|\gamma),
\end{aligned}
\end{equation*}
and we conclude with this estimate togheter with
$$
H(f|\gamma) \le \liminf_{N\to\infty}\frac1N\,H(G^N | \gamma^N )
$$
from Theorem~\ref{thm:ent-semicont-general}.
\end{proof}
\end{rem}

%
%
%

\subsection{On a more general class of chaotic probabilities}
In the subsection~\ref{ssec:ent-FN} we have constructed a particular probability measure on $\SS^N_\BB$ that is entropically chaotic. Hence, a natural question is whether it is true for a more general class of probabilities on the Boltzmann's sphere. Theorem~\ref{thm:ent-chaos-general}, which is a precise version of $(ii)\Rightarrow (iii)$ in Theorem~\ref{thm:intro-GN}, gives an answer with a quantitative rate.

First of all, let us present some results concerning different forms of measuring chaos that will be useful in te sequel.

\begin{lemma}\label{lem:W2toW1}
Consider $f,g \in \PPP(\R^d)$ and $F^N,G^N \in \PPP(\R^{dN})$. Let us define $M_k(F,G):= M_k(F) + M_k(G)$.

For any $k\ge 2$ we have
\begin{equation}\label{W2W1}
W_2(f,g) \le 2^{\frac32} \, M_k(f,g)^{\frac{1}{2(k-1)}} \, W_1(f,g)^{\frac{k-2}{2(k-1)}}
\end{equation}
and
\begin{equation}\label{W2W1N}
\frac{W_2(F^N,G^N)}{\sqrt N} \le 2^{\frac32} \, \left(\frac{M_k(F^N,G^N)}{N}\right)^{\frac{1}{2(k-1)}} \, \left( \frac{W_1(F^N,G^N)}{N}\right)^{\frac{k-2}{2(k-1)}}.
\end{equation}
\end{lemma}

The proof of Lemma~\ref{lem:W2toW1} come from \cite[Lemma 4.1]{MMchaos} for \eqref{W2W1} and \eqref{W2W1N} is a simple generalization of \eqref{W2W1} to the case of $N$ variables.

\medskip

We denote by $\overline W_1$ the MKW distance \eqref{eq:MKW} defined with a bounded distance in $\R^d$, more precisely, for all $f,g \in \PPP_1(\R^d)$,
$$
\overline W_1(f,g) = \inf_{\pi\in\Pi(f,g)} \int_{\R^d\times\R^d} \min\{|x-y|,1\}\,\pi(dx,dy).
$$
Consider $G^N \in \PPP(\R^{dN})$ and $f\in\PPP(\R^d)$. We define then $\widehat G^N$, $\delta_{f} \in \PPP(\PPP(\R^d))$ by, for all $\Phi\in C_b(\PPP(\R^d))$,
\begin{equation}\label{deltaf}
\begin{aligned}
&\int_{\PPP(\R^d)} \Phi(\rho) \,\widehat G^N(d\rho) = \int_{\R^{dN}}\Phi(\mu^N_V) \,G^N(dV),\qquad \mu^N_V = \frac1N \sum_{i=1}^N \delta_{v_i} \in \PPP(\R^d), \\
&\int_{\PPP(\R^d)} \Phi(\rho) \, \delta_{f}(d\rho) = \Phi(f).
\end{aligned}
\end{equation}
Furthermore, $\WW$ stands for the Wasserstein distance on $\PPP(\PPP(\R^d))$. More precisely, for some distance $D$ on $\PPP(\R^d)$ we define
$$
\forall\mu,\nu \in \PPP(\PPP(\R^d)),\quad \WW_{D}(\mu,\nu )  := \inf_{\pi\in\Pi(\mu,\nu)}\int_{\PPP(\R^d)\times \PPP(\R^d)} D(f,g)\, d\pi(f,g).
$$
In the particular case of $\widehat G^N$ and $\delta_f$ we have $\Pi(\widehat G^N,\delta_f)=\{\widehat G^N \otimes \delta_f\}$ and then
\begin{equation}\label{eq:WW}
\WW_{D}\Big(\widehat G^N,\delta_{f} \Big)  = \int_{\R^{dN}} D(\mu^N_V,f) \, G^N(dV).
\end{equation}

We have the following result from \cite{HaurayMischler}.
\begin{lemma}\label{lem:W2tobW1}
Consider $f,g \in \PPP(\R^d)$ and $F^N,G^N \in \PPP(\SS^N_\BB)$. Let us define $M_k(F,G):= M_k(F) + M_k(G)$.

\begin{enumerate}[(i)]
\item For any $k> 2$ we have
\begin{equation}\label{W2bW1}
W_2(f,g) \le 2^{\frac32} \, M_k(f,g)^{\frac{1}{k}} \,\overline W_1(f,g)^{\frac12-\frac1k}
\end{equation}
and
\begin{equation}\label{W2bW1N}
\frac{W_2(F^N,G^N)}{\sqrt N} \le 2^{\frac32} \, \left(\frac{M_k(F^N,G^N)}{N}\right)^{\frac{1}{k}} \, \left( \frac{\overline W_1(F^N,G^N)}{N}\right)^{\frac12-\frac1k}.
\end{equation}

\item For any $0< \alpha_1 < 1/(d+1) $ and $k> d(\alpha_1^{-1}-d-1)^{-1}$ there exists a constant $C:=C(d,\alpha_1,k)>0$ such that
\begin{equation}\label{WW1_W1}
\WW_{\overline W_1} (\widehat G^N,\delta_f) \le C M_k(G^N_1,f)^{1/k} \left( \overline W_1(G^N_2,f^{\otimes 2}) + \frac1N \right)^{\alpha_1}.
\end{equation}

\item For any $0< \alpha_2 < 1/d' $ and $k> d'(\alpha_2^{-1}-d')^{-1}$, with $d':=\max(d,2)$, there exists a constant $C:=C(d,\alpha_2,k)>0$ such that
\begin{equation}\label{bW1-WW1}
\left| \overline W_1(G^N,f^{\otimes N}) - \WW_{\overline W_1} (\widehat G^N,\delta_f)   \right| \le C \,\frac{M_k(f)^{1/k}}{N^{\alpha_2}}.
\end{equation}
\end{enumerate}

\end{lemma}

The equations \eqref{W2bW1} and \eqref{W2bW1N} come from \cite[Lemmas 2.1 and 2.2]{HaurayMischler}, and \eqref{WW1_W1}-\eqref{bW1-WW1} are proved in \cite[Theorem 1.2]{HaurayMischler}.

\smallskip
As a consequence of Lemma~\ref{lem:W2tobW1} we have the following result.
\begin{lemma}\label{lem:W2eqW1}
Consider $G^N \in \PPP(\SS^N_\BB)$ and $f\in \PPP(\R^d)$ such that $M_k(G^N_1)$ and $M_k(f)$ are finite, for $k>2$. Let us denote $\MM_k := M_k(G^N_1) + M_k(f)$. 

Then for any $0< \alpha_1 < 1/(d+1)$ and $\alpha_1 < k(dk+d+k)^{-1}$, $0 < \alpha_2 < 1/d'$ and $\alpha_2 < k(d'k+d')^{-1}$, with $d':=\max(d,2)$, there exists a constant $C:=C(d,k,\alpha_1,\alpha_2)$ such that
$$
\begin{aligned}
\frac{W_2(G^N,f^{\otimes N})}{\sqrt{N}} 
&\le C  \MM_k^{\frac1k}
\left(  \overline W_1(G^N_2, f^{\otimes 2})^{\alpha_1} + N^{-\alpha_1} +   N^{-\alpha_2}\right)^{\frac12-\frac1k}
\end{aligned}
$$
\end{lemma}

\begin{proof}
First of all, we remark that $N^{-1}M_k(G^N)$ is equivalent to $M_k(G^N_1)$ since $G^N$ is symmetric. Then, using Lemma~\ref{lem:W2tobW1} we have
$$
\begin{aligned}
\frac{W_2(G^N,f^{\otimes N})}{\sqrt{N}} 
&\le 2^{\frac23}  \MM_k^{\frac1k}
\left( \frac{\overline W_1(G^N,f^{\otimes N})}{N}\right)^{\frac12-\frac1k} \\
&\le 2^{\frac23}  \MM_k^{\frac1k}
\left( C\frac{M_k(f)^{\frac1k}}{N^{\alpha_2}} + 
\WW_{\overline W_1}\left( \widehat G^N,\delta_f \right) \right)^{\frac12-\frac1k} \\
&\le 2^{\frac23}  C \MM_k^{\frac1k}
\left( N^{-\alpha_2} + 
\left( \overline W_1(G^N_2, f^{\otimes 2}) + N^{-1}\right)^{\alpha_1} \right)^{\frac12-\frac1k}
\end{aligned}
$$
where we have used successively \eqref{W2bW1N}, \eqref{bW1-WW1} and \eqref{WW1_W1}, with $\alpha_1$ and $\alpha_2$ defined as above.

\end{proof}

We can now state a precise version of $(ii)\Rightarrow (iii)$ in Theorem~\ref{thm:intro-GN}.

\begin{thm}\label{thm:ent-chaos-general}
Consider $G^N \in \PPP(\SS^N_\BB)$. Moreover we suppose that $G^N$ is $f$-chaotic, for some $f\in\PPP(\R^d)$, and also that
$$
M_k(G^N_1) \le C_1, \; k \ge 6, \qquad 
\frac{1}{N} \, H(G^N|\gamma^N )\le C_2,\qquad 
\frac{1}{N} \, I(G^N|\gamma^N) \le C_3.
$$

Then 
$G^N$ is entropically $f$-chaotic. More precisely, there exists $C=C(C_1,C_2,C_3)>0$ and for 
any $\beta < (k-2)[4(dk+d+k)]^{-1}$ a constant $C':=C'(\beta)$ such that
\begin{equation*}\label{}
\begin{aligned}
\left| \frac{1}{N} \,  H(G^N|\gamma^N ) - H(f|\gamma ) \right| \leq C\left( \frac{W_2(G^N,f^{\otimes N})}{\sqrt{N}} + C' N^{-\beta}  \right).
\end{aligned}
\end{equation*}

\end{thm}

\begin{proof}
First of all, thanks to Theorem \ref{thm:ent-semicont-general} (with $g=\gamma$ and $\ell=1$) we have
$$
H(f|\gamma ) \leq \liminf_{N\to\infty} \frac{1}{N} \, H(G^N|\gamma^N ) \leq C_2
$$
and thanks to Theorem \ref{thm:fisher-semicont}
$$
I(f|\gamma ) \leq \liminf_{N\to\infty} \frac{1}{N} \, I(G^N|\gamma^N ) \leq C_3,
$$
which implies that $I(f)<\infty$. Indeed, $I(f|\gamma) = I(f) + M_2(f) - 2d$, from which we conclude.

Furthermore, since $I(f)\leq C$, $f$ lies in $L^p(\R^d)$ for some $p>1$ by Sobolev embeddings. Moreover $M_k(f)< \infty$ for some $k\ge 6$ since $M_k(G^N_1)$ is bounded and $G^N_1 \wto f$ weakly in $\PPP(\R^d)$. We have then all the conditions on $f$ to construct $F^N = [  f^{\otimes N} ]_{\SS^N_\BB}$ satisfying Theorems \ref{thm:FN-chaos} and \ref{thm:ent-chaos}.

Let us denote
$$
F^N = \frac{f^{\otimes N}}{Z_N(f;\sqrt{dN},0)}\, \gamma^N =: f^N \gamma^N
$$
and we compute the relative Fisher's information with respect to $\gamma^N$
$$
\frac{1}{N} \,  I(F^N | \gamma^N) = \frac{1}{N} \,  \int_{\SS^N_\BB} \frac{\left|  \nabla_{\SS} f^N \right|^2}{f^N} \, d\gamma^N
$$
where we recall that $\nabla_{\SS}$ is the tangent component to the sphere $\SS^N_\BB$ of the usual gradient $\nabla$ in $\R^{dN}$. Since $|\nabla_{\SS} f^N |^2 \leq |  \nabla f^N |^2$, let us compute the usual gradient of $f^N$
\begin{equation*}
\begin{aligned}
\frac{\left|\nabla f^N \right|^2}{f^N} 
&= \sum_{i=1}^{N} \frac{\left|\nabla_{\R^{d}} f^N \right|^2}{f^N} \\ 
&= \frac{1}{Z_N(f;\sqrt{dN},0)} \sum_{i=1}^{N} \frac{\left|  \nabla_{i} f_i \right|^2}{f_i}
f_1\cdots f_{i-1}\,f_{i+1}\cdots f_N
\end{aligned}
\end{equation*}
where $f_i = f(v_i)$.

We can return to the Fisher's information to obtain
$$
\begin{aligned}
\frac{1}{N} \,  I(F^N | \gamma^N) &\leq \frac{1}{N} \,  \int_{\SS^N_\BB} 
\frac{\left|\nabla f^N \right|^2}{f^N} \, d\gamma^N \\
& = \frac{1}{N} \, \int_{\SS^N_\BB} \frac{1}{Z_N(f;\sqrt{dN},0)} \sum_{i=1}^{N}\frac{\left|  \nabla_{i} f_i \right|^2}{f_i}
f_1\cdots f_{i-1}\,f_{i+1}\cdots f_N \, d\gamma^N\\
&= \int_{\R^d} \frac{\left|  \nabla_{\! v_1} f_1 \right|^2}{f_1} \, \frac{Z_{N-1}(f;\sqrt{dN-|v_1|^2},-v_1)}{Z_N(f;\sqrt{dN},0)} \, d\gamma^N_1.
\end{aligned}
$$

In the proof of Theorem \ref{thm:FN-chaos} we computed the quantity 
$$
\frac{Z'_{N-1}(f;\sqrt{dN-|v_1|^2},-v_1)}{Z'_N(f;\sqrt{dN},0)} \, \gamma^N_1(v_1) =
\theta^N_1(v_1)\gamma(v_1)
$$
with $|\theta^N_1(v_1)|\le C'$. Now, we use the fact that 
$$
\frac{Z_{N-1}(f;\sqrt{dN-|v_1|^2},-v_1)}{Z_N(f;\sqrt{dN},0)} = \frac{1}{\gamma(v_1)} \, \frac{Z'_{N-1}(f;\sqrt{dN-|v_1|^2},-v_1)}{Z'_N(f;\sqrt{dN},0)}
$$
to obtain
\begin{equation}\label{eq:IFNbound}
\begin{aligned}
\frac{1}{N} \, I(F^N | \gamma^N)
&\leq \int_{\R^d} \frac{\left|  \nabla_{\! v_1} f_1 \right|^2}{f_1} \,\theta^N_1(v_1) \, dv_1
\leq C.
\end{aligned}
\end{equation}


Since $\SS^N_\BB$ has positive Ricci curvature (because it has positive curvature), 
by \cite[Theorem 30.22]{VillaniOTO&N} and \cite{LottVillani} the following HWI inequalities hold
\begin{equation}\label{eq:HWI}
\begin{aligned}
 H (F^N | \gamma^N) -  H (G^N | \gamma^N) &\le \frac\pi 2 \, \sqrt{ I(F^N | \gamma^N)}\,  W_2 (F^N , G^N),\\
 H (G^N | \gamma^N) -  H (F^N | \gamma^N) &\le \frac\pi 2 \, \sqrt{ I(G^N | \gamma^N)} \,  W_2 (F^N , G^N).
\end{aligned}
\end{equation}

\begin{rem}
In the original HWI inequality, the $2$-MKW distance is defined with the geodesic distance on $\SS^N_\BB$, however here we use on $\SS^N_\BB$ the Euclidean distance inherited from $\R^{dN}$. Fortunately, these distance are equivalent, hence the HWI inequality holds in our case adding a factor $\pi /2$ on the right-hand side.
\end{rem}

Multiplying both sides by $1/N$ we obtain
$$
\begin{aligned}
 \frac{1}{N} \, H (F^N | \gamma^N) -  \frac{1}{N} \, H (G^N | \gamma^N) &\le \frac\pi 2 \, \sqrt{ \frac{I(F^N | \gamma^N)}{N}}\,  \frac{W_2 (F^N , G^N)}{\sqrt{N}},\\
 \frac{1}{N} \, H (G^N | \gamma^N) -  \frac{1}{N} \, H (F^N | \gamma^N) &\le  \frac\pi 2 \,\sqrt{ \frac{I(G^N | \gamma^N)}{N}}\,  \frac{W_2 (F^N , G^N)}{\sqrt{N}}.
\end{aligned}
$$
Since $N^{-1} I(F^N | \gamma^N)$ and $N^{-1} I(G^N | \gamma^N)$ are bounded, we deduce
\begin{equation}\label{eq:HFN-FGN}
\left| \frac{1}{N} \, H (F^N | \gamma^N) -  H (G^N | \gamma^N) \right| \le C \,\frac{W_2 (F^N , G^N)}{\sqrt{N}}.
\end{equation}
Finally, we write
$$
\begin{aligned}
\left|  \frac{1}{N} \, H (G^N | \gamma^N) -  H (f | \gamma) \right|  
&\leq 
\left|  \frac{1}{N} \, H (G^N | \gamma^N) -   \frac{1}{N} \, H (F^N | \gamma^N) \right|\\ 
&+ 
\left|  \frac{1}{N} \, H (F^N | \gamma^N) -  H (f | \gamma) \right|
\end{aligned}
$$
and thanks to the later estimate \eqref{eq:HFN-FGN} with the triangle inequality for the first term of the right-hand side and Theorem \ref{thm:ent-chaos} for the second one, we obtain
\begin{equation}\label{HGN-Hf}
\begin{aligned}
\left| \frac{1}{N} \,  H(G^N|\gamma^N ) - H(f|\gamma ) \right| \leq C\left( \frac{W_2(G^N,f^{\otimes N})}{\sqrt{N}} + \frac{W_2(F^N,f^{\otimes N})}{\sqrt{N}} + \frac{1}{\sqrt{N}} \right).
\end{aligned}
\end{equation}
Now we have to estimate the second term of the right-hand side. Hence, thanks to Lemma~\ref{lem:W2eqW1} we have
$$
\begin{aligned}
\frac{W_2(F^N,f^{\otimes N})}{\sqrt{N}} 
&\le C'  \MM_k^{\frac1k}
\left(  \overline W_1(F^N_2, f^{\otimes 2})^{\alpha_1} + N^{-\alpha_1} +   N^{-\alpha_2}\right)^{\frac12-\frac1k},
\end{aligned}
$$
and from Theorem~\ref{thm:FN-chaos} we have $\overline W_1(F^N_2, f^{\otimes 2}) \le  W_1(F^N_2, f^{\otimes 2}) \le C N^{-1/2}$, which yields
$$
\begin{aligned}
\frac{W_2(F^N,f^{\otimes N})}{\sqrt{N}} 
&\le C'  \MM_k^{\frac1k}
\left(  N^{-\alpha_1/2}+   N^{-\alpha_2} \right)^{\frac12-\frac1k} \\
&\le C'  N^{-\frac{\alpha_1}{2}(\frac12-\frac1k)},
\end{aligned}
$$
with $\alpha_1 < k(dk+d+k)^{-1}$.
We conclude putting this last estimate in \eqref{HGN-Hf}.
\end{proof}

We give a possible answer to \cite[Open problem 11]{CCLLV} in the Boltzmann's sphere framework, which is a precise version of Theorem~\ref{thm:intro-op}.

\begin{thm}\label{thm:op}
Consider $G^N \in \PPP(\SS^N_\BB)$ such that $G^N$ is $f$-chaotic, for some $f\in\PPP(\R^d)$, and suppose that
\begin{equation}\label{op:hyp0}
M_k(G^N_1) \le C, \; k> 2,\qquad 
\frac{1}{N} \, I(G^N|\gamma^N) \le C.
\end{equation}
Suppose further that 
\begin{equation}\label{op:hyp}
f\in L^{\infty}(\R^d) \qquad \text{and} \qquad f(v_1)\ge \exp(-a|v_1|^2)
\end{equation} 
for some constant $a>0$. 

Then for any fixed $\ell$, there exists a constant $C = C(d,\ell,\norm{f}_{L^{\infty}},M_k(G^N_1), N^{-1}I(G^N|\gamma^N)) >0$ such that for all $N\ge \ell + 1$ we have
$$
H(G^N_\ell | f^{\otimes \ell}) \le C\, W_1(G^N_\ell,f^{\otimes \ell})^{\theta(\ell,d,k)}
,
$$
where $\theta(\ell,d,k) = k[d\ell(k+3) + 2k + 4]^{-1}$.
As a consequence, $H(G^N_\ell | f^{\otimes \ell}) \to 0$ when $N\to\infty$ and condition \eqref{eq:cond2} holds.
\end{thm}


As discussed in the introduction just after Theorem~\ref{thm:intro-op}, assumptions \eqref{op:hyp0}-\eqref{op:hyp} of Theorem~\ref{thm:op} are natural in the case of Maxwellian molecules since they are propagated in time. However, the conditioned tensor product assumption can be made at initial time for the Boltzmann model but it is not propagated. 
As a consequence of this theorem, we shall obtain that condition \eqref{eq:cond2} is propagated under the master equation for Maxwellian molecules (see point $(iv)$ of Theorem~\ref{thm:intro-PropChaos} below).

\begin{proof}[Proof of Theorem~\ref{thm:op}]
We write
$$
\begin{aligned}
H(G^N_\ell | f^{\otimes \ell}) 
&= \left[H(G^N_\ell | \gamma^{\otimes \ell}) - H(f^{\otimes \ell}|\gamma^{\otimes \ell})\right] 
 + \int (G^N_\ell - f^{\otimes \ell}) \log \gamma^{\otimes \ell} \\
&\quad + \int (f^{\otimes \ell}-G^N_\ell) \log f^{\otimes \ell} \\
&=: T_1 + T_2 + T_3.
\end{aligned}
$$
Let us split the proof in several steps.

\subsubsection*{Step $1$}
For the first term we use the HWI inequality on $\R^{d\ell}$ \cite{OttoVillani},
\begin{equation*}
T_1 = H(G^N_\ell |\gamma^{\otimes \ell}) - H(f^{\otimes \ell}|\gamma^{\otimes \ell}) \le \sqrt{I(G^N_\ell | \gamma^{\otimes \ell})} \, W_2(G^N_\ell,f^{\otimes \ell}).
\end{equation*}
Let us first show that the Fisher's information $I(G^N_\ell | \gamma^{\otimes \ell})$ is bounded thanks to $N^{-1}I(G^N | \gamma^N) \le C$. Thanks to \cite[Example 2]{BartheCEM2006} (see also \cite{CarlenLL2004} for related inequalities) there exists some constant $C'>0$ such that
$$
\frac{I(G^N_\ell|\gamma^N_\ell)}{\ell} \leq C' \, \frac{I(G^N|\gamma^N)}{N}.
$$
We write then
\begin{equation}\label{Fisher-gammaNl}
\begin{aligned}
I(G^N_\ell | \gamma^N_\ell) &= \int \left| \nabla\log G^N_\ell - \nabla\log \gamma^N_\ell  \right|^2\, G^N_\ell \\
&= I(G^N_\ell) + \int\left[ 2 \Delta \log \gamma^N_\ell + |\nabla\log \gamma^N_\ell|^2 \right] G^N_\ell,
\end{aligned}
\end{equation}
and then we deduce that
\begin{equation}\label{Fisher-GNl}
I(G^N_\ell) \le I(G^N_\ell | \gamma^N_\ell) + \int\left[ 2 \Delta \log \gamma^N_\ell + |\nabla\log \gamma^N_\ell|^2 \right]_{-} G^N_\ell
\end{equation}
is bounded thanks to explicit computation of $\gamma^N_\ell$ in Lemma~\ref{lem:gammaNl}. We conclude that $I(G^N_\ell | \gamma^{\otimes \ell})$ is bounded since $M_2(G^N_\ell) = d\ell$ and writing 
\begin{equation}\label{Fisher-gammal}
\begin{aligned}
I(G^N_\ell | \gamma^{\otimes \ell}) &= I(G^N_\ell) + \int\left[ 2 \Delta \log \gamma^{\otimes \ell} + |\nabla\log \gamma^{\otimes \ell}|^2 \right] G^N_\ell \\
&= I(G^N_\ell) + M_2(G^N_\ell) - 2d\ell = I(G^N_\ell) - d\ell.
\end{aligned}
\end{equation}

Moreover, we have thanks to Lemma~\ref{lem:W2toW1} applied for $G^N_\ell,f^{\otimes \ell} \in \PPP(\R^{d\ell})$
$$
W_2(G^N_\ell,f^{\otimes \ell}) \le C \, M_{k}(G^N_\ell,f^{\otimes \ell})^{\frac{1}{2(k-1)}} \,
W_1(G^N_\ell,f^{\otimes \ell})^{\frac{k-2}{2(k-1)}},
$$
where $M_{k}(G^N_\ell,f^{\otimes \ell}) := M_{k}(G^N_\ell) + M_{k}(f^{\otimes \ell})$.
We conclude then
\begin{equation}\label{T1}
T_1 \le C \, M_{k}(G^N_\ell,f^{\otimes \ell})^{\frac{1}{2(k-1)}} \,
W_1(G^N_\ell,f^{\otimes \ell})^{\frac{k-2}{2(k-1)}}.
\end{equation}

\subsubsection*{Step $2$}
Let us denote by $B_R$ the ball centered at origin with radius $R>0$ on $\R^{d\ell}$, by $B_R^c$ its complementary and let $v=(v_1,\dots, v_\ell)\in \R^{d\ell}$. Since $\log \gamma^{\otimes \ell} = -(d/2)\log 2\pi - |v|^2/2$, we can write
\begin{equation*}
T_2 = \frac12 \int_{B_R} (f^{\otimes \ell} - G^N_\ell)|v|^2 + \frac12\int_{B_R^c} (f^{\otimes \ell} - G^N_\ell)|v|^2.
\end{equation*}
The function $\phi(v)=|v|^2$ lies in $\mathrm{Lip}(B_R)$ with $\norm{\nabla\phi}_{L^{\infty}(B_R)} = 2R$. We obtain then
\begin{equation}\label{T21}
\begin{aligned}
\int_{B_R} (f^{\otimes \ell} - G^N_\ell)|v|^2 
&\le 2R\, \sup_{\norm{\phi}_{\mathrm{Lip}(B_R)} \le 1} \left\{\int \phi(f^{\otimes \ell} - G^N_\ell) \right\} \\
&\le 2R\, \sup_{\norm{\phi}_{\mathrm{Lip}(\R^{d\ell})} \le 1} \left\{\int \phi(f^{\otimes \ell} - G^N_\ell) \right\}\\
&=2R\, W_1(G^N_\ell,f^{\otimes \ell}),
\end{aligned}
\end{equation}
where the last equality comes from the duality form for the $W_1$ distance (see for instance \cite{VillaniOTO&N}). Next we write
\begin{equation}\label{T22}
\int_{B_R^c} (f^{\otimes \ell} - G^N_\ell)|v|^2 \le \frac{1}{R^{k-2}}\int_{B_R^c} (f^{\otimes \ell} + G^N_\ell)|v|^k = \frac{M_{k}(G^N_\ell,f^{\otimes \ell})}{R^{k-2}}.
\end{equation}
Choosing $R$ such that \eqref{T21} is equal to \eqref{T22} we get
\begin{equation}\label{T2}
\begin{aligned}
T_2 &\le 2^{\frac{k-2}{k-1}} \,  {M_{k}(G^N_\ell,f^{\otimes \ell})}^{\frac{1}{k-1}} 
\, {W_1(G^N_\ell,f^{\otimes \ell})}^{\frac{k-2}{k-1}}.
\end{aligned}
\end{equation}

\subsubsection*{Step $3$}
Finally, let us investigate the third term $T_3$. We write
\begin{equation}\label{T3_0}
T_3 = \int_{B_R} (f^{\otimes \ell}-G^N_\ell) \log f^{\otimes \ell} + \int_{B_R^c} (f^{\otimes \ell}-G^N_\ell) \log f^{\otimes \ell}.
\end{equation}

For the first integral in \eqref{T3_0} we have, since $f\in L^{\infty}$ and $f^{\otimes \ell}(v)\ge e^{-a|v|^2}$,
$$
\int_{B_R} (f^{\otimes \ell}-G^N_\ell) \log f^{\otimes \ell} \leq \left( \ell \log \norm{f}_{L^\infty(B_R)} + a R^2\right) {\norm{f^{\otimes \ell}-G^N_\ell}}_{L^1(B_R)}.
$$
Let $g=f^{\otimes \ell}-G^N_\ell$ and consider a mollifier $\rho_\eps$, i.e. $\rho_\eps(v) = \eps^{-d\ell}\rho(\eps^{-1}v)$,  $\rho \in C^\infty_c(\R^{d\ell})$ with $\rho\ge 0$, $\int \rho = 1$ and $\supp \rho \subset B_1$. 
Then we have
$$
{\norm{g}}_{L^1(B_R)} \le {\norm{g \ast \rho_\eps}}_{L^1(B_R)} +  {\norm{g \ast \rho_\eps - g}}_{L^1(B_R)}.
$$
For the first term we obtain
$$
\begin{aligned}
{\norm{g \ast \rho_\eps}}_{L^1(B_R)} &= \int_{B_R}\left\{\int |\rho_\eps(w-v)|\, | f^{\otimes \ell}(v)-G^N_\ell(v)|\,dv \right\} dw \\
&\le {\norm{\nabla\rho_\eps}}_{L^\infty(B_R)} \, W_1(G^N_\ell,f^{\otimes \ell})\int_{B_R} dw \\
&\le \frac{C}{\eps^{d\ell+1}} \, R^{d\ell} \, W_1(G^N_\ell,f^{\otimes \ell}).
\end{aligned}
$$
Moreover, for the second one we have
$$
{\norm{g \ast \rho_\eps - g}}_{L^1(B_R)} \le \eps {\norm{\nabla g}}_{L^{1}} 
\le \eps \left( {\norm{\nabla f^{\otimes \ell}}}_{L^{1}} + {\norm{\nabla G^N_\ell}}_{L^{1}} \right).
$$
By Theorem~\ref{thm:fisher-semicont}, we have $I(f^{\otimes \ell} | \gamma^{\otimes \ell})\le C$ and then we deduce that $\norm{\nabla f^{\otimes \ell}}_{L^1}$ is finite.
Moreover, the boundness of $I(G^N_\ell)$ (see \eqref{Fisher-GNl})
implies that $\norm{\nabla G^N_\ell}_{L^1}$ is also finite.
We have then
$$
\begin{aligned}
{\norm{f^{\otimes \ell} - G^N_\ell}}_{L^1(B_R)} &\le \frac{C}{\eps^{d\ell+1}}\, R^{d\ell} \, W_1(G^N_\ell,f^{\otimes \ell}) +
C\eps \\
&\le C  \, R^{\frac{d\ell}{d\ell+2}}\, W_1(G^N_\ell,f^{\otimes \ell})^{\frac{1}{d\ell+2}},
\end{aligned}
$$
where we have optimized $\eps$.

For the second integral in \eqref{T3_0} we have
$$
\int_{B_R^c} (f^{\otimes \ell}-G^N_\ell) \log f^{\otimes \ell} \le \ell \log {\norm{f}}_{L^{\infty}} \, \frac{M_k(G^N_\ell,f^{\otimes \ell})}{R^k}.
$$
We conclude then, optimizing in $R$,
\begin{equation}\label{T3}
\begin{aligned}
T_3 &\le C\left( \ell \log \norm{f}_{L^\infty(B_R)} + a R^2\right) R^{\frac{d\ell}{d\ell+2}} \, W_1(G^N_\ell,f^{\otimes \ell})^{\frac{1}{d\ell+2}} + \ell \log {\norm{f}}_{L^{\infty}} \, \frac{M_k(G^N_\ell,f^{\otimes \ell})}{R^k} \\
&\le C \, W_1(G^N_\ell,f^{\otimes \ell})^{\frac{k}{d\ell(k+3) + 2k + 4}}.
\end{aligned}
\end{equation}

Finally, gathering \eqref{T1}, \eqref{T2} and \eqref{T3}, we obtain
$$
\begin{aligned}
H(G^N_\ell | f^{\otimes \ell}) 
&\le  C \left( W_1(G^N_\ell,f^{\otimes \ell})^{\frac{k-2}{2(k-1)}} 
+ W_1(G^N_\ell,f^{\otimes \ell})^{\frac{k-2}{k-1}}
+ W_1(G^N_\ell,f^{\otimes \ell})^{\frac{k}{d\ell(k+3) + 2k + 4}}\right)\\
&\le C \, W_1(G^N_\ell,f^{\otimes \ell})^{\frac{k}{d\ell(k+3) + 2k + 4}},
\end{aligned}
$$
where $C = C(d,\ell,\norm{f}_{L^{\infty}},M_k(G^N_1),N^{-1}I(G^N| \gamma^N))$.

\end{proof}

\section{Application to the Boltzmann equation}\label{sec:appli}

We can apply our results to the spatially homogeneous Boltzmann equation (equations \eqref{eq:Boltzmann} and \eqref{eq:master} in Section~\ref{sec:intro}) with true Maxwellian molecules \eqref{eq:Bmaxwell}.

We prove now Theorem~\ref{thm:intro-PropChaos}.

\begin{proof}[Proof of Theorem \ref{thm:intro-PropChaos} $(i)$] We found the proof in \cite[Theorem 7.10]{MMchaos}.
\end{proof}

\begin{proof}[Proof of Theorem \ref{thm:intro-PropChaos} $(ii)$] 
First of all, from \cite[Theorem 5.1]{MMchaos}, for all $t\geq 0$, $G^N_t$ is $f_t$-chaotic.
Now, we split the proof in several steps.

\medskip
\noindent{\it Step $1$.}
Let $G^N_0$ be built as in Theorem~\ref{thm:FN-chaos}, i.e. $G^N_0 = [ f^{\otimes N}_0]_{\SS^N_\BB}$, which is possible since $f_0 \in \PPP_6(\R^d)$ and $I(f_0 | \gamma)$ is finite.
We know from \cite[Lemma 7.4]{MMchaos} that for all $t\ge0$ the normalized Fisher's information $N^{-1}I(G^N_t | \gamma^N)$ is bounded since $N^{-1}I(G^N_t | \gamma^N)\le N^{-1} I(G^N_0 | \gamma^N)$ and the later one is bounded by 
construction (see equation \eqref{eq:IFNbound}). 
Moreover, $M_6 (\Pi_1 (G^N_0))$ is bounded by construction, thus for all $t\geq 0$, $M_6 (\Pi_1 (G^N_t))$ is also bounded thanks to \cite[Lemma 5.3]{MMchaos} .

We can then apply Theorem~\ref{thm:ent-chaos-general} to $G^N_t$ (taking $G^N = G^N_t$ and $f=f_t$ in the notation of that theorem) and we obtain that for any $\beta< (k-2)[4(dk+d+k)]^{-1}$ there exists $C'=C'(\beta)$ such that
\begin{equation}\label{eq:HGNt}
\begin{aligned}
\left|  \frac{1}{N} \,  H(G^N_t | \gamma^N) - H(f_t | \gamma)  \right| 
&\le C C' \left( \frac{W_2(G^N_t,f^{\otimes N}_t)}{\sqrt{N}} + N^{-\beta} \right).
\end{aligned}
\end{equation}

We have then to estimate the first term of the right-hand side and we shall use the result of propagation of chaos proved in \cite{MMchaos}.

\medskip
\noindent{\it Step $2$.}
Thanks to the result of propagation of chaos in \cite[Theorems 5.1 and 5.2]{MMchaos} we have, for $s> 2 + d/4$,
\begin{equation}\label{eq:prop-chaos}
\begin{aligned}
\sup_{t\ge 0} {\Norm{\Pi_2(G^N_t) - f^{\otimes 2}_t}}_{H^{-s}} 
&\le C \,\WW_{W_2}\left(\widehat G^N_0,\delta_{f_0} \right)
\end{aligned}
\end{equation}
where we recall that $\widehat G^N_0$, $\delta_{f_0} \in \PPP(\PPP(\R^d))$ are defined in \eqref{deltaf} and $\WW_{W_2}(\widehat G^N_0,\delta_{f_0} )$ in \eqref{eq:WW}, 
more precisely 
$$
\WW_{W_2}\Big(\widehat G^N_0,\delta_{f_0} \Big)  = \int_{\R^{dN}} W_2(\mu^N_V,f_0) \, G^N_0(dV).
$$

We recall that we want to estimate the first term of the right-hand side of \eqref{eq:HGNt} and we shall explain how we can obtain it from \eqref{eq:prop-chaos}. On the one hand, for the right-hand side of \eqref{eq:prop-chaos} we shall obtain a estimate of the type 
$$
\WW_{W_2}\Big(\widehat G^N_0,\delta_{f_0} \Big) \le C \,\left[ W_1(\Pi_2(G^N_0),f^{\otimes 2}_0) + N^{-\theta_2}\right]^{\theta_1}
$$ 
since we can estimate $ W_1(\Pi_2(G^N_0),f^{\otimes 2}_0)$ from Theorem~\ref{thm:FN-chaos}. On the other hand, for the left-hand side of \eqref{eq:prop-chaos}, we shall deduce an estimate like 
$$
\frac{1}{\sqrt N} \,  W_2(G^N_t,f^{\otimes N}_t) \le C\, {\Norm{\Pi_2(G^N_t) - f^{\otimes 2}_t}}_{H^{-s}}^{\theta_3}
$$
to be able to conclude.

\medskip
\noindent{\it Step $3$.}
First of all, we deduce from \eqref{W2bW1} in Lemma~\ref{lem:W2tobW1},
$$
\WW_{W_2}\left( \widehat G^N_0,\delta_{f_0} \right) \leq 2^{\frac23} \MM_k^{\frac1k} 
\WW_{\overline W_1}\left( \widehat G^N_0,\delta_{f_0} \right)^{\frac12-\frac1k}.
$$
Then, thanks to \eqref{WW1_W1} in Lemma~\ref{lem:W2tobW1} we obtain
$$
\WW_{W_2}\left( \widehat G^N_0,\delta_{f_0} \right)\leq 2^{\frac23} \MM_k^{\frac1k} 
\left( C_{\alpha_1}\MM_k^{\frac1k} \left( \overline W_1(\Pi_2(G^N_0),f^{\otimes 2}_0) + N^{-1}\right)^{\alpha_1} \right)^{\frac12-\frac1k},
$$
and using Theorem~\ref{thm:FN-chaos}, which tell us $\overline W_1(\Pi_2(G^N_0),f^{\otimes 2}_0)\le C N^{-1/2}$, we deduce
\begin{equation}\label{eq:WW2_N}
\WW_{W_2}\left( \widehat G^N_0,\delta_{f_0} \right)\leq C_{\alpha_1} 
N^{-\frac{\alpha_1}{2}\left(\frac12-\frac1k\right)},
\end{equation}
where we recall that $\alpha_1 < k(dk+d+k)^{-1}$.

\medskip
\noindent{\it Step $4$.}
Thanks to \cite[Lemma 2.1]{HaurayMischler} applied to $\Pi_2(G^N_t)$ and $f^{\otimes 2}_t \in \PPP(\R^{2d})$, for any $s> d/2$ (with $d\ge 2$) there exists $C:=C(d,s)$ such that 
$$
\overline W_1(\Pi_2(G^N_t),f^{\otimes 2}_t) \le C M_k(\Pi_2(G^N_t),f^{\otimes 2}_t)^{\frac{2d}{2d+2ks}} 
\Norm{\Pi_2(G^N_t) - f^{\otimes 2}_t}_{H^{-s}}^{\frac{2k}{2d+2ks}}.
$$
Furthermore, from Lemma~\ref{lem:W2eqW1} we obtain that there exists a constant $C:=C(d,k,\alpha_1,\alpha_2)$ such that
$$
\begin{aligned}
\frac{W_2(G^N_t,f^{\otimes N}_t)}{\sqrt{N}} 
&\le C  \MM_k^{\frac1k}
\left(  \overline W_1(\Pi_2(G^N_t), f^{\otimes 2}_t)^{\alpha_1} + N^{-\alpha_1} +   N^{-\alpha_2}\right)^{\frac12-\frac1k}.
\end{aligned}
$$
Finally, gathering these two estimates with \eqref{eq:prop-chaos} and \eqref{eq:WW2_N} we obtain that there exists $C:=C(d,s,\alpha_1,\alpha_2, M_k(f_0),M_k(\Pi_1(G^N_0)))$ such that
\begin{equation}\label{Neps}
\begin{aligned}
\frac{W_2(G^N_t,f^{\otimes N}_t)}{\sqrt{N}} 
&\le C  
\left(  N^{-\alpha_1^2\left(\frac{k}{d+ks}\right)\left(\frac12-\frac1k\right)} + N^{-\alpha_1} +   N^{-\alpha_2}\right)^{\frac12-\frac1k} \\
&\le C N^{-\epsilon},
\end{aligned}
\end{equation}
where 
$$
\begin{aligned}
\epsilon &= \alpha_1^2\left(\frac{k}{d+ks}\right)\left(\frac12-\frac1k\right)^2  \\
&< \left( \frac{k-2}{2(dk+d+k)} \right)^2 \frac{k}{d+ks} \\
&< \left( \frac{k-2}{2(dk+d+k)} \right)^2 \frac{4k}{dk+4d + 8k}
\end{aligned}
$$
using $\alpha_1 < k(dk+d+k)^{-1}$ and $s> 2 + d/4$ from \eqref{eq:prop-chaos}. We conclude taking $k=6$ and gathering \eqref{Neps} with \eqref{eq:HGNt}.

\end{proof}

\begin{proof}[Proof of Theorem \ref{thm:intro-PropChaos} $(iii)$]
The proof is a consequence of points $(i)$ and Theorem~\ref{thm:comp-ent}.
Since we have $f_0 \in \PPP_6 \cap L^\infty (\R^d)$, $f_0(v_1) \ge \exp(-\alpha|v_1|^2+\beta)$ and
$$
\lim_{N\to\infty} \frac1N\, H(G^N_0 | [f^{\otimes N}_0]_{\SS^N_\BB}) = 0,
$$
Theorem~\ref{thm:comp-ent} implies that $G^N_0$ is entropically $f_0$-chaotic. Moreover, for all $t>0$ the solution $f_t$ is bounded by below by a Maxwellian, i.e. $f_t(v_1) \ge \exp(-\bar\alpha|v_1|^2 + \bar\beta)$ for $\bar\alpha>0$ and $\bar\beta\in\R$, and also lies in $\PPP_6\cap L^\infty(\R^d)$
(see for example \cite{Villani-BoltzmannBook} and the references therein). By point $(i)$, for all $t>0$ the solution $G^N_t$ is entropically $f_t$-chaotic, then applying once more Theorem~\ref{thm:comp-ent} we deduce that
$$
\lim_{N\to\infty} \frac1N\, H(G^N_t | [f^{\otimes N}_t]_{\SS^N_\BB}) = 0.
$$

\end{proof}

\begin{proof}[Proof of Theorem \ref{thm:intro-PropChaos} $(iv)$]
The proof is a consequence of Theorem~\ref{thm:op}. From the assumptions on $f_0$ and $G^N_0$, we conclude by Theorem~\ref{thm:op} that $G^N_0$ satisfies condition \eqref{eq:cond2}
$$
\forall\; \ell \in \N, \qquad \lim_{N\to\infty} H (\Pi_\ell(G^N_0) | f_0^{\otimes \ell}) = 0.
$$

As already said in Step 1 of the proof of point $(iv)$ of Theorem~\ref{thm:intro-PropChaos}, for all $t\geq 0$, the normalized Fisher' information $N^{-1} I(G^N_t | \gamma^N)$ is bounded, as well as $M_k(\Pi_1(G^N_t))$. Furthermore, for all $t\geq0$, we have $f_t \in L^\infty (\R^d)$ and $f_t(v_1) \ge \exp(-\bar\alpha|v_1|^2 + \bar\beta)$ for some $\bar\alpha>0$ and $\bar\beta\in\R$ (see point $(iii)$ above). Hence, using once more Theorem~\ref{thm:op}, we conclude that for all $t\geq 0$, $G^N_t$ satisfies condition \eqref{eq:cond2}
$$
\forall\; \ell \in \N, \qquad \lim_{N\to\infty} H (\Pi_\ell(G^N_t) | f_t^{\otimes \ell}) = 0.
$$

\end{proof}


\appendix
\section{Auxiliary results}
We prove here some auxiliary results used in Section~\ref{sec:unif} and Section~\ref{sec:chaotic}.

\subsection{Change of variables}\label{ap:change}
We present the proof of Lemma \ref{lem:change} in Section \ref{sec:unif}.


\begin{proof}[Proof of Lemma \ref{lem:change}]
Thanks to \eqref{eq:change1} we have
\begin{equation*}
\begin{aligned}
|u_N|^2 &= \frac{1}{N}\left( \sum_{i=1}^N |v_i|^2 + 2\sum_{i=1}^{N-1} \sum_{j>i}^N v_i\cdot v_j \right)
\end{aligned}
\end{equation*}
and, for $1\leq k\leq N-1$,
\begin{equation*}
\begin{aligned}
|u_k|^2 &= \frac{1}{k(k+1)} \left( \sum_{i=1}^k |v_i|^2 + 2\sum_{i=1}^{k-1} \sum_{j>i}^{k} v_i\cdot v_j  + k^2 |v_{k+1}|^2 - 2k \sum_{i=1}^k v_i \cdot v_{k+1}\right).
\end{aligned}
\end{equation*}
We deduce from these estimates that $|u_1|^2 + \cdots + |u_N|^2 =: I_1+ I_2$ with
$$
\begin{aligned}
I_1 &= \sum_{k=1}^{N-1} \left(  \frac{1}{k(k+1)}\sum_{i=1}^k |v_k|^2 + \frac{k}{k+1}|v_{k+1}|^2  \right) + \frac{1}{N} \sum_{i=1}^N |v_i|^2\\
&=: \sum_{k=1}^{N-1} A_k + A_N
\end{aligned}
$$
and
$$
\begin{aligned}
I_2 &= 2\left[ \sum_{k=1}^{N-1} \left( \frac{1}{k(k+1)}\sum_{i=1}^{k-1}\sum_{j=i+1}^{k} v_i \cdot v_j -\frac{1}{k+1} \sum_{i=1}^{k} v_i \cdot v_{k+1} \right) 
- \frac{1}{N} \sum_{i=1}^{N-1}\sum_{j=i+1}^{N} v_i \cdot v_j   \right]\\
&=:2\left[ \sum_{k=1}^{N-1} B_k + B_N   \right].
\end{aligned}
$$

First of all, looking to $I_1$ we easily see that $|v_N|^2$ appears only in $A_{N-1}$ and $A_N$, so its coefficient is $(N-1)/N + 1/N = 1$. For $m$ such that $2\le m \le N-1$, $|v_m|^2$ appears in $A_{m-1}, A_{m},\dots, A_{N-1}$ and $A_N$, hence its coefficient is given by
$$
\frac{m-1}{m}+ \sum_{j=m}^{N-1}\frac{1}{j(j+1)} + \frac{1}{N} = 1.
$$
The coefficient of $|v_1|^2$ is the same of $|v_2|^2$ since there is no $A_{0}$. We conclude then 
$$
I_1 = |v_1|^2 + \cdots + |v_N|^2.
$$

We can compute $I_2$ in the same way. For $1\le m \le N-1$, $v_m \cdot v_N$ appears
only in $B_{N-1}$ and $B_N$, so its coefficient is $-1/N + 1/N=0$. Moreover, for $1\le m<p\le N-1$, $v_m \cdot v_p$ appears in $B_{p-1}, B_p,\dots, B_{N-1}$ and $B_N$, hence its coefficient is given by
$$
-\frac{1}{p} + \sum_{j=p}^{N-1} \frac{1}{p(p+1)} + \frac{1}{N} = 0.
$$

Finally, we conclude that $|u_1|^2 + \cdots + |u_N|^2 = |v_1|^2 + \cdots + |v_N|^2 = r^2$ and 
$u_N = z/\sqrt{N}$ follows easily from \eqref{eq:change1}.

The last point to prove is that the Jacobien is equal to one. To simplify we consider $d=1$, the general case being the same. Consider the matrix $M_N$ that represents the linear application in \eqref{eq:change1}, i.e. $M_N u=v$, where $u=(u_1,\dots, u_N) \in \R^N$ and $v=(v_1,\dots,v_N)\in \R^N$. 

We claim that $\Det(M_N)=1$. Indeed we have 
$$
M_N=
\left(
\begin{array}{ccccc}
\frac{1}{\sqrt{2}} & -\frac{1}{\sqrt{2}} & 0 & \cdots & 0 \\
\frac{1}{\sqrt{6}} & \frac{1}{\sqrt{6}} & -\frac{2}{\sqrt{6}} & \ddots & \vdots \\
\vdots &  & \ddots & \ddots & 0 \\
\frac{1}{\sqrt{(N-1)N}} & \cdots & \cdots & \frac{1}{\sqrt{(N-1)N}} & -\frac{(N-1)}{\sqrt{(N-1)N}} \\
\frac{1}{\sqrt{N}} & \cdots & \cdots & \cdots & \frac{1}{\sqrt{N}}
\end{array}\right)
$$
and it can be written in the form $M_N = D_N A_N$ with a diagonal matrix $D_N$, 
$$
M_N= \left(
\begin{array}{ccccc}
\frac{1}{\sqrt{2}} & & & &  \\
 & \frac{1}{\sqrt{6}} &  & & \\
 &  & \ddots &  &  \\
 &  &  & \frac{1}{\sqrt{(N-1)N}} & \\
 &  &  & & \frac{1}{\sqrt{N}}
\end{array}\right)
\left(
\begin{array}{ccccc}
1 & -1 & 0 & \cdots & 0 \\
1 & 1 & -2& \ddots & \vdots \\
\vdots &  & \ddots & \ddots & 0 \\
1 & \cdots & \cdots & 1 & -(N-1) \\
1 & \cdots & \cdots & \cdots & 1
\end{array}
\right).
$$

Let us prove the claim by recurrence. For $N=2$ is clear that $\Det(D_2)=1/2$ and $\Det(A_2)=2$, which implies $\det(M_2)=1$. Then, supposing that $\det(M_{N-1})=1$ we have
\begin{equation}\label{eq:det}
\det(M_{N-1})=
\left(\prod_{k=1}^{N-2} \frac{1}{\sqrt{k(k+1)}} \times \frac{1}{\sqrt{(N-1)}}\right) \Det(A_{N-1}) = 1
\end{equation}
since $\Det(D_{N-1})$ is easily computed. Moreover, we have the following relation $\Det(A_N) = N \Det(A_{N-1})$. Hence we deduce that
$$
\begin{aligned}
\Det(M_N) 
&= \left(\prod_{k=1}^{N-1} \frac{1}{\sqrt{k(k+1)}} \times \frac{1}{\sqrt{N}}\right)\Det(A_N)\\
&= \left(\prod_{k=1}^{N-2} \frac{1}{\sqrt{k(k+1)}} \times \frac{1}{\sqrt{(N-1)N}}\times\frac{1}{\sqrt{N}}\right) N \Det(A_{N-1})\\
&= 1
\end{aligned}
$$  
thanks to \eqref{eq:det}, which concludes the proof of the claim.

\end{proof}

\subsection{Regularity lemma}\label{ap:reg}

\begin{lemma}\label{lem:aux}
Let $f\in \PPP(\R^d)$. Suppose $f\in L^p\cap L_s(\R^d)$ for $p>1$ and $s>0$. Then $f\in {L}^q_m(\R^d)$ with $q<p$ and $m=s(p-q)(p-1)$.
\end{lemma}

\begin{proof}
Let us compute the $L^q_m$ norm of $f$,
\begin{equation*}
\begin{aligned}
{\norm{f}}_{L^q_m}^{q} &= \int {( 1 + {\abs{v}}^{2})}^{m/2} \, {f(v)}^q \, dv \\
&\leq C \left( \int {f(v)}^q\,dv + \int {\abs{v}}^{m} \, {f(v)}^q \, dv \right).
\end{aligned}
\end{equation*}
For the first term we have ${\norm{f}}_{L^q}^{q} \leq {\norm{f}}_{L^p}^{q}$ and for the second one we obtain
\begin{equation*}
\begin{aligned}
\int {\abs{v}}^{m} \, {f(v)}^q \, dv
&\leq {\left( \int  {\abs{v}}^{m r/(r-1)} {f(v)}^{(q-\alpha)r/(r-1)}      \right)}^{(r-1)/r}
{\left( \int {f(v)}^{\alpha r}      \right)}^{1/r}
\end{aligned}
\end{equation*}
by Holder's inequality for some $r>1$ and $0<\alpha < q$. Now choosing $r = p/\alpha $ and choosing $\alpha$ such that $(q-\alpha)r/(r-1) = 1$, i.e. $\alpha = p(q-1)/(p-1)$ we obtain
\begin{equation*}
\begin{aligned}
\int {\abs{v}}^{m} \, {f(v)}^q \, dv 
&\leq {\left( \int  {\abs{v}}^{m (p-1)/(p-q)} f(v)      \right)}^{(p-q)/(p-1)}
{\left( \int {f(v)}^{p}      \right)}^{(q-1)/(p-1)}.
\end{aligned}
\end{equation*}
Finally, choosing $m=s(p-q)/(p-1)$ we conclude with
\begin{equation*}
\begin{aligned}
{\norm{f}}_{L^q_m}^{q} 
&\leq C\left( {\norm{f}}_{L^p}^{q} + {\norm{f}}_{L_s}^{(p-q)/(p-1)}   \, {\norm{f}}_{L^p}^{p(q-1)/(p-1)}\right).
\end{aligned}
\end{equation*}

\end{proof}


\bibliographystyle{acm}
\bibliography{bib-SphereBoltzmann}

\begin{thebibliography}{10}

\bibitem{BartheCEM2006}
{\sc Barthe, F., Cordero-Erausquin, D., and Maurey, B.}
\newblock Entropy of spherical marginals and related inequalities.
\newblock {\em J. Math. Pures Appl. (9) 86}, 2 (2006), 89--99.

\bibitem{CarlenFisher}
{\sc Carlen, E.~A.}
\newblock Superadditivity of {F}isher's information and logarithmic {S}obolev
  inequalities.
\newblock {\em J. Funct. Anal. 101}, 1 (1991), 194--211.

\bibitem{CCLLV}
{\sc Carlen, E.~A., Carvalho, M.~C., Le~Roux, J., Loss, M., and Villani, C.}
\newblock Entropy and chaos in the {K}ac model.
\newblock {\em Kinet. Relat. Models 3}, 1 (2010), 85--122.

\bibitem{CCL}
{\sc Carlen, E.~A., Carvalho, M.~C., and Loss, M.}
\newblock Determination of the spectral gap for {K}ac's master equation and
  related stochastic evolution.
\newblock {\em Acta Math. 191}, 1 (2003), 1--54.

\bibitem{CarlenLL2004}
{\sc Carlen, E.~A., Lieb, E.~H., and Loss, M.}
\newblock A sharp analog of {Y}oung's inequality on {$\mathbb S^N$} and related
  entropy inequalities.
\newblock {\em J. Geom. Anal. 14}, 3 (2004), 487--520.

\bibitem{DiaconisFreedman1987}
{\sc Diaconis, P., and Freedman, D.}
\newblock A dozen de {F}inetti-style results in search of a theory.
\newblock {\em Ann. Inst. H. Poincar\'e Probab. Statist. 23}, 2, suppl. (1987),
  397--423.

\bibitem{Einav}
{\sc Einav, A.}
\newblock A counter-example to {C}ercignani's conjecture for the
  $d$-dimensional {K}ac model.
\newblock {\em J. Stat. Phys. 148}, 6 (2012), 1076--1103.

\bibitem{HaurayMischler}
{\sc Hauray, M., and Mischler, S.}
\newblock On {K}ac's chaos and related problems.
\newblock Preprint 2012, hal-00682782.

\bibitem{Kac1956}
{\sc Kac, M.}
\newblock Foundations of kinetic theory.
\newblock In {\em Proceedings of the {T}hird {B}erkeley {S}ymposium on
  {M}athematical {S}tatistics and {P}robability, 1954--1955, vol. {III}\/}
  (Berkeley and Los Angeles, 1956), University of California Press,
  pp.~171--197.

\bibitem{LottVillani}
{\sc Lott, J., and Villani, C.}
\newblock Ricci curvature for metric measure spaces via optimal transport.
\newblock {\em Ann. of Math. 169}, 3 (2009), 903--991.

\bibitem{CoursMischler}
{\sc Mischler, S.}
\newblock Introduction aux limites de champ moyen pour des syst\`emes de
  particules.
\newblock Cours en ligne C.E.L.
  http://cel.archives-ouvertes.fr/cel-00576329/fr/.

\bibitem{MischlerEDPX}
{\sc Mischler, S.}
\newblock Sur le programme de {K}ac (concernant les limites de champ moyen).
\newblock S\'eminaire EDP-X, {D}\'ecembre 2010.

\bibitem{MMchaos}
{\sc Mischler, S., and Mouhot, C.}
\newblock Kac's program in kinetic theory.
\newblock {\em Inventiones mathematicae\/} (2012), 1--147.

\bibitem{MMWchaos}
{\sc Mischler, S., Mouhot, C., and Wennberg, B.}
\newblock A new approach to quantitative chaos propagation for drift, diffusion
  and jump processes.
\newblock Preprint 2011, arXiv:1101.4727.

\bibitem{OttoVillani}
{\sc Otto, F., and Villani, C.}
\newblock Generalization of an inequality by {T}alagrand and links with the
  logarithmic {S}obolev inequality.
\newblock {\em J. Funct. Anal. 173}, 2 (2000), 361--400.

\bibitem{S6}
{\sc Sznitman, A.-S.}
\newblock Topics in propagation of chaos.
\newblock In {\em \'{E}cole d'\'{E}t\'e de {P}robabilit\'es de {S}aint-{F}lour
  {XIX}---1989}, vol.~1464 of {\em Lecture Notes in Math.} Springer, Berlin,
  1991, pp.~165--251.

\bibitem{Villani-BoltzmannBook}
{\sc Villani, C.}
\newblock A review of mathematical topics in collisional kinetic theory.
\newblock In {\em Handbook of mathematical fluid dynamics, {V}ol. {I}}.
  North-Holland, Amsterdam, 2002, pp.~71--305.

\bibitem{VillaniOTO&N}
{\sc Villani, C.}
\newblock {\em Optimal transport}, vol.~338 of {\em Grundlehren der
  Mathematischen Wissenschaften [Fundamental Principles of Mathematical
  Sciences]}.
\newblock Springer-Verlag, Berlin, 2009.
\newblock Old and new.

\end{thebibliography}

\end{document}